\documentclass[11pt]{amsart}

\numberwithin{equation}{section}

\usepackage{amssymb, amsthm, amsfonts}
\usepackage{MnSymbol}
\usepackage[all]{xy}
\usepackage{cleveref}
\usepackage{color}
\usepackage{graphicx}
\usepackage{multirow}

\newtheorem{theorem}{Theorem}[section]

\newtheorem{proposition}{Proposition}[section]
\newtheorem{corollary}{Corollary}[section]
\newtheorem{definition}{Definition}[section]
\theoremstyle{remark}
\newtheorem{example}{Example}[section]
\newtheorem{remark}{Remark}[section]

\newtheorem{problem}{Problem}

\newcommand{\teich}{\mathcal{T}}
\newcommand{\qc}[1]{\operatorname{qc}\!\left[#1\right]}
\newcommand{\normqc}[1]{\tilde{\operatorname{qc}}_0\!\left[#1\right]}
\newcommand{\teichmap}[1]{\operatorname{tm}\!\left[#1\right]}

\newcommand{\Bers}[1]{\mathcal{T}^B_{#1}}

\newcommand{\mcg}{\operatorname{MCG}}
\newcommand{\modu}{\operatorname{Mod}}
\newcommand{\slm}{\operatorname{SL}}
\newcommand{\psl}{\operatorname{PSL}}
\newcommand{\ml}{\mathcal{ML}}
\newcommand{\pml}{\mathcal{PML}}

\newcommand{\ext}{\operatorname{Ext}}

\newcommand{\ThursM}{\mu_{Th}}
\newcommand{\PThursM}{\hat{\mu}_{Th}}

\newcommand{\pushThursMBers}{{\boldsymbol \mu}^B}

\newcommand{\TRay}[1]{{\boldsymbol r}_{#1}}

\newcommand{\imaginaryunit}{{\sqrt{-1}}}

\title[Function theory, Dynamics and Ergodic theory]{Function theory, Dynamics and Ergodic theory via Thurston theory}
\author{Hideki Miyachi}
\date{}

\address{School of Mathematics and Physics,
College of Science and Engineering,
Kanazawa University,
Kakuma-machi, Kanazawa,
Ishikawa, 920-1192, Japan
}
\email{miyachi@se.kanazawa-u.ac.jp}
\thanks{This work is partially supported by JSPS KAKENHI Grant Numbers
25K00909,
23K22396 ,
20K20519}
\subjclass[2020]{32G05, 32G15, 32U35, 57M50, 57R30, 58F11, 22E40}
\keywords{Teichm\"uller space, Teichm\"uller distance, Pluriharmonic function, Pluriharmonic measure, Poisson integral formula, Ergodic theory, Mapping class group}

\begin{document}
\begin{abstract}
In this chapter, we discuss function theory on Teichm\"uller space through Thurston's theory, as well as the dynamics of subgroups of the mapping class group of a surface, with reference to Sullivan's theory on the ergodic actions of discrete subgroups of the isometry group of hyperbolic space at infinity.
\end{abstract}

\maketitle

\tableofcontents

\section{Introduction}
The aim of this chapter is to explore function theory on Teichm\"uller space through Thurston's theory, as well as the dynamics of subgroups of the mapping class group of a surface, with reference to Sullivan's theory on the ergodic actions of discrete subgroups of the isometry group of hyperbolic space at infinity.  
A corresponding dictionary for the discussions in this chapter is provided in Table~\ref{table:a_dictionary}.
\begin{table}[t]
\centering
\begin{tabular}{|c|c|}\hline
Discrete subgroups of ${\rm Isom}_+(\mathbb{H}^n)$  & Subgroups of the mapping class group\\
\hline\hline
${\rm Isom}_+(\mathbb{H}^n)$ & Mapping class group \\
\hline
Hyperbolic space $\mathbb{H}^n$ & Teichm\"uller space $\teich_{g,m}$\\
\hline
Sphere at infinity $\mathbb{S}^{n-1}$ & Thurston boundary $\mathcal{PML}_{g,m}$\\
\hline
Hyperbolic distance  & Teichm\"uller distance \\
\hline
Poisson kernel & Ratio of extremal lengths \\
\hline
Harmonic measure & Thurston's measure \\
\hline
\multicolumn{2}{|c|}{Radon--Nikodym derivative = Poisson kernel} \\
\hline
\multicolumn{2}{|c|}{Busemann function = logarithm of the Poisson kernel}\\
\multicolumn{2}{|c|}{(up to multiplication)} \\
\hline
\multicolumn{2}{|c|}{The action on the big horospherical limit set is conservative} \\
\hline
\multicolumn{2}{|c|}{Ergodicity implies absence of (pluri)harmonic functions} \\
\hline
Absence of harmonic functions & \multirow{2}{*}{?} \\
implies ergodicity & \\
\hline
\end{tabular}
\caption{A dictionary}
\label{table:a_dictionary}
\end{table}

\subsection{Function theory to Group action}
In a series of papers \cite{MR4028456}, \cite{MR4633651}, \cite{MR4830064}, and \cite{bounded-PIF-miyachi2025}, the author of the present chapter investigated function theory on the Teichm\"uller space $\teich_{g,m}$ and obtained Poisson integral formulas for bounded pluriharmonic functions on $\teich_{g,m}$.

The Poisson integral formula associates pluriharmonic functions on Teichm\"uller space with measurable functions on the ideal boundary: by taking radial limits, these pluriharmonic functions are represented as bounded measurable functions on the space of projective measured geodesic laminations, which form the ideal boundary of the Thurston compactification. The Poisson integral then recovers the pluriharmonic functions from their boundary values (cf.~\S\ref{subsec:function_theory_Tgm}).  
We will return to this correspondence later when discussing the dynamical properties of subgroups of the mapping class group $\mcg(\Sigma_{g,m})$ acting on the orientable closed surface $\Sigma_{g,m}$ of genus $g$ with $m$ punctures.

This idea goes back to a 1935 work by W.~Seidel~\cite{Seidel_metric}. In fact, Seidel characterized those Fuchsian groups $G$ acting on $\mathbb{D}$ for which the action on the ideal boundary $\partial \mathbb{D}$ is ergodic, using a function-theoretic property of the quotient Riemann surface $\mathbb{D}/G$. Underlying his characterization, the Poisson integral formula due to Fatou plays a crucial role (cf.\ \Cref{thm:seidel}).

\subsection{Group actions}
The action of discrete subgroups of the isometry group of hyperbolic space has been extensively studied. Such an action decomposes the ideal boundary into two parts: the limit set and the region of discontinuity. On the limit set, the group action is minimal and thus intricate, while the action is properly discontinuous on the region of discontinuity. Discrete subgroups are classified into two types based on the cardinality of their limit sets. A subgroup is called \emph{elementary} if its limit set is finite, and \emph{non-elementary} otherwise.

In a \emph{foundational paper}~\cite{MR982564}, McCarthy and Papadopoulos investigated the actions of subgroups of the mapping class group on the space of projective measured laminations, which forms the ideal boundary of Teichm\"uller space in the Thurston compactification. They classified these subgroups into six types based on their dynamical properties. Compared to the theory of discrete groups, four of the six types correspond to elementary groups. One of the remaining two types can potentially be reduced to cases of lower topological complexity. A subgroup of the remaining type is called \emph{sufficiently large}, and it can be regarded as analogous to a non-elementary group in the theory of discrete isometry groups.
Prior to the work of McCarthy and Papadopoulos, Masur~\cite{MR837978} studied the action of handlebody groups on the space of projective measured laminations (cf. \S\ref{subsec:McP_class}).

A sufficiently large group admits both a limit set and a region of discontinuity. McCarthy and Papadopoulos studied the structures of such invariant subsets in depth. Farb and Mosher~\cite{MR1914566} analyzed the actions of so-called convex cocompact subgroups, while Kent and Leininger~\cite{MR2465691} established structural results concerning the limit sets of convex cocompact subgroups (cf. \S\ref{subsec:convex_cocompact}).

\subsection{As measurable actions}
A group action on a measurable space decomposes the space into two measurable subsets: the conservative part and the dissipative part. On the conservative part, the action is intricate in the sense that the orbit of any subset of positive measure returns to itself infinitely often. The dissipative part contains a wandering set whose orbit never returns.

Sullivan~\cite{MR624833} identified the conservative and dissipative parts for the action of discrete subgroups of the isometry group of hyperbolic space. Specifically, he showed that the conservative part coincides with the horospherical limit set (up to null sets), while the dissipative part corresponds to the boundary of Dirichlet domains (again up to null sets).

\subsection{Purpose of this chapter}
In this chapter, we develop a measure-theoretic framework for the actions of subgroups of the mapping class group.

Motivated by the aforementioned result of Seidel, we first show that if a subgroup $H$ acts ergodically on the space of projective measured geodesic laminations, then the quotient $\teich_{g,m}/H$ admits no non-constant bounded pluriharmonic functions (cf.\ \Cref{thm:main4}).

We introduce the notions of the \emph{small} and \emph{big horospherical limit sets} for such actions, and show that the action on the big horospherical limit set is conservative (cf.\ \Cref{prop:Wandering_Dissipative}). In the case of convex cocompact subgroups, the big horospherical limit set coincides with the limit set, implying that every big horospherical limit point is conical (cf.\ \Cref{prop:convex_cocompact}).

In this chapter,  
we do not discuss conformal densities of arbitrary dimensions, critical exponents of the Poincar\'e series, or the Hausdorff dimensions of limit sets for subgroups of the mapping class group. See \cite{MR1041575}, \cite{MR2235713} and \cite{MR2057305}, for instance.

\subsection{Selected problems}
We now state several problems closely related to the subject of this chapter. Some of the problems below are simplified; their precise statements appear in the indicated sections.
\begin{itemize}
\item
{[\S\ref{subsub:Her}, \Cref{problem:Her}]}
Does the Herglotz-type formula hold for positive pluriharmonic functions on Teichm\"uller space?
\item
{[\S\ref{subsub:Her}, \Cref{problem:bdd}]}
Characterize the boundary functions of bounded pluriharmonic or holomorphic functions on $\teich_{g,m}$.
\item
{[\S\ref{subsub:Her}, \Cref{problem:bdd2}]}
Characterize Borel measures whose Poisson integrals are bounded pluriharmonic functions on $\teich_{g,m}$.
\item
{[\S\ref{subsub:Her}, \Cref{problem:hardy}]}
Study the Hardy spaces on $\teich_{g,m}$.
\item
{[\S\ref{subsub:Her}, \Cref{problem:CR}]}
Study the analytic and geometric properties of bounded measurable function which satisfies the tangential CR-equation given in \eqref{eq:CR}.
\item
{[\S\ref{subsubsec:analytic_measures}, \Cref{problem:Fourier}]}
Is there an analogue of Fourier coefficients for Borel (regular) measures on $\pml_{g,m}$ that characterizes whether a given measure is an analytic measure?
\item
{[\S\ref{subsubsec:analytic_measures}, \Cref{problem:abs_conti}]}
Is every analytic measure absolutely continuous with respect to the Thurston measure?
\item
{[\S\ref{subsubsec:ergodic_decom}, \Cref{problem:erg}]}
Study the ergodic decomposition for the action of the Torelli group on $\pml_{g,m}$.
\item
{[\S\ref{subsec:Horospherical_limit_points}, \Cref{problem:10}]}
Complete the Sullivan dichotomy concerning the conservativity of the action on the horospherical limit set $\Lambda_H(H)$ for sufficiently large subgroups $H$ of $\mcg(\Sigma_{g,m})$.
\item
{[\S\ref{subsec:convex_cocompact}, \Cref{problem:7}]}
Is $H$ convex cocompact if $\Lambda(H)$ coincides with the small horospherical limit set $\Lambda_{h}(H)$?
\item
{[\S\ref{subsec:convex_cocompact}, \Cref{problem:12}]}
Is $H$ convex cocompact if all points in the big horospherical limit set $\Lambda_H(H)$ are uniquely ergodic?
\item
{[\S\ref{subsec:convex_cocompact}, \Cref{problem:15}]}
Is a convex cocompact group totally dissipative?
\item
{[\S\ref{subsec:erg_flow}, \Cref{problem:9}]}
Clarify the relationships among the following properties:  
(1) ergodicity of the horocyclic flow on $\teich_{g,m}/H$;  
(2) ergodicity of the action of $H$ on $\pml_{g,m}$;  
(3) the absence of non-constant bounded pluriharmonic functions on $\teich_{g,m}/H$.
\item
{[\S\ref{subsec:erg_flow}, \Cref{problem:11}]}
Complete the Hopf--Tsuji--Sullivan theory for subgroups of $\mcg(\Sigma_{g,m})$.
\item
{[\S\ref{subsec:erg_flow}, \Cref{problem:13}]}
Does there exist a subgroup $H$ of $\mcg(\Sigma_{g,m})$ that acts ergodically on $\pml_{g,m}$ but not on the product $\pml_{g,m} \times \pml_{g,m}$?
\item
{[\S\ref{subsec:Monodromy}, \Cref{problem:4-1}]}
For a holomorphic family $(\mathcal{M}, \pi, B)$ of Riemann surfaces of analytically finite type, does $\teich_{g,m}/H$ admit no non-constant bounded pluriharmonic functions when the base Riemann surface $B$ belongs to the class $\mathcal{O}_{HB}$?

\item
{[\S\ref{subsec:Monodromy}, \Cref{problem:4-2}]}
For a holomorphic family $(\mathcal{M}, \pi, B)$ of Riemann surfaces of analytically finite type, does $\teich_{g,m}/H$ admit no non-constant bounded holomorphic functions when the base Riemann surface $B$ belongs to the class $\mathcal{O}_{AB}$?

\item
{[\S\ref{subsec:Monodromy}, \Cref{problem:4}]}  
For a holomorphic family $(\mathcal{M}, \pi, B)$ of Riemann surfaces of analytically finite type, is the action of the monodromy group $\rho(\Gamma)$ ergodic when the base Riemann surface $B$ belongs to the class $\mathcal{O}_{HB}$?

\item
{[\S\ref{subsec:Monodromy}, \Cref{problem:5}]}  
For a holomorphic family $(\mathcal{M}, \pi, B)$ of Riemann surfaces of analytically finite type, is the action of $\rho(\Gamma)$ on $\pml$ minimal when the base Riemann surface $B$ belongs to the class $\mathcal{O}_{HB}$?

\item
{[\S\ref{subsec:Monodromy}, \Cref{problem:6}]}  
For a holomorphic family $(\mathcal{M}, \pi, B)$ of Riemann surfaces of type $(g,m)$, which properties of $B$ as a Riemann surface are inherited by $\teich_{g,m}/H$?  
For example, if $B$ belongs to the class $\mathcal{O}_G$, does $\teich_{g,m}/H$ also admit no pluricomplex Green function?
\end{itemize}

\subsection{About the chapter}
In \S\ref{sec:Bachgroud_Teich}, \S\ref{sec:Back_Thurston_theory}, and \S\ref{sec:Back_Geom_T-space}, we recall basic properties of the Teichm\"uller space and the space of measured geodesic laminations. In \S\ref{sec:Toy}, we discuss the Teichm\"uller space of tori and give precise calculations for notions introduced in the previous sections.
In \S\ref{sec:Function_theory_T}, we review the author's research on function theory on Teichm\"uller space. The Poisson integral formula for bounded pluriharmonic functions is stated. The proof of the formula will be given in a forthcoming paper \cite{bounded-PIF-miyachi2025}.
In \S\ref{sec:DynPML}, we recall the McCarthy and Papadopoulos classification of subgroups of the mapping class group.
In the final section, \S\ref{sec:function_theory_dymamics}, we discuss the action of subgroups of the mapping class group from the viewpoint of measurable actions.

\subsection*{Acknowledgements}
The author would like to express sincere gratitude to Professor Ken'ichi Ohshika and Professor Athanase Papadopoulos for their warm encouragement and continuous support throughout his research.

\section{Background from Teichm\"uller theory}
\label{sec:Bachgroud_Teich}
Let $\Sigma_{g,m}$ denote a connected, orientable smooth surface of genus $g$ with $m$ punctures. When $m = 0$, we write $\Sigma_g$ instead of $\Sigma_{g,0}$ for notational simplicity. Throughout this paper, we primarily assume that $2g - 2 + m > 0$, which guarantees that $\Sigma_{g,m}$ admits a complete hyperbolic structure. However, we occasionally consider the torus case, i.e., $(g, m) = (1, 0)$, when discussing examples.

For further details, see \cite{MR590044}, \cite{MR2241787}, \cite{MR2245223}, \cite{MR1215481}, and \cite{MR927291}.

\subsection{Quasiconformal maps}
Let $M$ be a Riemann surface.
An orientation-preserving homeomorphism $h$ from $M$ to another Riemann surface $M'$ is said to be \emph{$K$-quasiconformal}\index{quasiconformal map} for $K \ge 1$ if it is absolutely continuous on lines in each coordinate chart $(U, z)$ and satisfies
\begin{equation}
\label{eq:maximal_dilatation}
|\overline{\partial} h| \le \frac{K - 1}{K + 1} |\partial h|
\end{equation}
almost everywhere on $M$.
The infimum $K(h)$ over all $K \ge 1$ satisfying \eqref{eq:maximal_dilatation} is called the \emph{maximal dilatation}\index{maximal dilatation}\index{quasiconformal map!maximal dilatation} of $h$.

Let $L^\infty(M)$ be the complex Banach space of bounded measurable $(-1,1)$-forms $\mu=\mu(z)dz^{-1}d\overline{z}$ with the $L^\infty$-norm
$$
\|\mu\|_\infty=\mathop{{\rm ess.sup}}\{|\mu(p)|\mid p\in M\}.
$$
Let $BL^\infty(M)$ denote the open unit ball in $L^\infty(M)$.
By the measurable Riemann mapping theorem, for each $\mu \in BL^\infty(M)$, there exists a unique quasiconformal map $\qc{\mu}$ on $M$, up to post-composition with holomorphic automorphisms, which satisfies the Beltrami equation
\begin{equation}
\label{eq:measurable_Riemann_mapping_theorem}
\overline{\partial}(\qc{\mu}) = \mu\, \partial (\qc{\mu})
\end{equation}
(cf. \cite{MR0115006}).

Suppose $M$ is the Riemann sphere $\hat{\mathbb{C}}$, and fix three points in $\hat{\mathbb{C}}$. The \emph{normalized solution} $\normqc{\mu}$ for $\mu \in BL^\infty(\hat{\mathbb{C}})$ is the solution of the Beltrami equation \eqref{eq:measurable_Riemann_mapping_theorem} that fixes the chosen three points. The measurable Riemann mapping theorem also asserts that for any $z \in \hat{\mathbb{C}}$, the evaluation map
\begin{equation}
\label{eq:evaluation}
BL^\infty(\hat{\mathbb{C}}) \ni \mu \mapsto \normqc{\mu}(z)
\end{equation}
is holomorphic on the complex Banach manifold $BL^\infty(\hat{\mathbb{C}})$.

\subsection{Teichm\"uller space}
A \emph{marked Riemann surface of analytically finite type $(g,m)$}\index{marked Riemann surface} (or simply of type $(g,m)$) is a pair $(M,f)$, where $M$ is a Riemann surface of genus $g$ with $m$ punctures, and $f\colon \Sigma_{g,m}\to M$ is an orientation-preserving homeomorphism.
Two marked Riemann surfaces $(M_1,f_1)$ and $(M_2,f_2)$ of type $(g,m)$ are said to be \emph{Teichm\"uller equivalent}\index{Teichm\"uller equivalence}\index{Teichm\"uller space!Teichm\"uller equivalence} if there exists a biholomorphism $h\colon M_1\to M_2$ such that $h\circ f_1$ is homotopic to $f_2$.
The \emph{Teichm\"uller space}\index{Teichm\"uller space} $\teich_{g,m}$ is the space of Teichm\"uller equivalence classes of marked Riemann surfaces of analytically finite type $(g, m)$.
When $m = 0$, we write $\teich_g$ instead of $\teich_{g,0}$ for simplicity.

The Teichm\"uller space $\teich_{g,m}$ admits a natural complete metric, called the \emph{Teichm\"uller distance}\index{Teichm\"uller distance}\index{Teichm\"uller space!Teichm\"uller distance}, defined by
$$
d_T(x_1,x_2)=\dfrac{1}{2}\log\inf_h K(h)
$$
for $x_1=(M_1,f_1)$ and $x_2=(M_2,f_2)\in \teich_{g,m}$, where $h$ ranges over all quasiconformal maps $h\colon M_1\to M_2$ such that $h\circ f_1$ is homotopic to $f_2$.

\subsection{Teichm\"uller theorem}
\label{subsec:teichmuller_theorem}
Let $x = (M, f) \in \teich_{g,m}$. Let $\mathcal{Q}_x$ denote the complex Banach space of holomorphic quadratic differentials $q = q(z)dz^2$ on $M$, equipped with the $L^1$-norm
$$
\|q\|=\int_M|q(z)|\dfrac{\imaginaryunit}{2}dz\wedge d\overline{z}
$$
where $\imaginaryunit$ is the imaginary unit such that the orientation determined by the ordered base $(1,\imaginaryunit)$ of the underlying real vector space $\mathbb{C}$ coincides with a canonical orientation of the complex plane $\mathbb{C}$.

Any $q \in \mathcal{Q}_x$ extends to a meromorphic quadratic differential on the puncture-filled surface $\overline{M}$ associated with $M$, possibly having poles of order at most one at the punctures.
Let $\mathcal{BQ}_x$ denote the open unit ball in $\mathcal{Q}_x$.
By the Riemann--Roch theorem, the space $\mathcal{Q}_x$ is isomorphic to $\mathbb{C}^{3g - 3 + m}$.

Let $x_0 = (M_0, f_0)\in \teich_{g,m}$. For $q \in \mathcal{BQ}_{x_0}$,
We define a quasiconformal mapping $\teichmap{q}$ on $M_0$ by
$$
\teichmap{q}=
\begin{cases}
\qc{\|q\| \dfrac{\overline{q}}{|q|}} &(q \ne 0) \\
\operatorname{id}_{M_0} & (q=0),
\end{cases}
$$
where $\operatorname{id}_{M_0}$ is the identity map on $M_0$.
Let $M[q]$ denote the Riemann surface of type $(g,m)$ that is the image of $\teichmap{q}$.
We call $\teichmap{q} \colon M_0 \to M[q]$ the \emph{Teichm\"uller map}\index{Teichm\"uller map} associated with the holomorphic quadratic differential $q$.

Teichm\"uller \cite{MR0003242} showed that the map
\begin{equation}
\label{eq:Teichmuller_homeomorphism}
\mathfrak{T}_{x_0}\colon \mathcal{BQ}_{x_0} \ni q \mapsto (M[q], \teichmap{q} \circ f_0) \in \teich_{g,m}
\end{equation}
is a homeomorphism satisfying
\begin{equation}
\label{eq:Teichmuller-geodesic}
d_T\left(\mathfrak{T}_{x_0}(kq),\mathfrak{T}_{x_0}(k'q)\right)=\mathop{{\rm Arctanh}}\left|\frac{k-k'}{1-k'k}\right|
= d_{\mathbb{D}}(k,k')
\end{equation}
for $0 \le k, k' < 1$ and $q \in \mathcal{Q}_{x_0}$ with $\|q\| = 1$,  
where $d_{\mathbb{D}}$ is the hyperbolic distance on $\mathbb{D}$ with curvature $-4$ (cf. \cite{MR3497300}).  
In particular, the Teichm\"uller space $\teich_{g,m}$ is homeomorphic to an open ball of dimension $6g - 6 + 2m$, and the quasiconformal map $\teichmap{q} \colon M_0 \to M[q]$ is \emph{extremal}\index{quasiconformal map!extremal} in the sense that every quasiconformal map $h \colon M_0 \to M[q]$ homotopic to $\teichmap{q}$ satisfies
\[
\|q\| \le \frac{K(h) - 1}{K(h) + 1}.
\]
It is also known that the Teichm\"uller map $\teichmap{q}$ for $q \in \mathcal{BQ}_{x_0}$ is \emph{uniquely extremal}\index{Teichm\"uller map!uniquely extremal} in the sense that equality
\[
\|q\| = \frac{K(h) - 1}{K(h) + 1}
\]
implies $h = \teichmap{q}$ on $M_0$.

Thus, for $q \in \mathcal{Q}_{x_0}-\{0\}$, the map
\begin{equation}
\label{eq:Teichmuller-ray}
\TRay{q}(t) = \left(M\left[\tanh(t) \frac{q}{\|q\|}\right], \teichmap{\tanh(t) \frac{q}{\|q\|}} \circ f_0\right) \quad (t \ge 0)
\end{equation}
is a geodesic ray with respect to the Teichm\"uller distance.
This ray $\TRay{q}$ is called the \emph{Teichm\"uller geodesic ray}\index{Teichm\"uller geodesic ray} emanating from $x$, defined by $q \in \mathcal{Q}_x$.  
The geodesic space $(\teich_{g,m}, d_T)$ is a uniquely geodesic space in the sense that for any $x_1, x_2 \in \teich_{g,m}$, there exists a unique Teichm\"uller geodesic connecting $x_1$ and $x_2$.

\subsection{Complex structure and Holomorphic infinitesmal spaces}
\label{subsec:complex_structure_holomorphicinfinitesimal_spaces}
Let $x_0 = (M_0, f_0) \in \teich_{g,m}$. The map
\begin{equation}
\label{eq:Bers_projection}
\mathfrak{BP}\colon BL^\infty(M_0) \ni \mu \mapsto \left(\qc{\mu}(M), \qc{\mu} \circ f_0\right) \in \teich_{g,m}
\end{equation}
is surjective.
The Teichm\"uller space $\teich_{g,m}$ admits a complex structure that makes the projection \eqref{eq:Bers_projection} a holomorphic split submersion.

There exists a natural pairing between $L^\infty(M_0)$ and $\mathcal{Q}_{x_0}$ defined by
\begin{equation}
\label{eq:pairing_pre}
\llangle \mu, q \rrangle = \iint_{M_0} \mu(z) q(z) \frac{\imaginaryunit}{2} dz \wedge d\overline{z}
\end{equation}
for $\mu \in L^\infty(M_0)$ and $q \in \mathcal{Q}_{x_0}$.
The Teichm\"uller theorem asserts that the kernel of the differential $\left.D(\mathfrak{BP})\right|_0$ of the submersion $\mathfrak{BP}$ at the origin $0 \in BL^\infty(M_0)$ is given by
$$
\ker(\left.D(\mathfrak{BP})\right|_0)=\{\mu\in L^\infty(M_0)\mid \llangle \mu,q\rrangle=0,\ \forall q\in \mathcal{Q}_{x_0}\}.
$$
Therefore, the holomorphic tangent space $T_{x_0}(\teich_{g,m})$ of $\teich_{g,m}$ at $x_0 \in \teich_{g,m}$ is represented by
$$
T_{x_0}(\teich_{g,m}) = L^\infty(M_0) \big/ \ker\left(\left.D(\mathfrak{BP})\right|_0\right),
$$
and the pairing \eqref{eq:pairing_pre} descends to a pairing between $T_{x_0}(\teich_{g,m})$ and $\mathcal{Q}_{x_0}$ as
\begin{equation}
\label{eq:pairing}
\langle v,q\rangle=\langle [\mu],q\rangle=\llangle \mu,q\rrangle
\end{equation}
for $v = [\mu] \in T_{x_0}(\teich_{g,m})$ and $q \in \mathcal{Q}_{x_0}$.
This pairing is non-degenerate and establishes a duality between $T_{x_0}(\teich_{g,m})$ and $\mathcal{Q}_{x_0}$.
The pairing \eqref{eq:pairing} provides a natural identification between $\mathcal{Q}_{x_0}$ and the holomorphic cotangent space of $\teich_{g,m}$ at $x_0$.

Royden \cite{MR0288254} showed that under the complex sturcture, the Teichm\"uller distance coincides with the Kobayashi distance. Hence, $(\teich_{g,m},d_T)$ is complete Kobayashi hyperbolic. Earle \cite{MR0352450.1} showed that the Carath\'eodory distance on $\teich_{g,m}$ is also complete. Hence it is holomorphically convex \cite[Theorem 3.6]{MR2194466}, which was first shown by Bers and Ehrenpreis \cite{MR0168800}.

\subsection{Complex tangent and cotangent spaces}
%
In this section, we present a model for the holomorphic and anti-holomorphic tangent and cotangent bundles over Teichm\"uller space. To formulate this model, we characterize the eigenspace corresponding to the eigenvalue $-\imaginaryunit$, given a complex vector space realized as the eigenspace of the eigenvalue $\imaginaryunit$ under its natural complex structure (\Cref{prop:complex_structure}). This yields the eigenspace decomposition necessary to distinguish holomorphic from anti-holomorphic directions from an abstract point of view.

\subsubsection{Complex tangent spaces of complex manifolds}
We begin with a brief review of the basics of the complex tangent space of a general complex manifold. For references, see \cite[Chapter IX]{MR1393941}, \cite[p.~71]{MR507725}, and \cite[\S2.2]{MR2093043}.

Let $M$ be a complex manifold of dimension $n$.  
Consider the real tangent space $T_p(M)^{\mathbb{R}}$ at a point $p$ on the underlying real analytic manifold of $M$.  
The complexification $T_p(M)^{\mathbb{C}} = T_p(M)^{\mathbb{R}} \otimes_{\mathbb{R}} \mathbb{C}$ of the real tangent space is called the \emph{complex tangent space}\index{complex tangent space} of $M$.  
We have the decomposition
\begin{equation}
\label{eq:decomposition_complex_tangent}
T_p(M)^{\mathbb{C}} = T^{1,0}_p(M) \oplus T^{0,1}_p(M)
\end{equation}
where $T^{1,0}_p(M)$ and $T^{0,1}_p(M)$ are the eigenspaces corresponding to the eigenvalues $\imaginaryunit$ and $-\imaginaryunit$, respectively, for the action of the complex structure on $T_p(M)^{\mathbb{C}}$.

The space $T^{1,0}_p(M)$ is called the \emph{holomorphic tangent space}\index{holomorphic tangent space}\index{complex tangent space!holomorphic tangent space} of $M$, and is usually denoted by $T_p(M)$ in the context of complex geometry.  
The action of complex conjugation on $T_p(M)^{\mathbb{C}}$ interchanges $T^{1,0}_p(M)$ and $T^{0,1}_p(M)$.  
The real tangent space $T_p(M)^{\mathbb{R}}$ is the invariant subspace of $T_p(M)^{\mathbb{C}}$ under complex conjugation.
%

Let $T^*_p(M)^{\mathbb{R}}$ denote the real cotangent space at $p \in M$ associated with the underlying real analytic structure of $M$, and let $T^*_p(M)^{\mathbb{C}} = T^*_p(M)^{\mathbb{R}} \otimes_{\mathbb{R}} \mathbb{C}$ be its complexification, which is called the \emph{complex cotangent space}\index{complex cotangent space}.  
A natural pairing between $T_p(M)^{\mathbb{R}}$ and $T^*_p(M)^{\mathbb{R}}$ extends naturally to their complexifications.

The decomposition \eqref{eq:decomposition_complex_tangent} induces the decomposition
\begin{equation}
\label{eq:decomposition_complex_cotangent}
T^*_p(M)^{\mathbb{C}} = \mathfrak{C}^{1,0}_p(M) \oplus \mathfrak{C}^{0,1}_p(M)
\end{equation}
such that $\mathfrak{C}^{1,0}_p(M)$ (resp. $\mathfrak{C}^{0,1}_p(M)$) is the dual space to $T^{1,0}_p(M)$ (resp. $T^{0,1}_p(M)$) under the duality pairing.
The $(1,0)$-part $\mathfrak{C}^{1,0}_p(M)$ is also called the \emph{holomorphic cotangent space}\index{holomorphic cotangent space}\index{complex cotangent space!holomorphic cotangent space} and is denoted by $T_p^*(M)$ in the context of complex geometry.
Under the decompositions \eqref{eq:decomposition_complex_tangent} and \eqref{eq:decomposition_complex_cotangent}, for $v_1 + v_2 \in T_p(M)^{\mathbb{C}}$ and $\omega_1 + \omega_2 \in T^*_p(M)^{\mathbb{C}}$, the natural pairing between $T_p(M)^{\mathbb{C}}$ and $T^*_p(M)^{\mathbb{C}}$ is given by
\begin{equation}
\label{eq:duality_complexification}
(\omega_1 + \omega_2)(v_1 + v_2) = \omega_1(v_1) + \omega_2(v_2).
\end{equation}

\subsubsection*{Calculations in a chart}
Let $z = (z_1, \dots, z_n)$ be a local coordinate chart of $M$ around $p$ as a complex manifold, and set $z_k = x_k + y_k\imaginaryunit$ for $k = 1, \dots, n$.  
Then, bases of $T^{1,0}_p(M)$ and $T^{0,1}_p(M)$ are given by
$$
\frac{\partial}{\partial z_k} = \frac{1}{2} \left( \frac{\partial}{\partial x_k} - \imaginaryunit \frac{\partial}{\partial y_k} \right), \quad
\frac{\partial}{\partial \overline{z}_k} = \frac{1}{2} \left( \frac{\partial}{\partial x_k} + \imaginaryunit\frac{\partial}{\partial y_k} \right)
$$
for $k = 1, \dots, n$.  
Using these bases, the decomposition \eqref{eq:decomposition_complex_tangent} of $T_p(M)^{\mathbb{C}}$ is represented by
\begin{equation}
\label{eq:decomposition_complex_tangent2}
\sum_{k=1}^n \left( a_k \frac{\partial}{\partial x_k} + b_k \frac{\partial}{\partial y_k} \right)
=
\sum_{k=1}^n (a_k + b_k\imaginaryunit) \frac{\partial}{\partial z_k}
+
\sum_{k=1}^n (a_k - b_k\imaginaryunit) \frac{\partial}{\partial \overline{z}_k}
\end{equation}
for $a_k, b_k \in \mathbb{C}$.
Therefore, the underlying real tangent space $T_p(M)^{\mathbb{R}}$ is realized in $T_p(M)^{\mathbb{C}}=T^{1,0}_p(M)\oplus T^{0,1}_p(M)$ by the map
\begin{equation}
\label{eq:real_T}
\xymatrix@R=1mm@C=10mm{
T_p(M)^{\mathbb{R}} \ar[r] &T_p(M)^{\mathbb{C}}=T^{1,0}_p(M)\oplus T^{0,1}_p(M) \\
\frac{\partial}{\partial x_k} \ar@{|->}[r] &
\left(\frac{\partial}{\partial z_k},\frac{\partial}{\partial \overline{z}_k}\right)
=
\left(\frac{\partial}{\partial z_k},\overline{\frac{\partial}{\partial z_k}}\right) \\
\frac{\partial}{\partial y_k} \ar@{|->}[r] &
\imaginaryunit\left(\frac{\partial}{\partial z_k},-\frac{\partial}{\partial \overline{z}_k}\right)
=
\left(\imaginaryunit\frac{\partial}{\partial z_k},\overline{\imaginaryunit\frac{\partial}{\partial z_k}}\right)
}
\end{equation}
that induces an isomorphism $T_p(M)^{\mathbb{R}} \otimes_{\mathbb{R}}\mathbb{C}\cong T_p(M)^{\mathbb{C}}$.

In the case of cotangent spaces, the real and complex cases are related by
\begin{align}
\label{eq:decomposition_complex_cotangent2}
\sum_{k=1}^n\left(a_kdx_k+b_kdy_k\right)
&=\sum_{k=1}^n\left(a_k\frac{1}{2}(dz_k+d\overline{z}_k)+b_k\frac{1}{2i}(dz_k-d\overline{z}_k)\right) \\
&=\sum_{k=1}^n\left(\frac{a_k-b_k\imaginaryunit}{2}dz_k+\frac{a_k+b_k\imaginaryunit}{2}d\overline{z}_k\right)
\nonumber
\end{align}
for $a_k$, $b_k\in\mathbb{C}$.
Therefore, the underlying real cotangent space $T^*_p(M)^{\mathbb{R}}$ is realized in $T^*_p(M)^{\mathbb{C}}=\mathfrak{C}^{1,0}_p(M) \oplus \mathfrak{C}^{0,1}_p(M)$ by the map
\begin{equation}
\label{eq:real_Tstar}
\xymatrix@R=1mm@C=10mm{
T^*_p(M)^{\mathbb{R}} \ar[r] &T^*_p(M)^{\mathbb{C}}=\mathfrak{C}^{1,0}_p(M) \oplus \mathfrak{C}^{0,1}_p(M) \\
dx_k \ar@{|->}[r] &
\frac{1}{2}\left(dz_k,d\overline{z}_k\right)
=
\left(\frac{1}{2}dz_k,\overline{\frac{1}{2}dz_k}\right) \\
d y_k \ar@{|->}[r] &
\frac{1}{2\imaginaryunit}\left(dz_k,-d\overline{z}_k\right)
=
\left(\frac{1}{2\imaginaryunit}dz_k,\overline{\frac{1}{2\imaginaryunit}dz_k}\right).
}
\end{equation}
that induces an isomorphism $T^*_p(M)^{\mathbb{R}} \otimes_{\mathbb{R}}\mathbb{C}\cong T^*_p(M)^{\mathbb{C}}$.

The two realizations \eqref{eq:real_T} and \eqref{eq:real_Tstar} are compatible with duality in the sense that the pairing on $T_p(M)^{\mathbb{R}} \times T^*_p(M)^{\mathbb{R}}$ coincides with the restriction of the pairing on $T_p(M)^{\mathbb{C}} \times T^*_p(M)^{\mathbb{C}}$.
For instance, for a $C^1$ function $f$ on $M$, the total derivative decomposes as $df = \partial f + \overline{\partial} f$, as in \eqref{eq:decomposition_complex_cotangent2}, and a vector $X \in T_p(M)^{\mathbb{R}}$ can be written as $X = Z + \overline{Z}$ for some $Z \in T^{1,0}_p(M)$, as in \eqref{eq:decomposition_complex_tangent2}.
In this case,
\begin{equation}
\label{eq:decomposition_complex_tangent3}
df(X) = \partial f(Z) + \overline{\partial} f(\overline{Z})
\end{equation}
as given in \eqref{eq:duality_complexification}.


\subsubsection{Complexification of the underlying real vector space}
Let $V$ be a complex vector space of dimension $n$, and let $V^*$ denote the $\mathbb{C}$-dual space of $V$. Let $V^{\mathbb{R}}$ and $(V^*)^{\mathbb{R}}$ be the underlying real vector spaces of $V$ and $V^*$. The complex structure $J$ on $V^{\mathbb{R}}$ is induced by the multiplication map $v \mapsto \imaginaryunit v$ on $V$.  The complex structure $J^*$ on the dual space $(V^*)^{\mathbb{R}}$ is defined by $J^*(\varphi)(v)=\varphi(J(v))$ for $\varphi \in V^*$.

Let $W$ be a $\mathbb{C}$-vector space of dimension $n$, and let $W^*$ denote its $\mathbb{C}$-dual.
Suppose there exist anti-$\mathbb{C}$-linear isomorphisms $A \colon V \to W$ and $A^* \colon V^* \to W^*$ such that
\begin{equation}
\label{eq:conjugations}
A^*(\varphi)(A(v)) = \overline{\varphi(v)} \quad \text{for all } v \in V \text{ and } \varphi \in V^*.
\end{equation}
Under this condition, the spaces $W$ and $W^*$ can be regarded as the anti-holomorphic parts of $V$ and $V^*$, respectively, as follows.

We identify the underlying real vector spaces $V^\mathbb{R}$ and $(V^*)^\mathbb{R}$ with subspaces of $V \oplus W$ and $V^* \oplus W^*$, respectively, via
\begin{align*}
V^\mathbb{R} &\ni v \mapsto (v, A(v)) \in V \oplus W, \\
(V^*)^\mathbb{R} &\ni \varphi \mapsto (\varphi, A^*(\varphi)) \in V^* \oplus W^*.
\end{align*}
Notice that the two realizations \eqref{eq:real_T} and \eqref{eq:real_Tstar}
are obtained in this manner when $A$ and $A^*$ are chosen to be complex conjugation.

The actions of the complex structures $J$ and $J^*$ on $V\oplus W$ and $V^*\oplus W^*$ are represented as
\[
\xymatrix@C=15mm@R=3mm{
V\oplus W \ni (v,w) \ar@{|->}[r]^{J} & (\imaginaryunit v,-\imaginaryunit w)\in V\oplus W \\
V^*\oplus W^* \ni (\varphi,\psi) \ar@{|->}[r]^{J^*} & (\imaginaryunit\varphi,-\imaginaryunit\psi)\in V^*\oplus W^*.
}
\]
Hence, the direct products $V\oplus W$ and $V^*\oplus W^*$ represents the decompositions of the eigenspaces of the complex structures $J$ and $J^*$ such that $V$ and $V^*$ are the eigenspaces of eigenvalue $\imaginaryunit$ and $W$ and $W^*$ are the eigenspaces of eigenvalue $-\imaginaryunit$.
The duality between $V\oplus W$ and $V^*\oplus W^*$ is given by the pairing
\[
(\varphi,\psi)(v,w)=\varphi(v)+\psi(w).
\]
If we restrict the pairing to the underlying real vector spaces, we have
\[
(\varphi,A^*(\varphi))(v,A(v))=\varphi(v)+A^*(\varphi)(A(v))=\varphi(v)+\overline{\varphi(v)}.
\]

We now discuss the naturality of the construction. Let $X$ be a $\mathbb{C}$-vector space of dimension $n$, and let $X^*$ denote its $\mathbb{C}$-dual.  
Suppose that there exist anti-$\mathbb{C}$-linear isomorphisms $B \colon V \to X$ and $B^* \colon V^* \to X^*$ such that
\[
B^*(\varphi)(B(v)) = \overline{\varphi(v)}
\]
for all $\varphi \in V^*$ and $v \in V$.
Then the $\mathbb{C}$-linear isomorphisms
\[
\begin{aligned}
V \oplus W &\ni (v, w) \mapsto (v, B \circ A^{-1}(w)) \in V \oplus X, \\
V^* \oplus W^* &\ni (\varphi, \psi) \mapsto (\varphi, B^* \circ (A^*)^{-1}(\psi)) \in V^* \oplus X^*
\end{aligned}
\]
preserve the underlying real vector spaces and the natural duality pairings, and respect the decomposition into eigenspaces associated with the complex structures acting on $V$ and $V^*$ (and hence on $V^{\mathbb{R}}$ and $(V^*)^{\mathbb{R}}$).  
For clarity, we summarize this as follows.

\begin{proposition}[Characterization of anti-holomorphic part]
\label{prop:complex_structure}
Let $V$ be a complex vector space of dimension $n$, and let $V^*$ denote its $\mathbb{C}$-dual. Let $J$ be the complex structure on the underlying real vector space $V^{\mathbb{R}}$, defined by the multiplication map $v \mapsto \imaginaryunit v$. Let $J^*$ denote the complex structure on the real vector space underlying $V^*$ that is dual to $J$.

Let $W$ be a complex vector space of dimension $n$, and let $W^*$ denote its $\mathbb{C}$-dual.
Suppose that there exist anti-complex linear isomorphisms $A\colon V \to W$ and $A^\colon V^* \to W^*$ such that $A^*(\varphi)(A(v))=\overline{\varphi(v)}$ for $\varphi\in V^*$ and $v\in V$.
We identify $V^{\mathbb{R}}$ and $(V^*)^{\mathbb{R}}$ with subspaces of $V \oplus W$ and $V^* \oplus W^*$, respectively, via the maps
\begin{equation}
\label{eq:real_v}
\begin{cases}
V^{\mathbb{R}} \ni v \mapsto (v, A(v)) \in V \oplus W, \\
(V^*)^{\mathbb{R}} \ni \varphi \mapsto (\varphi, A^*(\varphi)) \in V^* \oplus W^*.
\end{cases}
\end{equation}
Then, under the identification in \eqref{eq:real_v}, the spaces $V \oplus W$ and $V^* \oplus W^*$ are regarded as the complexifications of $V^{\mathbb{R}}$ and $(V^*)^{\mathbb{R}}$, respectively.  
In this identification, the decompositions $V \oplus W$ and $V^* \oplus W^*$ correspond to the eigenspace decompositions of the complex structures $J$ and $J^*$ with eigenvalues $\imaginaryunit$ and $-\imaginaryunit$.

Moreover, these decompositions are natural in the following sense:  
For any complex vector space $X$, its $\mathbb{C}$-dual $X^*$, and anti-$\mathbb{C}$-linear isomorphisms $B \colon V \to X$ and $B^* \colon V^* \to X^*$ satisfying
\[
B^*(\varphi)(B(v)) = \overline{\varphi(v)} \quad (\varphi \in V^*,\, v \in V),
\]
if we identify the underlying real vector spaces $V^{\mathbb{R}}$ and $(V^*)^{\mathbb{R}}$ via
\[
\begin{cases}
V^{\mathbb{R}} \ni v \mapsto (v, B(v)) \in V \oplus X, \\
(V^*)^{\mathbb{R}} \ni \varphi \mapsto (\varphi, B^*(\varphi)) \in V^* \oplus X^*,
\end{cases}
\]
then the isomorphisms
\[
\begin{aligned}
V \oplus W &\ni (v, w) \mapsto (v, B \circ A^{-1}(w)) \in V \oplus X, \\
V^* \oplus W^* &\ni (\varphi, \psi) \mapsto (\varphi, B^* \circ (A^*)^{-1}(\psi)) \in V^* \oplus X^*
\end{aligned}
\]
preserve the underlying real vector spaces and the natural pairings, and are compatible with the above eigenspace decompositions.

\end{proposition}

\subsubsection{Model of complex tangent and cotangent spaces of Teichm\"uller space}
\label{subsubsec:models_tangent}
Fix a point $x_0=(M_0,f_0)\in \teich_{g,m}$.
We consider the complex Banach space $L^\infty_{(1,-1)}(M_0)$ of bounded measurable $(1,-1)$-forms $\nu=\nu(z)dzd\overline{z}^{-1}$ on $M_0$.
Let $\overline{\mathcal{Q}}_{x_0}$ be the Banach space of integrable anti-holomorphic $(0,2)$-forms $Q=Q(z)d\overline{z}^2$ with $L^1$-norm on $M_0$.
Then, there is a natural pairing
\begin{equation}
\label{eq:01-pairing_pre}
\llangle\nu,Q\rrangle^-=\iint_{M_0}\nu(z)Q(z)\frac{\sqrt{-1}}{2}dz\wedge d\overline{z}.
\end{equation}
We define the quotient space
$$
T'_{x_0}(\teich_{g,m})=L^\infty_{(1,-1)}(M_0)/\{\nu\in L^\infty_{(1,-1)}(M_0)\mid
\mbox{$\llangle\nu,Q\rrangle^-=0$, $\forall Q\in \overline{\mathcal{Q}}_{x_0}$}\}.
$$
Then, the pairing \eqref{eq:01-pairing_pre} descends to the pairing
$$
T'_{x_0}(\teich_{g,m})\times  \overline{\mathcal{Q}}_{x_0}
\ni ([\nu],Q)\mapsto \langle [\nu],Q\rangle^-=\llangle\nu,Q\rrangle^-.
$$
The complex conjugations
\begin{align*}
T_{x_0}(\teich_{g,m})&\ni [\mu]\mapsto \overline{[\mu]}=[\overline{\mu}]\in T'_{x_0}(\teich_{g,m}) \\
\mathcal{Q}_{x_0}&\ni q\mapsto \overline{q}\in \overline{\mathcal{Q}}_{x_0}
\end{align*}
are anti-complex isomorphisms satisfying
$$
\langle\overline{[\mu]},\overline{q}\rangle^-=\overline{\langle [\mu],q\rangle}.
$$
Set
\begin{align*}
T_{x_0}(\teich_{g,m})^{\mathbb{C}}&=T_{x_0}(\teich_{g,m})\oplus T'_{x_0}(\teich_{g,m}) \\
T_{x_0}^*(\teich_{g,m})^{\mathbb{C}}&=\mathcal{Q}_{x_0}\oplus \overline{\mathcal{Q}}_{x_0}.
\end{align*}
From \Cref{prop:complex_structure},
the spaces $T_{x_0}(\teich_{g,m})^{\mathbb{C}}$ and $T^*_{x_0}(\teich_{g,m})^{\mathbb{C}}$ are regarded as models of the complexifications of the real tangent space and the real cotangent space of the Teichm\"uller space.
We denote by $[\mu]\oplus [\nu]$ and $q\oplus Q$ elements of $T_{x_0}(\teich_{g,m})^{\mathbb{C}}$ and $T^*_{x_0}(\teich_{g,m})^{\mathbb{C}}$,
where $[\mu]\in T_{x_0}(\teich_{g,m})$, $[\nu]\in T'_{x_0}(\teich_{g,m})$,
$q\in \mathcal{Q}_{x_0}$ and $Q\in \overline{\mathcal{Q}}_{x_0}$.
Under this situation, the models of the underlying real tangent and cotangent spaces are
\begin{align*}
&\{[\mu]\oplus [\overline{\mu}]\mid [\mu]\in T_{x_0}(\teich_{g,m})\}\subset  T_{x_0}(\teich_{g,m})^{\mathbb{C}}\\
&\{q \oplus \overline{q}\mid q\in \mathcal{Q}_{x_0}\}\subset
T_{x_0}(\teich_{g,m})^{\mathbb{C}}.
\end{align*}
Furthermore, it is natural to associate the real tangent vector directed by a Beltrami differential $\mu$ with  
\begin{equation}
\label{eq:holomorphic_tangent_to_real}
v=[\mu]\oplus [\overline{\mu}]\in T_{x_0}(\teich_{g,m})\oplus T'_{x_0}(\teich_{g,m})=T_{x_0}(\teich_{g,m})^{\mathbb{C}}.
\end{equation}
Following \eqref{eq:decomposition_complex_cotangent}, we define a pairing between $T_{x_0}(\teich_{g,m})^{\mathbb{C}}$ and $T_{x_0}^*(\teich_{g,m})^{\mathbb{C}}$ by
$$
\langle[\mu]\oplus [\nu],q\oplus Q\rangle^{c}=\langle [\mu],q\rangle+\langle [\nu],Q\rangle^-,
$$
where the superscript ``$c$" in the left-hand side stands for the first letter of ``complexification". Note that this pairing is consistent with the pairing defined in \eqref{eq:pairing} between the holomorphic tangent and cotangent spaces.

The complex conjugations on $T_{x_0}(\teich_{g,m})^{\mathbb{C}}$ and $T^*_{x_0}(\teich_{g,m})^{\mathbb{C}}$ is defined by
$$
\xymatrix@C=36pt@R=4pt@M=4pt{
T_{x_0}(\teich_{g,m})^{\mathbb{C}} \ar[r]^-{\overline{\cdot}} & T_{x_0}(\teich_{g,m})^{\mathbb{C}} \\
[\mu]\oplus [\nu] \ar@{|->}[r] & \overline{[\mu]\oplus [\nu]}=[\overline{\nu}]\oplus [\overline{\mu}]
}
$$
$$
\xymatrix@C=36pt@R=4pt@M=4pt{
T^*_{x_0}(\teich_{g,m})^{\mathbb{C}} \ar[r]^-{\overline{\cdot}} & T^*_{x_0}(\teich_{g,m})^{\mathbb{C}} \\
q\oplus Q \ar@{|->}[r] & \overline{q\oplus Q}=\overline{Q}\oplus \overline{q}.
}
$$
Thus, the real tangent space $T_{x_0}(\teich_{g,m})^{\mathbb{R}}$ and the real cotangent space $T^*_{x_0}\!(\teich_{g,m})^{\mathbb{R}}$  at $x_0$ are identified with the invariant subspaces of $T_{x_0}(\teich_{g,m})^{\mathbb{C}}$ and $T^*_{x_0}(\teich_{g,m})^{\mathbb{C}}$ under the complex conjugation.

\subsection{The Bers embedding and the Bers compactification}
Represent $M_0 = \mathbb{D} / \Gamma_0$ by a Fuchsian group $\Gamma_0$ acting on the unit disk $\mathbb{D}$. Let $A_2(\mathbb{D}^*, \Gamma_0)$ denote the complex Banach space of holomorphic automorphic forms $\varphi$ of weight $-4$ on $\mathbb{D}^* = \hat{\mathbb{C}} \setminus \overline{\mathbb{D}}$ with respect to $\Gamma_0$, equipped with the norm
\[
\|\varphi\|_\infty = \sup_{z \in \mathbb{D}} (|z|^2 - 1)^2 |\varphi(z)|.
\]

For $\mu \in BL^\infty(M_0)$, let $\tilde{\mu} \in BL^\infty(\mathbb{D})$ be a lift of $\mu$ via the covering map $\mathbb{D} \to M_0 = \mathbb{D} / \Gamma_0$. Extend $\tilde{\mu}$ to $\hat{\mathbb{C}}$ by setting it to zero outside the unit disk. Fix three points on $\partial \mathbb{D}$, and let $\normqc{\tilde{\mu}}$ be the normalized quasiconformal map solving the Beltrami equation with coefficient $\tilde{\mu}$ that fixes those three points.

Then the map
\begin{equation}
\label{eq:Bers_projection2}
BL^\infty(M_0) \ni \mu \mapsto
\frac{(\normqc{\tilde{\mu}}|_{\mathbb{D}^*})'''}{(\normqc{\tilde{\mu}}|_{\mathbb{D}^*})'}
-
\frac{3}{2}
\left( \frac{(\normqc{\tilde{\mu}}|_{\mathbb{D}^*})''}{(\normqc{\tilde{\mu}}|_{\mathbb{D}^*})'} \right)^2
\in A_2(\mathbb{D}^*, \Gamma_0)
\end{equation}
defined via the Schwarzian derivative descends to a holomorphic embedding
\[
\mathfrak{B}_{x_0} \colon \teich_{g,m} \to A_2(\mathbb{D}^*, \Gamma_0).
\]
This embedding $\mathfrak{B}_{x_0}$ is called the \emph{Bers embedding}\index{Bers embedding} of $\teich_{g,m}$ with basepoint $x_0$.  
Its image $\Bers{x_0} \subset A_2(\mathbb{D}^*, \Gamma_0)$ is referred to as the \emph{Bers slice}\index{Bers slice} with basepoint $x_0$. The term “slice” comes from the fact that the image of the Bers embedding coincides with a component of the intersection between the quasi-Fuchsian space and the space of complex projective structures on $x_0$, within the space of $\psl_2(\mathbb{C})$-representations of $\pi_1(\Sigma_{g,m})$ (cf.~\cite{MR910228}).

The Nehari--Kraus theorem asserts that $\Bers{x_0}$ is a bounded domain that contains a ball of radius $2$ and is contained in a ball of radius $6$.  
Hence, the closure $\overline{\Bers{x_0}}$ in $A_2(\mathbb{D}^*, \Gamma_0)$ provides a compactification of the Teichm\"uller space $\teich_{g,m} \cong \Bers{x_0}$.  
This compactification is called the \emph{Bers compactification}\index{Bers compactification} of $\teich_{g,m}$ with basepoint $x_0$.

Since $\teich_{g,m}$ is holomorphically convex, the Bers slice $\Bers{x_0}$ is a pseudoconvex domain. Furthermore, Shiga \cite{MR766636} showed that $\Bers{x_0}$ is polynomially convex, meaning that any holomorphic function on $\Bers{x_0}$ can be approximated uniformly on compact subsets of $\Bers{x_0}$ by holomorphic functions defined on the entire space $A_2(\mathbb{D}^*, \Gamma_0)$.
Krushkal \cite{MR1119946} observed that the Teichm\"uller space $\teich_{g,m}$ is hyperconvex.

\subsection{Action of the mapping class group}
\label{subsec:Action_MCG}
The \emph{mapping class group}\index{mapping class group} $\operatorname{MCG}(\Sigma_{g,m})$ of $\Sigma_{g,m}$ is the group of homotopy classes of orientation-preserving homeomorphisms of $\Sigma_{g,m}$.
We denote by $[\omega]$ the mapping class of an orientation-preserving homeomorphism $\omega$ of $\Sigma_{g,m}$.
The mapping class group $\operatorname{MCG}(\Sigma_{g,m})$ acts on $\teich_{g,m}$ by
\begin{equation}
\label{eq:mcg_action_T}
[\omega](M,f)=(M,f\circ \omega^{-1}).
\end{equation}
The action is isometric in terms of the Teichm\"uller distance. The action of the mapping class group corresponds to the change of markings and the quotient map
$$
\teich_{g,m}\to \mathcal{M}_{g,m}=\teich_{g,m}/\operatorname{MCG}(\Sigma_{g,m})
$$
forgets the markings.
The space $\mathcal{M}_{g,m}$ is the \emph{moduli space}\index{moduli space} of Riemann surfaces of analytically finite type $(g,m)$.

\subsubsection*{{\bf Action as Automorphisms}}
The action \eqref{eq:mcg_action_T} of the mapping class group on the Teichm\"uller space is biholomorphic with respect to the complex structure. The homomorphism induced by this action,
\[
\operatorname{MCG}(\Sigma_{g,m}) \to \operatorname{Aut}(\teich_{g,m}),
\]
is an isomorphism except in the cases $(g,m) = (0,3)$, $(0,4)$, $(1,1)$, $(1,2)$, and $(2,0)$. The image  $\operatorname{Mod}_{g,m}$ of the homomorphism is called the \emph{Teichm\"uller modular group}\index{Teichm\"uller modular group}.

It is known that $\teich_{0,3}$ consists of a single point, and that
\[
\teich_{0,4} \cong \teich_{1,1}, \quad
\teich_{0,5} \cong \teich_{1,2}, \quad
\teich_{0,6} \cong \teich_{2,0}.
\]
Since both $\teich_{0,4}$ and $\teich_{1,1}$ can be identified biholomorphically with the upper half-plane, we have
\[
\operatorname{Aut}(\teich_{0,4}) \cong \operatorname{Aut}(\teich_{1,1}) \cong \psl_2(\mathbb{R}).
\]
The mapping class group $\operatorname{MCG}(\Sigma_{1,2})$ does not act effectively on $\teich_{1,2}$, and
\[
\operatorname{Aut}(\teich_{1,2}) \cong \operatorname{Aut}(\teich_{0,5}) \cong \operatorname{MCG}(\Sigma_{0,5}).
\]
Similarly, the mapping class group $\operatorname{MCG}(\Sigma_{2,0})$ does not act effectively on $\teich_{2,0}$, and
\[
\operatorname{Aut}(\teich_{2,0}) \cong \operatorname{Aut}(\teich_{0,6}) \cong \operatorname{MCG}(\Sigma_{0,6}) \cong \operatorname{MCG}(\Sigma_{2,0}) / \mathbb{Z}_2.
\]
See Earle--Kra \cite{MR0430319} or Patterson \cite{MR299774, MR310239} for further details.

\section{Background from Thurston theory}
\label{sec:Back_Thurston_theory}
In this section, we briefly recall Thurston's theory on surface topology. We give the definition and fundamental properties of the space of measured laminations, and then define the Thurston compactification of the Teichm\"uller space.

%
%

\subsection{Measured laminations}
In this section, we shall recall the basic properties measured geodesic laminations. For references, see \cite{MR1810534}.

\subsubsection{Geometric formulation}
A \emph{geodesic lamination}\index{geodesic lamination} on a hyperbolic surface $M$ is a closed set consisting of disjoint complete geodesics, called \emph{leaves}\index{leaf}\index{geodesic lamination!leaf}. A \emph{transverse measure}\index{transverse measure}\index{measured geodesic lamination!transverse measure} for a geodesic lamination $\lambda$ is an assignment of a Radon measure to each arc $k$ transverse to $\lambda$, subject to the following two conditions:
\begin{enumerate}
\item If an arc $k'$ is contained in a transverse arc $k$, then the measure assigned to $k'$ is the restriction of the measure assigned to $k$; and
\item
If two arcs $k$ and $k'$ are homotopic through a family of arcs transverse to $\lambda$, then the homotopy sends the measure on $k$ to the measure on $k'$.
\end{enumerate}
A \emph{measured geodesic lamination}\index{measured geodesic lamination} $\lambda$ consists of a compact geodesic lamination $|\lambda|$, called the \emph{support}\index{support}\index{measured geodesic lamination!support} of $\lambda$, endowed with a transverse measure (also denoted by $\lambda$) whose support coincides with the entire set $|\lambda|$. For $t > 0$ and a measured geodesic lamination $\lambda$, the notation $t\lambda$ denotes the measured lamination with the same support as $\lambda$ and with transverse measure scaled by a factor of $t$; that is, $t\lambda$ assigns $t$ times the measure that $\lambda$ assigns to each transverse arc.
A simple closed geodesic $\alpha$ on $M$ is recognized as a measured lamination in the manner that for any transverse arc $k$ to $\alpha$, the transverse measure is defined as the Dirac measure with atoms at $k\cap \alpha$.

An arc $k$ in $M$ is said to be \emph{generic} with respect to simple geodesics if it is transverse to every simple geodesic on $M$. Birman and Series showed that the union of all simple geodesics has Hausdorff dimension $1$. Hence, almost every geodesic arc is generic, and every arc can be approximated by a generic arc (cf. \cite{MR793185}).

Let $M$ be a complete hyperbolic surface of genus $g$ with $m$ punctures.
Let $\ml(M)$ denote the set of measured geodesic laminations on $M$.
A sequence $\{\lambda_n\}_n \subset \ml(M)$ is said to converge to a measured geodesic lamination $\lambda \in \ml(M)$ if, for any generic arc $k$, the transverse measures on $k$ assigned by $\lambda_n$ converge weakly to the transverse measure assigned by $\lambda$.
Thurston showed that there exists a finite family $\{k_j\}_{j=1}^N$ of generic arcs on $M$ such that the map
$$
\ml(M)\ni \lambda \mapsto
\left(\lambda(k_1), \ldots, \lambda(k_N)\right)\in \mathbb{R}^{N}
$$
where $\lambda(k)$ denotes the total mass of the transverse measure on $k$ assigned by $\lambda$, defines a homeomorphism from $\ml(M)$ onto a piecewise linear submanifold of $\mathbb{R}^N$.

The \emph{intersection number}\index{intersection number}\index{measued geodesic lamination!intersection number} between $\lambda\in \ml(M)$ and the weighted simple closed geodesic $t\alpha$ (that is a measured lamination) is defined by
$$
I_M(\lambda,t\alpha)=t\inf_{\alpha'\sim \alpha}\lambda(\alpha').
$$
In particular, for weighted simple geodesics $t\alpha$ and $s\beta$,
the intersection number
$$
I_M(t\alpha,s\beta)=ts \inf_{\beta'\sim \beta}\alpha(\beta')
=ts\,\inf\{\alpha\cap \beta'\mid \beta'\sim \beta\}
=ts\,i(\alpha,\beta)
$$
coincides with $ts$ times the geometric intersection number $i(\alpha,\beta)$ of $\alpha$, $\beta\in \mathcal{S}$.
Thurston observed that the intersection number extends continuously to $\ml(M) \times \ml(M)$, and satisfies $I_M(\lambda_1, \lambda_2) = I_M(\lambda_2, \lambda_1)$ for all $\lambda_1, \lambda_2 \in \ml(M)$.
In its original definition, the intersection number $I_M(\lambda_1, \lambda_2)$ on $M$ is given as the total mass of a measure $\lambda_1 \times \lambda_2$, where $\lambda_1 \times \lambda_2$ is defined as the product of the two transverse measures in any small open set where the laminations are transverse to each other, and is set to be zero on the leaves that the two laminations have in common.

\subsubsection{Topological formulation}
\label{subsub:topological_ML}
We discuss the topological description of measured geodesic laminations.
Let $\mathcal{S}$ be the set of homotopy classes of non-trivial and non-peripheral simple closed curves on $\Sigma_{g,m}$. Consider the closure $\ml_{g,m}$ of the image
$$
\mathbb{R}_{+}\times \mathcal{S}\ni (t,\alpha)\mapsto [\mathcal{S}\ni \beta\mapsto t\,i(\alpha,\beta)]\in \mathbb{R}_{\ge 0}^{\mathcal{S}}
$$
with respect to the topology of pointwise convergence,
where $\mathbb{R}_+$ is the set of positive real numbers.
By definition, when $\lambda_{t,\alpha}\in \ml_{g,m}$ is the image of $(t,\alpha)\in \mathbb{R}_+\times \mathcal{S}$,
$$
\lambda_{t,\alpha}(\beta)=t\,i(\alpha,\beta),
$$
that is regarded as the intersection number between a weighted simple closed curve $(t,\alpha)$ and a simple closed curve $\beta$. Henceforth, we denote by $t\alpha$ the image $\lambda_{t,\alpha}$ of $(t,\alpha)$.
The space $\ml_{g,m}$ contains the weighted simple closed curves $\mathcal{WS}=\{t\alpha\mid t\ge 0, \alpha\in \mathcal{S}\}$ as a dense subset.
A measured geodesic lamination $\lambda \in \ml_{g,m}$ is said to be \emph{filling}\index{filling}\index{measured geodesic lamination!filling} if $\lambda(\alpha) > 0$ for all $\alpha \in \mathcal{S}$. It is called \emph{minimal}\index{minimal}\index{measured geodesic lamination!minimal} if every half-leaf of the support of the corresponding measured geodesic lamination on $M_0$ is dense in its support.

Let $x=(M,f)\in \teich_{g,m}$. The map
\begin{equation}
\label{eq:identify_ML}
\ml(M)\ni \lambda\mapsto \left[
\mathcal{S}\ni \beta\mapsto I_M(\lambda,f(\beta)^*)
\right]\in \mathbb{R}_{+}^{\mathcal{S}}
\end{equation}
induces a homeomorphism from $\ml(M)$ onto $\ml_{g,m}-\{0\}$,
where $f(\beta)^*$ is the simple closed geodesic (and hence it is a measured geodesic lamination) on $M$ which is in the homotopy class $f(\beta)$. From the definition,
we have
$$
I_M(f(\alpha)^*,f(\beta)^*)=i(\alpha,\beta)
$$
for $\alpha$, $\beta\in \mathcal{S}$.
The space $\ml_{g,m}$ is called the \emph{space of measured geodesic laminations}\index{space of measured geodesic laminations}\index{measured geodesic lamination!space of measured geodesic laminations} on $\Sigma_{g,m}$.

The space $\ml_{g,m}$ admits a natural action of the positive numbers $\mathbb{R}_+$ by
$$
\mathbb{R}_+\times \ml_{g,m}\ni (t,\lambda)\mapsto t\lambda\in \ml_{g,m}.
$$
The quotient space $\pml_{g,m}=(\ml_{g,m}-\{0\})/\mathbb{R}_+$ is called the \emph{space of projective measure laminations}\index{space of projective measured geodesic laminations}\index{projective measured geodesic lamination!space of projective measured geodesic laminations} on $\Sigma_{g,m}$.

\begin{theorem}[Thurston]
\label{thm:Thurston_mf}
As a piecewise linear manifold, $\ml_{g,m}$ is isomorphic to $\mathbb{R}^{6g - 6 + 2m}$.
The space $\pml_{g,m}$ is homeomorphic to the sphere $\mathbb{S}^{6g-7+2m}$.
\end{theorem}

\subsubsection{Uniquely ergodic measured laminations}
\label{subsub:uel}
A measured geodesic lamination $\lambda \in \ml_{g,m}$ is said to be \emph{uniquely ergodic}\index{uniquely ergodic}\index{mesured geodesic lamination!uniquely ergodic} if $i(\lambda, \mu) = 0$ for some $\mu \in \ml_{g,m} \setminus \{0\}$ implies that $[\lambda] = [\mu]$ in $\pml_{g,m}$. Any uniquely ergodic measured lamination is automatically minimal; otherwise, it would contain a proper measured geodesic sublamination $\lambda'$ such that $i(\lambda', \lambda) = 0$ (cf.~\cite[Proposition 3]{MR1810534}). Furthermore, when $3g - 3 + m > 1$, every uniquely ergodic measured lamination is also filling. Otherwise, $[\lambda] \in \mathcal{S} \subset \pml_{g,m}$ and $i(\lambda, \beta) = 0$ for some $\beta \in \mathcal{S}$ with $\lambda \ne \beta$.

Moreover, for any uniquely ergodic measured lamination $\lambda \in \ml(M)$, the transverse measure on the support $|\lambda|$ is unique up to scaling. Otherwise, there would exist another measured lamination $\lambda'$ with $|\lambda'| = |\lambda|$ and $[\lambda'] \ne [\lambda]$. However, since the supports coincide, we would have $i(\lambda', \lambda) = 0$, contradicting the unique ergodicity of $\lambda$ (see, e.g., \cite[Theorem 1.12]{MR662738}).

We denote by $\ml^{\mathrm{ue}}_{g,m}$ and $\pml^{\mathrm{ue}}_{g,m}$ the sets of uniquely ergodic (projective) measured laminations in $\ml_{g,m}$ and $\pml_{g,m}$, respectively.

\begin{remark}
The definition of unique ergodicity used here is not the standard one for uniquely ergodic measured geodesic laminations. In the usual definition, a measured geodesic lamination $\lambda$ is said to be uniquely ergodic if its support $|\lambda|$ admits a unique transverse measure, up to scalar multiplication. Under this definition, a uniquely ergodic measured foliation is not necessarily filling; see, e.g., \cite[\S4]{MR576605}.
\end{remark}

\subsection{Action of mapping class group}
\label{subsec:Action_MCG_on_ML}
It is natural to define the action of $[\omega] \in \operatorname{MCG}(\Sigma_{g,m})$ on the set of simple closed curves $\mathcal{S}$ by declaring that, for $\alpha \in \mathcal{S}$, $[\omega](\alpha)$ is the homotopy class of $\omega(\alpha')$, where $\alpha'$ is a representative curve of the class $\alpha$.
A mapping class $[\omega] \in \operatorname{MCG}(\Sigma_{g,m})$ acts on $\ml_{g,m}$ by
\begin{equation}
\label{eq:actionMCG_on_ML}
[\omega](\lambda) = \lambda \circ [\omega]^{-1}.
\end{equation}
This action is a natural extension of the action on weighted simple closed curves. Indeed,
\begin{align*}
[\omega](t\alpha)(\beta)
&= (t[\omega](\alpha))(\beta) = t\, i([\omega](\alpha), \beta) \\
&= t\, i(\alpha, [\omega]^{-1}(\beta)) = (t\alpha)([\omega]^{-1}(\beta)),
\end{align*}
for $\beta \in \mathcal{S}$, regarded as an element of $\ml_{g,m}$.  
Since the set of weighted simple closed curves is dense in $\ml_{g,m}$, this formula extends to the entire space $\ml_{g,m}$. Moreover, since $[\omega](t\lambda) = t[\omega](\lambda)$ for $\lambda \in \ml_{g,m}$ and $t \ge 0$, the action in \eqref{eq:actionMCG_on_ML} naturally descends to an action on $\pml_{g,m}$ given by
\begin{equation}
\label{eq:actionMCG_on_PML}
[\omega]([\lambda]) = \big[[\omega](\lambda)\big],
\end{equation}
for $[\lambda] \in \pml_{g,m}$ and $[\omega] \in \operatorname{MCG}(\Sigma_{g,m})$.

\subsection{Thurston measure}
In this section, we recall a canonical ${\rm MCG}(\Sigma_{g,m})$-invariant measure on $\ml_{g,m}$, called the \emph{Thurston measure} on $\ml_{g,m}$.

\subsubsection{Train tracks}
We begin by recalling the notion of a \emph{train track} in order to define $PL$-local charts on $\ml_{g,m}$. For references, see \cite{MR1144770}.

A \emph{train track}\index{train track} $\tau$ on a smooth surface $S$ is a $CW$-complex on $S$ with the following properties:
\begin{itemize}
\item
$\tau$ is $C^1$ away from its switch (vertex) and has tangent vectors at every point; and
\item
for each component of $R$ of $S\setminus \tau$, the double of $R$ along the interior edges of $\partial R$ has negative Euler characteristic.
\end{itemize}
The switches of a train track are the points where three or more smooth arcs come together.
The edges of a train track are called \emph{branches}\index{branch}\index{train track!branch} which connect switches, and each branch is a smooth path with a well-defined tangent vector. The inward-pointing tangent of an edge divides
the branches that are incident to a vertex into incoming and outgoing branches.
Let $W(\tau)$ be a real vector space consisting of real-valued functions on the set of branches of $\tau$ which satisfy the \emph{switch condition} that at each switch the sum of the values of the incoming branches is equal to the sum of the values of the outgoing branches. Let $V(\tau)\subset  W(\tau)$ be the cone consisting of non-negative functions. Any function in $V(\tau)$ is called a \emph{transverse measure}\index{transverse measure}\index{train track!transverse measure} on $\tau$. A train track is called \emph{generic} if all switches are at most trivalent. A train track is said to be \emph{maximal}\index{maximal}\index{train track!maximal} if each component of its complement is either a trigon, or a once punctured monogon. For simplicity, we always assume each train track is generic.

A train track $\tau$ is said to be \emph{recurrent}\index{recurrent}\index{train track!recurrent} if there exists $\mu \in V(\tau)$ with $\mu(b) > 0$ for all branches $b$ of $\tau$. A train track $\tau$ is called \emph{transversely recurrent}\index{transverse recurrent}\index{train track!transverse recurrent} if, for each branch $b$ of $\tau$, there exists a multicurve $C_b$ that hits $\tau$ efficiently, meaning that $C_b$ intersects $b$ at least once and contains no arc connecting to $\tau$ that is homotopic to a smooth arc in the boundary of a component of $S \setminus \tau$ (cf. \cite[p.23, Remark]{MR1144770}). A train track that is both recurrent and transversely recurrent is called \emph{birecurrent}\index{birecurrent}\index{train track!birecurrent}.

Let $\tau$ be a train track on a hyperbolic surface $M$ of finite area. A measured geodesic lamination $\lambda$ is said to be \emph{carried} by $\tau$, denoted by $\lambda\prec \tau$,  if there exists a differentiable map $f\colon M \to M$, homotopic to the identity, that takes the support $|\lambda|$ to $\tau$, such that the restriction of $df$ to any tangent line of $|\lambda|$ is non-singular. When a measured geodesic lamination $\lambda$ is carried by $\tau$, the carrying map defines a counting measure $\lambda(b)$ --- that is, the total mass of $\lambda$ on a transverse arc to $b$ --- for each branch $b$ of $\tau$, and the transverse measure $\lambda$ induces a transverse measure on $\tau$, which is also denoted by $\lambda$ for simplicity. Any measured geodesic lamination is carried by a birecurrent train track (cf. \cite[Corollary 1.7.6]{MR1144770}). Any birecurrent train track is a subtrack of a birecurrent and maximal train track (cf. \cite[Theorem 1.3.6]{MR1144770}).

Fix $x_0 = (M_0, f_0) \in \teich_{g,m}$. By the Koebe uniformization theorem, when $2g - 2 + m > 0$, the surface $M_0$ admits a unique hyperbolic metric of finite area that is compatible with its conformal structure (cf.\ \cite[III, 11G]{MR0114911}).

Suppose that the marking $f_0$ is a diffeomorphism. We identify $\ml_{g,m}$ with $\ml(M_0)$ via $f_0$, as in \eqref{eq:identify_ML}.
For a train track $\tau$ on $\Sigma_{g,m}$, let $\ml_{g,m}(\tau) \subset \ml_{g,m}$ denote the set of measured laminations $\lambda \in \ml(M_0) \cong \ml_{g,m}$ that are carried by $f_0(\tau)$. Let $\ml^\circ_{g,m}(\tau)$ be the subset of $\ml_{g,m}(\tau)$ consisting of those $\mu$ such that $\mu(b) > 0$ for every branch $b$ of $\tau$.
When $\tau$ is recurrent, any transverse measure $\lambda \in V(\tau)$ defines a measured lamination $\lambda \in \ml_{g,m}(\tau)$, and the resulting map $V(\tau) \to \ml_{g,m}(\tau)$ is injective and continuous (cf.\ \cite[Theorem 2.7.4]{MR1144770}).
Moreover, when the train track $\tau$ is birecurrent and maximal, the set $\ml^\circ_{g,m}(\tau)$ is open in $\ml_{g,m}$ (cf.\ \cite[Lemma 3.1.2]{MR1144770}).

\subsubsection{Thurston's symplectic form and Thurston's measure}
Let $\tau$ be a generic train track.  
For $u_1, u_2 \in W(\tau)$, the symplectic pairing is defined by
\[
\omega_{\mathrm{Th}}(u_1, u_2) = \frac{1}{2} \sum_v \left( u_1(e_1) u_2(e_2) - u_1(e_2) u_2(e_1) \right),
\]
where the sum is taken over all switches $v$ of the train track, and $e_1$ and $e_2$ are the two incoming branches at $v$, as shown in Figure~\ref{fig:train_switch} (cf.\ \cite{MR850748}).
\begin{figure}
\centering
\includegraphics[width=5cm, bb=0 0 289 131]{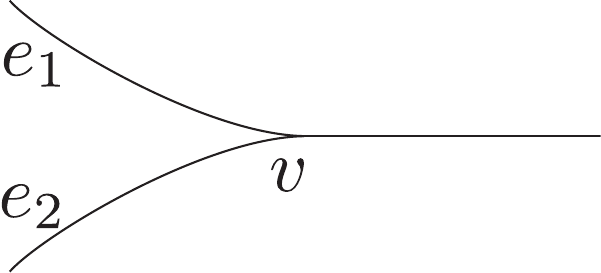}
\caption{Switches at a generic train track}
\label{fig:train_switch}
\end{figure}
This expression defines a skew-symmetric bilinear form on $W(\tau)$ (\cite[Lemma 3.2.1]{MR1144770}).  
If $\tau$ is maximal, the bilinear form $\omega_{\mathrm{Th}}$ is non-degenerate (\cite[Theorem 3.2.4]{MR1144770}).  
Hence, the skew-symmetric pairing $\omega_{\mathrm{Th}}$ defines a well-defined symplectic form on the piecewise linear manifold $\ml_{g,m}$. This symplectic form was first introduced by Thurston (unpublished) and is called the \emph{Thurston symplectic form}\index{Thurston symplectic form} on $\ml_{g,m}$.

The \emph{Thurston volume form}, or the \emph{Thurston measure}\index{Thurston measure}, $\ThursM$ on $\ml_{g,m}$ is the volume form defined by
\begin{equation}
\label{eq:Thurston_volume}
\ThursM = \frac{1}{(3g - 3 + m)!} \omega_{\mathrm{Th}}^{3g - 3 + m}.
\end{equation}
The Thurston measure $\ThursM$ is a locally finite, $\operatorname{MCG}(\Sigma_{g,m})$-invariant measure on $\ml_{g,m}$.  
Moreover, it is characterized as the unique (up to scaling) locally finite, $\operatorname{MCG}(\Sigma_{g,m})$-invariant, ergodic measure supported on the set of filling measured geodesic laminations (cf.\ \cite[Proposition 5.5, Theorem 5.7]{MR2495764} and \cite[Theorem 7.1]{MR2424174}; see also \cite{MR787893} and \cite{MR662738}).
Masur \cite{MR644018} showed that the set $\ml^{\mathrm{ue}}_{g,m}$ of uniquely ergodic measured foliations is of full measure with respect to the Thurston measure.

\subsection{Thurston compactification}

In this section, we recall a compactification of the Teichm\"uller space $\teich_{g,m}$ introduced by W.~Thurston, which describes degenerations of complete hyperbolic structures on $\Sigma_{g,m}$ of finite area.

\subsubsection{Thurston compactification}

Let $x = (M, f) \in \teich_{g,m}$. For $\alpha \in \mathcal{S}$, let $\ell_x(\alpha)$ denote the hyperbolic length of the closed geodesic on $M$ in the homotopy class $f(\alpha)$ (with respect to the hyperbolic metric compatible with the complex structure of $M$).  
Then the map
\begin{equation}
\label{eq:Thurston_compactification}
\teich_{g,m} \ni x \mapsto \big[\mathcal{S} \ni \alpha \mapsto \ell_x(\alpha)\big] \in P\mathbb{R}_{\ge 0}^\mathcal{S} = \big(\mathbb{R}_{\ge 0}^\mathcal{S} \setminus \{0\}\big) / \mathbb{R}_+
\end{equation}
is an embedding, and the closure of its image (with respect to the topology of pointwise convergence) is compact.  
This closure is called the \emph{Thurston compactification}\index{Thurston compactification} of $\teich_{g,m}$, and the boundary
\[
\partial_{Th} \teich_{g,m} := \overline{\teich_{g,m}}^{Th} \setminus \teich_{g,m}
\]
is called the \emph{Thurston boundary}\index{Thurston boundary} (cf.~\cite[Expos\'e 8]{MR568308}).  
Note that the Thurston compactification is defined within the same ambient space as the space $\pml_{g,m}$.

\begin{theorem}[Thurston]
\label{thm:thurston_compactification}
The Thurston compactification is homeomorphic to a closed ball $\mathbb{B}^{6g-6+2m}$ of dimension $6g - 6 + 2m$.  
Moreover, the image of the map \eqref{eq:Thurston_compactification} is disjoint from $\pml_{g,m}$, and the Thurston boundary $\partial_{Th} \teich_{g,m}$ coincides with $\pml_{g,m} \cong \mathbb{S}^{6g-7+2m} = \partial \mathbb{B}^{6g-6+2m}$.
\end{theorem}

\subsubsection{Action of the mapping class group}
\label{subsubsec:action_mcg_Tcomp}
The action \eqref{eq:mcg_action_T} of the mapping class group on $\teich_{g,m}$ extends continuously to the Thurston compactification.  
Take a sequence $\{x_n\}_{n=1}^\infty = \{(M_n, f_n)\}_{n=1}^\infty \subset \teich_{g,m}$ that converges to $[\lambda] \in \partial_{Th} \teich_{g,m} = \pml_{g,m}$.  
By definition, there exists a sequence of positive numbers $\{t_n\}_{n=1}^\infty$ such that
\[
t_n \ell_{x_n}(\alpha) = \ell_{x_n}(f_n(\alpha)) \to \lambda(\alpha)
\]
for all $\alpha \in \mathcal{S}$ as $n \to \infty$.
Therefore, for $[\omega] \in \operatorname{MCG}(\Sigma_{g,m})$, we have
\begin{align*}
t_n \ell_{[\omega](x_n)}(\alpha)
&= t_n \ell_{(M_n, f_n \circ \omega^{-1})}(\alpha) \\
&= t_n \ell_{M_n}(f_n \circ [\omega]^{-1}(\alpha)) \\
&= t_n \ell_{x_n}([\omega]^{-1}(\alpha)) \to \lambda([\omega]^{-1}(\alpha)) = [\omega](\lambda)(\alpha)
\end{align*}
for all $\alpha \in \mathcal{S}$, which coincides with the action on $\pml_{g,m}$ defined in \eqref{eq:actionMCG_on_PML}.
Furthermore, the action of each mapping class on the Thurston compactification is a homeomorphism (cf.\ \cite[Th\'eor\`eme III.3]{MR568308}).

\section{Background from Geometry on Teichm\"uller spaces}
\label{sec:Back_Geom_T-space}
In the previous sections, we introduced two compactifications of the Teichm\"uller space $\teich_{g,m}$: the Bers compactification, from a complex-analytic viewpoint, and the Thurston compactification, from a topological viewpoint.
We recall the definition and properties of extremal length, and discuss the geometry of these two compactifications.

\subsection{Hubard--Masur theorem and Extremal length}
Let $x = (M, f) \in \teich_{g,m}$.  
The \emph{vertical foliation}\index{vertical foliation} $v(q)$ associated with a quadratic differential $q = q(z)\,dz^2 \in \mathcal{Q}_x$ is defined as a measured geodesic lamination by the formula
\begin{equation}
\label{eq:vertical_foliation}
v(q)(\alpha) = \inf_{\alpha' \in f(\alpha)} \int_{\alpha'} \left| \operatorname{Re} \left( \sqrt{q(z)}\,dz \right) \right|
\end{equation}
for every $\alpha \in \mathcal{S}$.  
At first glance, it is not evident that the function $v(q)$ on $\mathcal{S}$, defined in terms of a quadratic differential $q \in \mathcal{Q}_x$, indeed yields a measured geodesic lamination.  
This fact is justified by a canonical correspondence between measured foliations and measured geodesic laminations, which underlies the terminology referring to $v(q)$ as the vertical \emph{foliation}.  
For further details, the reader is referred to \cite{MR568308}.

Hubbard and Masur \cite{MR523212} showed that for any $\lambda \in \ml_{g,m}$, there exists a unique holomorphic quadratic differential $q_{\lambda,x} \in \mathcal{Q}_x$ such that $\lambda = v(q_{\lambda,x})$. See Gardiner's paper \cite{MR736212} for the case of analytically finite surfaces.  
The holomorphic quadratic differential $q_{\lambda,x}$ is called the \emph{Hubbard--Masur differential}\index{Hubbard--Masur differential} for $\lambda \in \ml_{g,m}$ on $x = (M, f) \in \teich_{g,m}$.

The \emph{extremal length}\index{extremal length} of $\lambda \in \ml_{g,m}$ on $x = (M, f)$ is defined by
\[
\ext_x(\lambda) = \ext_\lambda(x) = \|q_{\lambda,x}\|,
\]
where the second expression emphasizes the extremal length as a function on $\teich_{g,m}$. When $\lambda \in \mathcal{S} \subset \ml_{g,m}$, the differential $q_{\lambda,x}$ is the Jenkins--Strebel differential corresponding to the simple closed curve $f(\lambda)$ on $M$, and the above definition of extremal length agrees with the classical one. Namely, for $\alpha \in \mathcal{S}$, the quantity $\ext_x(\alpha)$ coincides with the extremal length (in the sense of Beurling--Ahlfors \cite{MR36841}) of the family of rectifiable simple closed curves in the homotopy class $f(\alpha)$ on $M$.

The extremal length is a \emph{quasiconformal invariant}, meaning that
\begin{equation}
\label{eq:K-qc-extremal_length}
e^{-2d_T(x_1,x_2)} \ext_{x_1}(\lambda) \le
\ext_{x_2}(\lambda)
\le e^{2d_T(x_1,x_2)} \ext_{x_1}(\lambda)
\end{equation}
for $x_1, x_2 \in \teich_{g,m}$ and $\lambda \in \ml_{g,m}$ (cf. \cite[Chapter I, D]{MR2241787}).  
Furthermore, from the unique extremality discussed in \S\ref{subsec:teichmuller_theorem}, the inequality \eqref{eq:K-qc-extremal_length} is sharp in the sense that $\ext_{x_2}(\lambda)
= e^{2d_T(x_1,x_2)} \ext_{x_1}(\lambda)$ if and only if $x_2$ lies on the Teichm\"uller geodesic ray emanating from $x_1$ defined by $q_{\lambda,x_1}$.

\subsection{Gardiner's formula}
\label{subsec:Gardiner-formula-revisited}
Let $x_0=(M_0,f_0)\in \teich_{g,m}$.
Let $\mu\in L^\infty_{(-1,1)}(M_0)$.
Take $\epsilon>0$ with $\epsilon<1/\|\mu\|_\infty$, and let $g_t\colon M_0\to M_t$ be the quasiconformal map with complex coefficient $t\mu$ for $t\in \mathbb{R}$ with $|t|<\epsilon$.
Set $x_t=(M_t,g_t\circ f_0)\in \teich_{g,m}$.
Then, $\{x_t\}_{|t|<\epsilon}$ is a differentible path in $\teich_{g,m}$. Let $v\in T_{x_0}(\teich_{g,m})^{\mathbb{R}}$ be the real tangent vector of this path at $t=0$.

Let $\lambda\in \ml_{g,m}-\{0\}$.
In \cite{MR736212}, Gardiner showed that
the extremal length function $\ext_\lambda$ on $\teich_{g,m}$ for $\lambda$ is of class $C^1$ and
$$
d\,\ext_\lambda[v]=
\left.\dfrac{d}{dt}\ext_\lambda(x_t)
\right|_{t=0}
=-2{\rm Re}\iint_{M}\mu q_{\lambda,x}
=\llangle\mu,-q_{\lambda,x}\rrangle+\llangle\overline{\mu},-\overline{q_{\lambda,x}}\rrangle^-.
$$
The presence of the minus sign in the third term is due to our convention that $\lambda$ is the vertical foliation of $q_{\lambda,x_0}$.

From the representation \eqref{eq:holomorphic_tangent_to_real}, it is natural to associate the real tangent vector directed by $\mu$ with  
$$
v=[\mu]\oplus [\overline{\mu}]\in T_{x_0}(\teich_{g,m})^{\mathbb{R}}\subset T_{x_0}(\teich_{g,m})\oplus T'_{x_0}(\teich_{g,m})=T_{x_0}(\teich_{g,m})^{\mathbb{C}}.
$$
Hence, under the notation in \S\ref{subsubsec:models_tangent}, the total differential $d\,\ext_\lambda$ of the extremal length function of $\lambda\in \ml_{g,m}-\{0\}$ is represented by
$$
d\ext_\lambda|_{x_0}=-q_{\lambda,x_0}\oplus \overline{q_{\lambda,x_0}}\in T_{x_0}^*(\teich_{g,m})^{\mathbb{C}}=\mathcal{Q}_{x_0}\oplus \overline{\mathcal{Q}}_{x_0}
$$
for $x\in \teich_{g,m}$. From \eqref{eq:decomposition_complex_tangent3}, we have
\begin{equation}
\label{eq:Gardiner-formula}
\begin{cases}
&\partial\ext_\lambda|_{x_0}
=-q_{\lambda,x_0}\in \mathcal{Q}_{x_0}=\mathfrak{C}^{1,0}_{x_0}(\teich_{g,m})
\\
&\overline{\partial}\ext_\lambda|_{x_0}
=-\overline{q_{\lambda,x_0}}\in \overline{\mathcal{Q}}_{x_0}
=\mathfrak{C}^{0,1}_{x_0}(\teich_{g,m}).
\end{cases}
\end{equation}

\subsection{The Bers embedding from the theory of Kleinian groups}
Let $x_0 = (M_0, f_0) \in \teich_{g,m}$, and let $\Gamma_0$ be the Fuchsian group acting on $\mathbb{D}$ such that $M_0 = \mathbb{D} / \Gamma_0$.

\subsubsection{From Function Theory}

Let $\varphi \in \Bers{x_0}$.  
Choose $\mu \in L^\infty(M_0)$ corresponding to $\varphi$ via the Bers projection \eqref{eq:Bers_projection}, and let $\tilde{\mu} \in L^\infty(\hat{\mathbb{C}})$ denote its extension, also as defined in \eqref{eq:Bers_projection}.
Let $W_\varphi$ be the restriction to $\mathbb{D}^*$ of the normalized solution $\normqc{\tilde{\mu}}$ of the Beltrami equation with coefficient $\tilde{\mu}$.  
We renormalize $W_\varphi$ so that $W_\varphi(z) = z + o(1)$ as $z \to \infty$.  
Then $W_\varphi$ is a univalent function on $\mathbb{D}^*$, and there exists a faithful discrete representation $\chi_\varphi \colon \Gamma_0 \to \psl_2(\mathbb{C})$ such that
\begin{equation}
\label{eq:conjugation_varphi}
\chi_\varphi(\gamma) \circ W_\varphi = W_\varphi \circ \gamma
\end{equation}
for all $\gamma \in \Gamma_0$.

Jørgensen's theorem asserts that the space of faithful discrete representations of $\Gamma_0$ into $\psl_2(\mathbb{C})$ is closed in the topology of algebraic convergence (cf.\ \cite{MR427627}).  
Due to the normalization, the family $\{W_\varphi \mid \varphi \in \Bers{x_0}\}$ forms a normal family.  
Therefore, any point $\varphi$ in the Bers compactification $\overline{\Bers{x_0}}$ defines a univalent function $W_\varphi$ on $\hat{\mathbb{C}}$ and a faithful discrete representation $\chi_\varphi \colon \Gamma_0 \to \psl_2(\mathbb{C})$ satisfying \eqref{eq:conjugation_varphi}.
The image $\Gamma_\varphi := \chi_\varphi(\Gamma_0)$ for $\varphi \in \overline{\Bers{x_0}}$ is a $b$-group whose invariant component is $W_\varphi(\mathbb{D}^*)$, such that the quotient surface $W_\varphi(\mathbb{D}^*) / \Gamma_\varphi$ is biholomorphically equivalent to the mirror image of $M_0$ (cf.\ \cite[\S5, \S8]{MR0297992}).  
The group $\Gamma_\varphi$ is called a \emph{boundary group}\index{boundary group}\index{boundary group!Kleinian group} if $\varphi \in \partial \Bers{x_0}$.

\subsubsection{From the theory of hyperbolic $3$-manifolds}

The Riemann sphere $\hat{\mathbb{C}}$ is identified with the boundary at infinity of hyperbolic $3$-space $\mathbb{H}^3$, and the action of M\"obius transformations extends isometrically to $\mathbb{H}^3$ via the Poincar\'e extension (see, e.g., \cite{MR1219310}). Hence, for each $\varphi \in \overline{\Bers{x_0}}$, the quotient manifold $N_\varphi = \mathbb{H}^3 / \Gamma_\varphi$ is a complete hyperbolic $3$-manifold.

Bonahon's tameness theorem \cite{MR847953} asserts that $N_\varphi$ is homeomorphic to $\Sigma_{g,m} \times \mathbb{R}$ via a homeomorphism that is compatible with the representation $\chi_\varphi \colon \Gamma_0\ (\cong \pi_1(\Sigma_{g,m})) \to \Gamma_\varphi \cong \pi_1(N_\varphi)$, and such that each peripheral curve on $\Sigma_{g,m}$ corresponds to a rank-one cusp in $N_\varphi$.

Henceforth, we assume that $\varphi$ corresponds to a \emph{totally degenerate group}\index{totally degenerate group without accidental parabolics}\index{boundary group!totally degenerate group without accidental parabolics} without accidental parabolics; that is, there exists no hyperbolic element $\gamma \in \Gamma_0$ such that $\chi_\varphi(\gamma)$ is parabolic (cf.~\cite[\S8, \S9]{MR0297992}). Otherwise, $\varphi$ is called a \emph{cusp}\index{cusp}\index{boundary group!cusp}. The set of cusps in $\partial \Bers{x_0}$ is contained in a countable union of complex analytic subvarieties of $A_2(\mathbb{D}^*, \Gamma_0)$. Therefore, ``most'' points $\varphi$ (in the sense of dimension or Baire category) correspond to totally degenerate groups without accidental parabolics (cf.~\cite[Theorem 14]{MR0297992}).

In this case, the hyperbolic $3$-manifold $N_\varphi$ has two ends, denoted $e_+$ and $e_-$, corresponding respectively to $\Sigma_{g,m} \times (0, \infty)$ and $\Sigma_{g,m} \times (-\infty, 0)$. An end $e_s$ (for $s = \pm$) is said to be \emph{geometrically infinite}\index{geometrically infinite}\index{Kleinian group!geometrically infinite} if there exists a sequence of simple closed curves $\{\alpha_n\}_{n \in \mathbb{N}}$ on $\Sigma_{g,m}$ such that their geodesic representatives $\alpha_n^*$ in $N_\varphi$ exit the end $e_s$; otherwise, $e_s$ is called \emph{geometrically finite}\index{geometrically finite}\index{Kleinian group!geometrically finite}.

By construction, the negative end $e_-$ is always geometrically finite for all $\varphi \in \overline{\Bers{x_0}}$. When the positive end $e_+$ is also geometrically finite, the corresponding point $\varphi$ lies in the interior $\Bers{x_0}$ of the Bers compactification, and the group $\Gamma_\varphi$ is a quasi-Fuchsian group uniformizing the mirror of $x_0 = (M_0, f_0)$ and a marked Riemann surface corresponding to $\varphi$ via the Bers embedding (cf.~\cite{MR0111834}).

Otherwise, if $e_+$ is geometrically infinite, then $\varphi \in \partial \Bers{x_0}$, and any sequence of simple closed curves $\{\alpha_n\}_{n \in \mathbb{N}}$ on $\Sigma_{g,m}$ whose geodesic representatives exit $e_+$ converges (after passing to a subsequence) to a projective class $[\lambda]$ of a minimal and filling measured geodesic lamination $\lambda$ in $\pml_{g,m}$. Although the projective class $[\lambda]$ may not be uniquely determined, its support $|\lambda|$ is independent of the choice of such a sequence. This support $|\lambda|$ is called the \emph{ending lamination}\index{ending lamination}\index{Kleinian group!ending lamination} of $N_\varphi$ (and hence of $\varphi$). See \cite[\S VI]{MR847953}, \cite[\S2.5]{MR2925381}, and \cite[Definition 9.3.6]{Thuston-LectureNote}.

The Thurston double limit theorem states that the support of any minimal and filling measured geodesic lamination is realized as the ending lamination of some $\varphi \in \partial \Bers{x_0}$ (cf.~\cite[Theorem 4.1]{MR4556467}; see also \cite[Theorem 3.13]{MR1029395}).

The Ending Lamination Theorem, due to Brock--Canary--Minsky, asserts that the hyperbolic $3$-manifold $N_\varphi$ (marked by $\chi_\varphi$) is determined by its ending lamination (cf.~\cite{MR2925381}). Hence, there exists a map
\begin{equation}
\label{eq:map_PMLmf_to_BB}
\Xi_{x_0} \colon \pml_{g,m}^{\mathrm{mf}} \to \partial \Bers{x_0}
\end{equation}
that assigns to each $[\lambda] \in \pml_{g,m}^{\mathrm{mf}}$ a point $\varphi \in \partial \Bers{x_0}$ such that the ending lamination of $N_\varphi$ coincides with the support of $\lambda$,  
where $\pml_{g,m}^{\mathrm{mf}}$ denotes the set of projective classes of minimal and filling measured geodesic laminations on $\Sigma_{g,m}$.

The map $\Xi_{x_0}$ is continuous and closed---that is, it maps closed sets to closed sets (cf.~e.g., \cite[p.~190]{MR2258749}, \cite[Theorem 6.5]{MR2582104}, and \cite[Theorem 3.13]{MR1029395}).  
In particular, $\Xi_{x_0}$ is Borel measurable.
\subsection{Asymptotic behavior of Teichm\"uller geodesic rays}
Let $x_0 = (M_0, f_0) \in \teich_{g,m}$. We always identify $\ml_{g,m} \setminus \{0\}$ with $\ml(M_0)$ via \eqref{eq:identify_ML}.

With the notation of \eqref{eq:Teichmuller-ray}, for $[\lambda] \in \pml_{g,m}$ and $x = (M, f) \in \teich_{g,m}$, we refer to $\TRay{x;[\lambda]} = \TRay{q_{\lambda,x}}$ as the \emph{Teichm\"uller geodesic ray}\index{Teichm\"uller geodesic ray} emanating from $x$ in the direction $[\lambda]$.

From the Teichm\"uller theorem discussed in \S\ref{subsec:teichmuller_theorem}, the map
$$
\pml_{g,m} \times (0, \infty) \ni ([\lambda], t) \mapsto \TRay{x_0;[\lambda]}(t) \in \teich_{g,m} \setminus \{x_0\}
$$
is a homeomorphism. By the Ending Lamination Theorem, we obtain the following:

\begin{proposition}
\label{prop:limit_T-ray}
Let $x_0 \in \teich_{g,m}$.  
For any $x \in \teich_{g,m}$ and $[\lambda] \in \pml_{g,m}^{\mathrm{mf}}$, the Teichm\"uller ray $\mathfrak{B}_{x_0} \circ \TRay{x;[\lambda]}$ in $\Bers{x_0}$ converges to a totally degenerate group without accidental parabolics, whose ending lamination is equal to the support $|\lambda|$ of $\lambda$.
\end{proposition}
For a proof, see \cite[Proposition 5.1]{MR4633651}, for example.

The following two theorems are due to Masur:

\begin{theorem}[Masur {\cite[Theorem 2]{MR576605}}]
\label{thm:masur_geodesic}
Let $[\lambda] \in \pml^{\mathrm{ue}}_{g,m}$ and $x \in \teich_{g,m}$. Then the Teichm\"uller geodesic ray $\TRay{x;[\lambda]}$ converges to $[\lambda] \in \partial_{Th} \teich_{g,m} \cong \pml_{g,m}$ in the Thurston compactification.  
Moreover, for any $x_1, x_2 \in \teich_{g,m}$, the Teichm\"uller geodesic rays $\TRay{x_1;[\lambda]}$ and $\TRay{x_2;[\lambda]}$ are asymptotic at infinity, in the sense that
$$
\lim_{t \to \infty} d_T\big( \TRay{x_1;[\lambda]}(t),\ \TRay{x_2;[\lambda]}((0,\infty)) \big) = \lim_{t \to \infty} d_T\big( \TRay{x_2;[\lambda]}(t),\ \TRay{x_1;[\lambda]}((0,\infty)) \big) = 0.
$$
\end{theorem}

\begin{theorem}[Masur {\cite[Theorem 1.1]{MR1167101}}]
\label{thm:divergence_non-ue}
If the direction of a Teichm\"uller geodesic ray is not uniquely ergodic, then its projection
to the moduli space eventually leaves every compact subset.
\end{theorem}

\subsection{Action of the mapping class group on the Bers slice}

As discussed in \S\ref{subsec:Action_MCG} and \S\ref{subsec:Action_MCG_on_ML}, the mapping class group $\mcg(\Sigma_{g,m})$ acts on both the Teichm\"uller space and the space of (projective) measured laminations. In \S\ref{subsubsec:action_mcg_Tcomp}, we saw that these two actions are compatible from the viewpoint of the Thurston compactification.

The case of the Bers slice is more delicate. For $x_1, x_2 \in \teich_{g,m}$, there exists a biholomorphism given by the change of basepoints:
\begin{equation}
\label{eq:two_Bers_slices}
\mathfrak{B}_{x_2} \circ \mathfrak{B}_{x_1}^{-1} \colon \Bers{x_1} \to \Bers{x_2}.
\end{equation}
The map \eqref{eq:two_Bers_slices} is referred to as the \emph{geometric isomorphism} in \cite{MR1322950}.

Kerckhoff and Thurston \cite{MR1037141} provided examples in the case $(g,m) = (2,0)$ showing that the biholomorphism \eqref{eq:two_Bers_slices} does not extend continuously to the Bers compactifications, and that the action of the mapping class group similarly fails to extend continuously.

Ohshika \cite{MR3330544} introduced the concept of the \emph{reduced Bers slice}\index{reduced Bers slice} as the quotient of the Bers slice $\Bers{x_0}$ by the equivalence relation of quasiconformal conjugacy on the Bers boundary. He showed that both geometric isomorphisms and the mapping class group action extend homeomorphically to the reduced Bers slice.

This result implies the following: if a boundary point $\varphi \in \partial \Bers{x_0}$ is \emph{quasiconformally rigid}\index{quasiconformal rigid}\index{Kleinian group!quasiconformal rigid}—in the sense that every quasiconformal self-conjugacy of $\chi_\varphi$ that is conformal on the invariant component is induced by a M\"obius transformation—then the action of any mapping class on $\Bers{x_0}$ extends continuously at $\varphi$. According to Sullivan's rigidity theorem \cite{MR624833}, this quasiconformal rigidity is equivalent to the condition that the stabilizer subgroup of each component of the region of discontinuity of $\Gamma_\varphi$, except for the invariant component, is a triangle group.

In particular, the map \eqref{eq:map_PMLmf_to_BB} commutes with the action of the mapping class group:
\begin{equation}
\label{eq:map_PMLmf_to_BB_MCG}
\Xi_{x_0} \circ [\omega] = [\omega] \circ \Xi_{x_0}
\end{equation}
for $[\omega] \in \operatorname{MCG}(\Sigma_{g,m})$. This commutativity describes an equivariant correspondence between directions in $\pml_{g,m}$ and points on the Bers boundary, as determined by the limiting behavior discussed in \Cref{prop:limit_T-ray}.

While Bers \cite{MR624803} had already observed the continuous extendability on the set of quasiconformally rigid groups, Ohshika's breakthrough result provides a complete and conceptually satisfying resolution to the long-standing problem of extending basepoint changes and the mapping class group action to the Bers compactification.

\section{Toy model: Teichm\"uller space of tori}
\label{sec:Toy}
In this section, we discuss the Teichm\"uller space of tori. Although this space is simply the hyperbolic plane, its properties raise motivating questions for the study of general Teichm\"uller spaces.

\subsection{Teichm\"uller space of tori}
Let $\Sigma_1$ be a (topological) torus.\index{torus}
Fix a symplectic basis $\{A, B\}$ of $H_1(\Sigma_1)$ such that the intersection number satisfies $A \cdot B = +1$ (cf. Figure~\ref{fig:Torus1}).

\begin{figure}
\includegraphics[width=5cm, bb = -1 0 384 221]{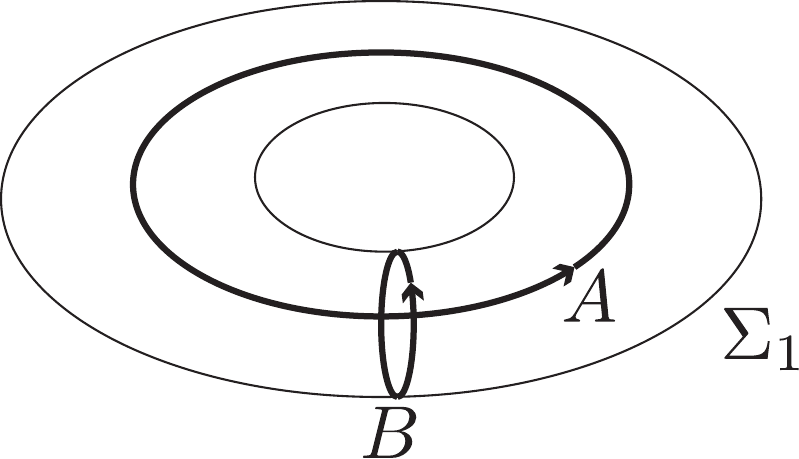}
\caption{Torus and symplectic basis}
\label{fig:Torus1}
\end{figure}

Let $\tau \in \mathbb{H}$, and let $M_\tau$ be the one-dimensional complex torus defined by the lattice $\mathbb{Z} \oplus \tau \mathbb{Z} = \{m + n\tau \in \mathbb{C} \mid m,n \in \mathbb{Z}\}$ in $\mathbb{C}$.
Let $h_\tau\colon \Sigma_1 \to M_\tau$ be an orientation-preserving homeomorphism that sends $A$ and $B$ to the cycles defined by $1$ and $\tau$ in $\mathbb{Z} \oplus \tau \mathbb{Z} \cong \pi_1(M_\tau)$, respectively.
Then, $(M_\tau, h_\tau)$ defines a marked torus.

Let $\teich_1$ denote the Teichm\"uller space of tori.\index{Teichm\"uller space of tori}
The \emph{period map}\index{period map}\index{Teichm\"uller space of tori!period map} is a function from $\teich_1$ to $\mathbb{H}$ defined as follows:
For $x = (M, f) \in \teich_1$, let $\omega_x$ be the holomorphic one-form on $M$ whose $f(A)$-period is normalized to one. Since $f$ is orientation-preserving, the imaginary part of the $f(B)$-period of $\omega_x$ is positive:
\[
\tau(x) = \int_{f(B)} \omega_x.
\]
Then the two marked tori $x = (M, f)$ and $(M_{\tau(x)}, h_{\tau(x)})$ are Teichm\"uller equivalent, and the period map $\teich_1 \ni x \mapsto \tau(x) \in \mathbb{H}$ is bijective.
Hence, this correspondence induces a complex structure on $\teich_1$, which coincides with the one arising from the quasiconformal deformation theory discussed in \S\ref{subsec:complex_structure_holomorphicinfinitesimal_spaces}.

The mapping class group $\operatorname{MCG}(\Sigma_1)$ is isomorphic to $\slm_2(\mathbb{Z})$. The action \eqref{eq:mcg_action_T} of $\operatorname{MCG}(\Sigma_1)$ induces a homomorphism
\[
\operatorname{MCG}(\Sigma_1) \cong \slm_2(\mathbb{Z}) \to \operatorname{Aut}(\mathbb{H}),
\]
whose image coincides with the modular group $\psl_2(\mathbb{Z})$. Therefore, the moduli space of tori is the modular surface $\mathcal{M}_1 = \mathbb{H} / \psl_2(\mathbb{Z})$.

The Teichm\"uller distance $d_T$ on $\teich_1$ coincides with the Kobayashi distance on $\mathbb{H}$, which in turn coincides with the Poincar\'e distance of curvature $-4$.
The Teichm\"uller geodesic rays are precisely the hyperbolic geodesic rays.

\subsection{Spaces in Thurston theory for tori}

By the Gauss--Bonnet theorem, any complex torus admits no hyperbolic structure.  
To define analogues of the spaces discussed in \S\ref{sec:Back_Thurston_theory} for $\Sigma_1$,  
we adopt the topological formulation presented in \S\ref{subsub:topological_ML}.  

To avoid confusion, we refer to the corresponding spaces as the \emph{space of measured foliations}, denoted by $\mathcal{MF}$, and the \emph{space of projective measured foliations}, denoted by $\mathcal{PMF}$,  
rather than as the space of measured geodesic laminations and projective measured geodesic laminations, respectively.

\subsubsection{Space of measured foliations}

The set $\mathcal{S} = \mathcal{S}(\Sigma_1)$ of homotopy classes of simple closed curves on $\Sigma_1$ is parametrized by the extended rational numbers $\hat{\mathbb{Q}} = \mathbb{Q} \cup \{\infty\}$ as follows.
We identify $\Sigma_1$ with the marked torus $M_{\imaginaryunit}$ for $\tau = \imaginaryunit$.  
Via the identification $\pi_1(\Sigma_1) \cong H_1(\Sigma_1) \cong \mathbb{Z} \oplus \imaginaryunit \mathbb{Z}$, any unoriented simple closed curve is represented by $C_{p/q} = -pA + qB$,  
where $p$ and $q$ are relatively prime integers with $q \ge 0$, and $q = 0$ is allowed only when $p = 1$.  
Then the map
\[
\mathcal{S} \ni -pA + qB \mapsto p/q \in \hat{\mathbb{Q}}
\]
is a bijection, where we define $1/0 = \infty$. We call the homotopy class of the unoriented curve represented by $C_{p/q}$ the \emph{$p/q$-curve}\index{$p/q$-curve}\index{torus!$p/q$-curve} on $\Sigma_1$.

The geometric intersection number between $C_{p/q}$ and $C_{r/s}$ is given by $i(C_{p/q}, C_{r/s}) = |ps - rq|$.

The $\pi$-rotation of $\mathbb{R}^2$ about the origin defines a $\mathbb{Z}_2 = \mathbb{Z}/2\mathbb{Z}$-action. We denote by $[a,b] \in \mathbb{R}^2/\mathbb{Z}_2$ the equivalence class of $(a,b) \in \mathbb{R}^2$.

The set of weighted simple closed curves $\mathcal{WS}$ on $\Sigma_1$ is identified with the subset of $\mathbb{R}^2/\mathbb{Z}_2$ given by
\[
\{[x,y] \in \mathbb{R}^2/\mathbb{Z}_2 \mid x/y \in \hat{\mathbb{Q}}\}.
\]
The map
\begin{equation}
\label{eq:MF_torus}
\mathcal{WS} \ni [x,y] \mapsto \left[ C_{p/q} \in \mathcal{S} \mapsto |py + xq| \right] \in \mathbb{R}_{\ge 0}^{\mathcal{S}}
\end{equation}
is injective, and the closure of its image coincides with the image of the canonical extension of the map to $\mathbb{R}^2/\mathbb{Z}_2$.

Therefore, the domain $\mathcal{MF} = \mathbb{R}^2/\mathbb{Z}_2$ is regarded as the \emph{space of measured foliations}\index{space of measured foliation}\index{torus!space of measured foliation} on $\Sigma_1$, which corresponds to the space of measured geodesic laminations in this case.
The map
\begin{equation}
\label{eq:PMF_torus}
\mathbb{R}^2/\mathbb{Z}_2 \ni [x,y] \mapsto -x/y \in \mathbb{R} \cup \{\infty\}
\end{equation}
descends to a homeomorphism $\mathcal{PMF} = (\mathcal{MF} \setminus \{0\})/\mathbb{R}_+ \to \mathbb{S}^1 = \mathbb{R} \cup \{\infty\}$. We call the space $\mathcal{PMF}$ the \emph{space of projective measured foliations}\index{space of projective measured foliation}\index{torus!space of projective measured foliation} on $\Sigma_1$.

\subsection{Thurston's symplectic form and Thurston's measure}

The space $\mathcal{MF}$ admits train track charts as in Figure~\ref{fig:Torus2}.
\begin{figure}
\centering
\includegraphics[width = 10cm, bb = 0 0 641 393]{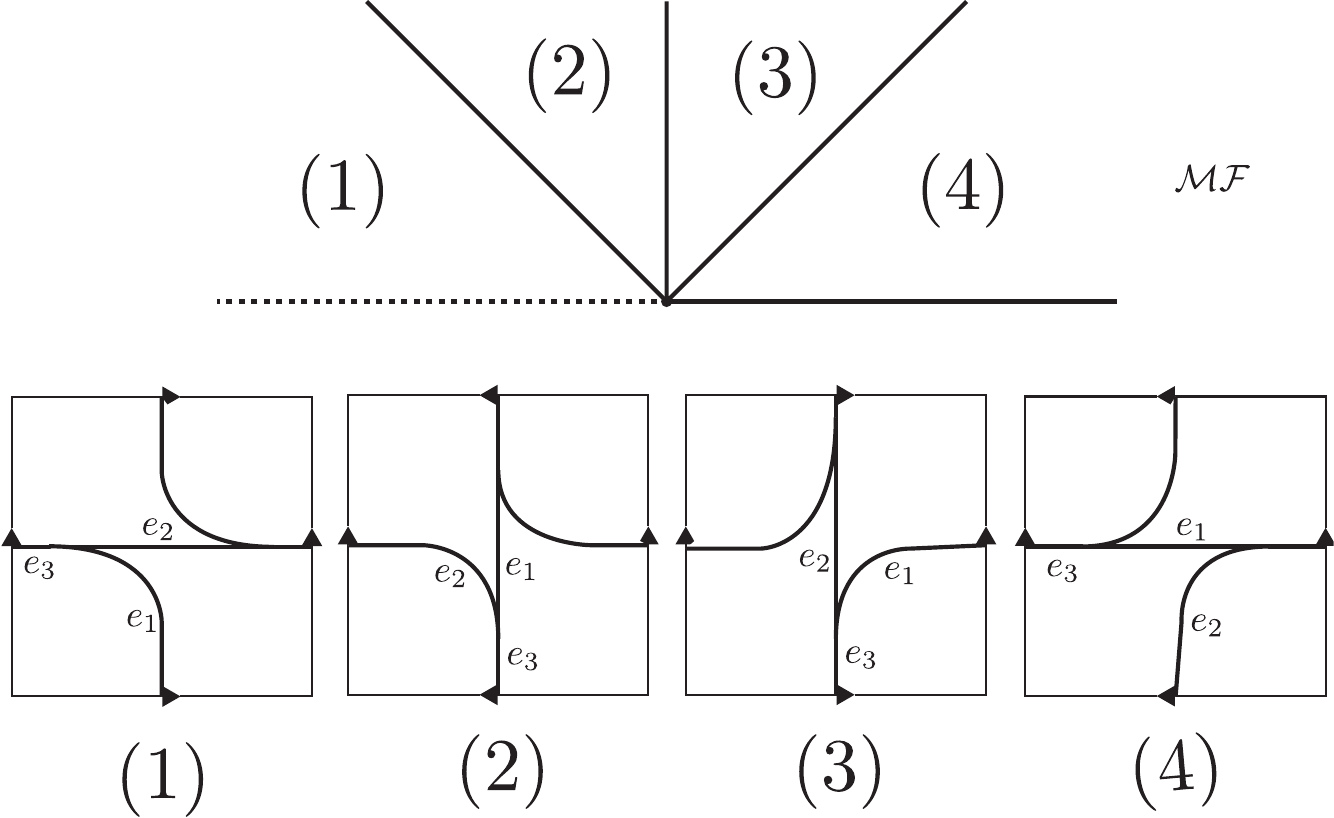}
\caption{Train track charts for $\mathcal{MF}$}
\label{fig:Torus2}
\end{figure}
The four squares in Figure~\ref{fig:Torus2} represent tori, obtained by first identifying the top and bottom edges of each square to form cylinders, and then identifying the left and right edges of each cylinder to obtain a torus. Each resulting torus contains a train track induced by the edge identifications.

However, it should be noted that these train tracks are not genuine in the strict sense, as their complements contain bigons, which are prohibited in the classical definition. A common convention is to consider once-punctured tori by removing the point corresponding to the corner, thereby turning the bigons into once-punctured bigons. Under this convention, these issues do not interfere with formal calculations (cf.~\cite[\S9.5]{Thuston-LectureNote}).

Let $\tau$ be one such train track. In each case, the measured foliation corresponding to a weight $\mu \in V(\tau)$ is given as follows:
\begin{itemize}
 \item[(1)] $V(\tau) \ni \mu \mapsto [x,y] = [-(\mu(e_1)+\mu(e_2)), \mu(e_1)] \in \mathcal{MF}$;
 \item[(2)] $V(\tau) \ni \mu \mapsto [x,y] = [-\mu(e_2), \mu(e_1)+\mu(e_2)] \in \mathcal{MF}$;
 \item[(3)] $V(\tau) \ni \mu \mapsto [x,y] = [\mu(e_1), \mu(e_1)+\mu(e_2)] \in \mathcal{MF}$;
 \item[(4)] $V(\tau) \ni \mu \mapsto [x,y] = [\mu(e_1)+\mu(e_2), \mu(e_2)] \in \mathcal{MF}$.
\end{itemize}
For instance, in Case (1), the chart is defined by
\[
x = -\mu(e_1) - \mu(e_2), \quad y = \mu(e_1),
\]
for $[x, y] \in \mathcal{MF}$. In this chart, we have $\mu(e_1) = y$ and $\mu(e_2) = -x - y$. Therefore, Thurston's symplectic form $\omega_{\mathrm{Th}}$ is given by
\begin{align*}
\omega_{\mathrm{Th}}(u_1, u_2)
&= \frac{1}{2}\big(u_1(e_1) u_2(e_2) - u_1(e_2) u_2(e_1)\big) \\
&\quad + \frac{1}{2}\big(u_1(e_1) u_2(e_2) - u_1(e_2) u_2(e_1)\big) \\
&= u_1(e_1) u_2(e_2) - u_1(e_2) u_2(e_1) \\
&= \frac{1}{2}dy \wedge (-dx - dy)(u_1, u_2) = \frac{1}{2}dx \wedge dy(u_1, u_2),
\end{align*}
for $u_1, u_2 \in W(\tau)$, after identifying $W(\tau)$ with the tangent cone to the PL-manifold structure on $\mathcal{MF}$. By performing the same computation in the other charts, we verify that the Thurston symplectic form coincides with one half of the standard Euclidean symplectic form, that is, $\omega_{\mathrm{Th}} = \frac{1}{2}dx \wedge dy$ on these charts.\footnote{The factor $1/2$ depends on the convention of the exterior derivative.}

These charts do not cover the diagonal, horizontal, and vertical rays in $\mathcal{MF}$. To cover these rays, we fix one of the train tracks $\tau \subset \Sigma_1$ and consider the action of a mapping class $[\omega]$ on $\Sigma_1$ such that the interior of $V(\omega(\tau))$ contains one of these rays. The identification described above is compatible with the action of the mapping class group. Mapping class group actions are represented by elements of $\slm_2(\mathbb{Z})$ acting linearly on $\mathcal{MF}$, which preserve the symplectic form described above.

As a conclusion, the Thurston symplectic form $\omega_{\mathrm{Th}}$ on $\mathcal{MF} = \mathbb{R}^2 / \mathbb{Z}_2$ coincides with one half of the Euclidean symplectic form, and the Thurston volume form is equal to one half of the Euclidean volume form.\footnote{As discussed above, this is a formal calculation mimicking the case of once-punctured tori. Hence, we assume $3g - 3 + m = 1$ in this case.}

\subsection{Extremal length}

Let $\tau \in \mathbb{H} \cong \teich_1$.
For $\lambda = [x, y] \in \mathcal{MF} \setminus \{0\} = \mathbb{R}^2 / \mathbb{Z}_2 \setminus \{0\}$, we define a Jenkins--Strebel differential $q_{\lambda,\tau}$ on $M_\tau = \mathbb{C} / (\mathbb{Z} \oplus \tau \mathbb{Z})$ by
\begin{equation}
\label{eq:JS_torus}
q_{\lambda,\tau} = -\left( \frac{x + y\overline{\tau}}{\operatorname{Im}(\tau)} \right)^2 dz^2.
\end{equation}
The infimum in the integral on the right-hand side of \eqref{eq:vertical_foliation} with respect to the $p/q$-curve $C_{p/q} = [-p, q] \in \mathcal{MF}$ is attained by the projection to $M_\tau$ of the Euclidean straight line $\tilde{C}_{p/q} \subset \mathbb{C}$ connecting the origin to $-p + q\tau$. Then,
\begin{align*}
\inf_{\alpha' \in C_{p/q}} \int_{\alpha'} \left| \operatorname{Re} \left( \sqrt{q_{\lambda,\tau}} \right) \right|
&= \int_{\tilde{C}_{p/q}} \left| \operatorname{Re} \left( i \left( \frac{x + y\overline{\tau}}{\operatorname{Im}(\tau)} \right) dz \right) \right| \\
&= \int_{\tilde{C}_{p/q}} \left| \operatorname{Im} \left( \left( \frac{x + y\overline{\tau}}{\operatorname{Im}(\tau)} \right) dz \right) \right| \\
&= |py + qx| = \lambda(C_{p/q}),
\end{align*}
for $C_{p/q} = [-p, q] \in \mathcal{MF}$ (cf.~\eqref{eq:MF_torus}). Therefore, $q_{\lambda,\tau}$ is the Hubbard--Masur differential for $\lambda$ on $M_\tau$.

As a result, the extremal length of $\lambda = [x, y] \in \mathcal{MF}$ on $M_\tau$ is given by
\begin{equation}
\label{eq:extremal_length_torus}
\ext_\tau(\lambda) = \| q_{\lambda,\tau} \| = \frac{|x + y\tau|^2}{\operatorname{Im}(\tau)}.
\end{equation}
Observe that as $\ext_{\tau}([x,y]) \to 0$ with varying $\tau \in \mathbb{H}$, the parameter $\tau$ tends to $-x/y \in \hat{\mathbb{R}}$. This behavior justifies the choice of sign in \eqref{eq:PMF_torus} and confirms the compatibility of this parametrization with Thurston's compactification of $\teich_1$.
%
%
\section{Function theory on Teichm\"uller spaces}
\label{sec:Function_theory_T}
In this section, we consider bounded (pluri)harmonic and holomorphic functions on the Teichm\"uller space.
We begin by recalling the case of the upper half-plane or the unit disk, which corresponds to the Teichm\"uller space of tori.
Then, we discuss recent progress in the author's research on function theory.

\subsection{Prototype example: Function theory on $\teich_1$}
The function theory on $\teich_1$ is equivalent to that on the upper half-plane $\mathbb{H}$, or, equivalently, on the unit disk $\mathbb{D}$.

\subsubsection{Function theory on $\mathbb{H}$}
We begin with the classical Poisson integral formula due to Fatou~\cite{MR1555035}.

\begin{theorem}[Fatou]
\label{thm:Fatou1}
Let $\mathbb{D}$ denote the unit disk\index{unit disk} in $\mathbb{C}$.
Any bounded harmonic function on $\mathbb{D}$ admits a non-tangential limit $u^*$ on $\partial \mathbb{D}$, and satisfies the Poisson integral formula
\begin{equation}
\label{eq:PIF}
u(z) = \int_0^{2\pi} u^*(e^{i\theta}) \frac{1 - |z|^2}{|z - e^{i\theta}|^2} \frac{d\theta}{2\pi},
\end{equation}
for all $z \in \mathbb{D}$.
\end{theorem}
See, e.g., \cite[Chapter IV, §2]{MR0114894}.
The (weighted) measure on $\partial \mathbb{D}$ given by
\begin{equation}
\label{eq:HM}
d\boldsymbol{\omega}_z = \frac{1 - |z|^2}{|z - e^{i\theta}|^2} \frac{d\theta}{2\pi}
\end{equation}
in the integral \eqref{eq:PIF} is called the \emph{harmonic measure}\index{harmonic measure}\index{unit disk!harmonic measure} on $\partial \mathbb{D}$ at $z \in \mathbb{D}$, and the factor $(1 - |z|^2)/|z - e^{i\theta}|^2$ is known as the \emph{Poisson kernel}\index{Poisson kernel}\index{unit disk!Poisson kernel} for the unit disk $\mathbb{D}$ (see, e.g., \cite{MR2450237}).

To express this result in terms of the upper half-plane $\mathbb{H}$,\index{upper half plane} we consider the push-forward of the measure via the conformal map $T \colon \mathbb{D} \to \mathbb{H}$ with $T(0) = \imaginaryunit$. We then obtain
\begin{equation}
\label{eq:HM_upper-half-plane}
d(T_*\boldsymbol{\omega}_{z}) = \frac{\operatorname{Im}(\tau)}{|\tau - t|^2} \frac{dt}{\pi}
= \frac{\operatorname{Im}(\tau)(1 + t^2)}{|\tau - t|^2} \, d(T_*\boldsymbol{\omega}_{0}),
\end{equation}
where $\tau = T(z)$ and $t = T(e^{i\theta})$, and
\begin{equation}
\label{eq:normalized_spherical_measure}
d(T_*\boldsymbol{\omega}_{0}) = \frac{dt}{\pi(1 + t^2)}
\end{equation}
is the normalized spherical measure on $\hat{\mathbb{R}} = \partial \mathbb{H}$.
The measure \eqref{eq:HM_upper-half-plane} is called the \emph{harmonic measure}\index{harmonic measure}\index{upper half plane!harmonic measure} at $\tau \in \mathbb{H}$, and the factor $\operatorname{Im}(\tau)/|\tau - t|^2$ in the middle expression is referred to as the \emph{Poisson kernel}\index{Poisson kernel}\index{upper half plane!Poisson kernel} (see \cite[p.~4]{MR2450237}).
In certain contexts, the Poisson kernel can be interpreted as the Radon--Nikodym derivative of one harmonic measure with respect to another, reflecting the absolute continuity between harmonic measures at different points. Under this interpretation, the factor appearing in the last expression of \eqref{eq:HM_upper-half-plane} is also regarded as a Poisson kernel (cf.~\cite{MR881709}).

When $z = 0 \in \mathbb{D}$, the harmonic measure \eqref{eq:HM} coincides with the normalized angular measure on $\partial \mathbb{D}$. By conformal invariance, the harmonic measure \eqref{eq:HM_upper-half-plane} corresponds to the visual angle measure seen from $\tau \in \mathbb{H}$ (cf.~\cite[\S1.5]{MR279280}).

\subsubsection{Function theory on $\teich_1$}
\label{subsubsec:function_theory_T1}

We interpret the observation in the previous section in the context of Teichm\"uller theory.  
Note that the Teichm\"uller space $\teich_1$ can be identified with the upper half-plane $\mathbb{H}$.  
In a formal calculation, for $t \in \mathbb{R} \subset \hat{\mathbb{R}} = \partial \mathbb{H}$ and $\lambda = [x, y] \in \mathcal{MF} \setminus \{0\}$ corresponding to $t = -x/y$ via the map \eqref{eq:PMF_torus}, we have  
\begin{equation}
\label{eq:Poisson_kernel_Torus}
\frac{\ext_{\imaginaryunit}(\lambda)}{\ext_{\tau}(\lambda)}
= (x^2 + y^2)\frac{\operatorname{Im}(\tau)}{|x + y\tau|^2}
= \frac{\operatorname{Im}(\tau)(1 + t^2)}{|\tau - t|^2},
\end{equation}
which coincides with the Poisson kernel appearing in the last term of \eqref{eq:HM_upper-half-plane}.

Next, we discuss how the normalized spherical measure $d(T_*\boldsymbol{\omega}_0)$ in \eqref{eq:normalized_spherical_measure} can be interpreted from the viewpoint of Teichm\"uller theory.

For $\tau \in \mathbb{H} \cong \teich_1$, define
\[
\mathcal{SMF}_\tau = \left\{ \lambda \in \mathcal{MF} \mid \ext_\tau(\lambda) = 1 \right\}
\]
(cf. Figure \ref{fig:Indicatrix}).
\begin{figure}[t]
\includegraphics[width=10cm, bb=13 11 586 78]{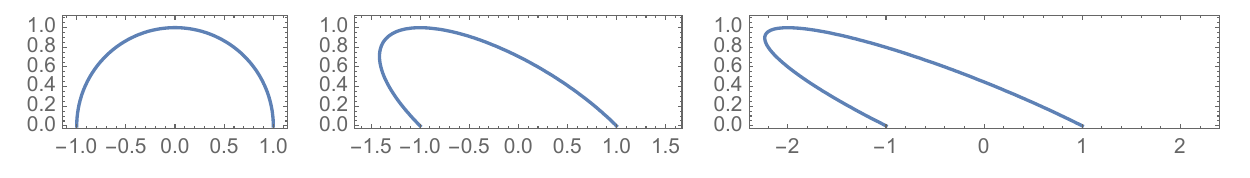}
\caption{$\mathcal{SMF}_\tau$ in $\mathcal{MF}=\mathbb{R}^2/\mathbb{Z}_2$ for $\tau=\imaginaryunit$, $1+\imaginaryunit$, and $2+\imaginaryunit$ (from left to right).}
\label{fig:Indicatrix}
\end{figure}

The projection $\mathcal{MF} \setminus \{0\} \to \mathcal{PMF}$ induces a homeomorphism
\[
\boldsymbol{\psi}_\tau \colon \mathcal{SMF}_\tau \to \mathcal{PMF}.
\]
For any subset $A \subset \mathcal{PMF}$, we define
\begin{equation}
\label{eq:cone_PMF}
\operatorname{Cone}_\tau(A) = \left\{ t\lambda \in \mathcal{MF} \mid
\lambda \in \mathcal{SMF}_\tau \text{ with } \boldsymbol{\psi}_\tau(\lambda) \in A, \quad 0 \le t \le 1 \right\}.
\end{equation}
Then, the \emph{Thurston measure}\index{Thurston measure} on $\mathcal{PMF}$ associated with $\tau \in \mathbb{H} \cong \teich_1$ is defined by
\[
\PThursM^\tau(A) = \frac{\ThursM\left(\operatorname{Cone}_\tau(A)\right)}{\ThursM\left(\operatorname{Cone}_\tau(\mathcal{PMF})\right)}
\]
for any Borel set $A \subset \mathcal{PMF}$.

We claim that $\PThursM^\tau$ coincides with the harmonic measure \eqref{eq:HM_upper-half-plane} at $\tau \in \mathbb{H}$. First, consider the special case $\tau = \imaginaryunit$. 
Since $\ext_\imaginaryunit([x,y]) = x^2 + y^2$, the set $\mathcal{SMF}_\imaginaryunit$ is a (half) circle.  
Hence, when we parametrize $\mathcal{SMF}_\imaginaryunit$ by $x = \cos\theta$ and $y = \sin\theta$ for $0 \le \theta \le \pi$, the measure $\PThursM^\imaginaryunit$ is simply the pushforward of the normalized angular measure $d\theta / \pi$ via $\boldsymbol{\psi}_\imaginaryunit$. Thus,
\begin{equation}
\label{eq:pushforward_measure_imaginary_unit}
d\PThursM^\imaginaryunit= (\boldsymbol{\psi}_\imaginaryunit)_*\left(\frac{d\theta}{\pi}\right)
= \frac{1}{\pi(1 + t^2)}\,dt,
\end{equation}
where $t = -x/y = -\cot(\theta)$, which coincides with the normalized spherical measure \eqref{eq:normalized_spherical_measure}.

Next, we verify the following formula:
\begin{equation}
\label{eq:harmonic_measure_teich1}
d\PThursM^\tau = \frac{\ext_{\imaginaryunit}(\lambda)}{\ext_{\tau}(\lambda)}\, d\PThursM^\imaginaryunit
\end{equation}
on $\mathcal{PMF}$, which implies that $\PThursM^\tau$ is the harmonic measure at $\tau \in \mathbb{H} \cong \teich_1$ via \eqref{eq:HM_upper-half-plane} and \eqref{eq:Poisson_kernel_Torus}.

Let $\tau \in \mathbb{H}$.  
Define the map $T_\tau \colon \mathcal{MF} \to \mathcal{MF}$ by
\[
T_\tau(\lambda) = \left(\frac{\ext_\imaginaryunit(\lambda)}{\ext_\tau(\lambda)}\right)^{1/2} \lambda.
\]
Then $T_\tau$ maps $\mathcal{SMF}_\imaginaryunit$ onto $\mathcal{SMF}_\tau$.
Let $A \subset \mathcal{PMF}$ be a Borel set.
Let $A' \subset \mathcal{SMF}_\tau$ satisfy $\boldsymbol{\psi}_\tau(A')=A$.
Consider the cone $\operatorname{Cone}(A') \subset \mathcal{MF}$ as defined in \eqref{eq:cone_PMF}.  
Let $(r, \theta)$ denote polar coordinates on $\mathcal{MF}$ (with $r \ge 0$ and $0 \le \theta < \pi$), and let $A'' \subset [0, \pi)$ be the set of angles corresponding to $A'$. Then we have
\begin{align*}
\ThursM(\operatorname{Cone}_\tau(A))
&= \frac{1}{2} \int_{\operatorname{Cone}(A')} dx\,dy \\
&= \frac{1}{2} \int_{A''} \int_0^{T_\tau([\cos\theta, \sin\theta])} r\,dr\,d\theta \\
&= \frac{1}{4} \int_{A''} T_\tau([\cos\theta, \sin\theta])^2\, d\theta \\
&= \frac{\pi}{4} \int_{A''} \frac{\ext_\imaginaryunit([\cos\theta, \sin\theta])}{\ext_\tau([\cos\theta, \sin\theta])} \cdot \frac{d\theta}{\pi} \\
&= \frac{\pi}{4} \int_{\boldsymbol{\psi}_\tau(A')} \frac{\ext_\imaginaryunit(\lambda)}{\ext_\tau(\lambda)}\, d\PThursM^\imaginaryunit([\lambda]),
\end{align*}
and in particular, using \eqref{eq:Poisson_kernel_Torus} and \eqref{eq:pushforward_measure_imaginary_unit},
\begin{align*}
\ThursM(\operatorname{Cone}_\tau(\mathcal{PMF}))
&= \frac{\pi}{4} \int_{\mathcal{PMF}} \frac{\ext_\imaginaryunit(\lambda)}{\ext_\tau(\lambda)}\, d\PThursM^\imaginaryunit([\lambda]) \\
&= \frac{\pi}{4} \int_{\hat{\mathbb{R}}} \frac{\operatorname{Im}(\tau)(1 + t^2)}{|\tau - t|^2} \cdot \frac{1}{\pi(1 + t^2)}\, dt
= \frac{\pi}{4}.
\end{align*}
These calculations confirm the validity of formula \eqref{eq:harmonic_measure_teich1}.

\subsection{Function theory on $\teich_{g,m}$}
\label{subsec:function_theory_Tgm}

In the previous sections, we provided an interpretation of the Poisson formula in the Teichm\"uller theoretic setting for the Teichm\"uller space of tori.  
Recently, this observation has been generalized to the case of arbitrary analytically finite Riemann surfaces.

\subsubsection{Recap of the situation}

Let $x \in \teich_{g,m}$. We denote
\[
\mathcal{SML}_x = \{\lambda \in \ml_{g,m} \mid \ext_x(\lambda) = 1\}.
\]
Then, the projection $\ml_{g,m} \setminus \{0\} \to \pml_{g,m}$ induces a homeomorphism
\[
\boldsymbol{\psi}_x \colon \mathcal{SML}_x \to \pml_{g,m}.
\]

For any Borel set $A \subset \pml_{g,m}$, we define
\[
\PThursM^x(A) = \frac{\ThursM(\operatorname{Cone}_x(A))}{\ThursM(\operatorname{Cone}_x(\pml_{g,m}))},
\]
where the cone $\operatorname{Cone}_x(A)$ of $A$ is defined as in \eqref{eq:cone_PMF}.  
The measure $\PThursM^x$ is called the \emph{Thurston measure}\index{Thurston measure} on $\pml_{g,m}$ at $x \in \teich_{g,m}$.

Let $[\omega] \in \mcg(\Sigma_{g,m})$ and $x = (M,f) \in \teich_{g,m}$.  
Since $\ext_{[\omega](x)}(\lambda) = \ext_x([\omega]^{-1}(\lambda))$, for any $A \subset \pml_{g,m}$, we have
\begin{align*}
\operatorname{Cone}_{[\omega](x)}(A)
&= \{\lambda \in \ml_{g,m} \mid \ext_{[\omega](x)}(\lambda) \le 1, \ [\lambda] \in A\} \\
&= \{\lambda \in \ml_{g,m} \mid \ext_x([\omega]^{-1}(\lambda)) \le 1, \ [\lambda] \in A\} \\
&= \{[\omega](\lambda) \in \ml_{g,m} \mid \ext_x(\lambda) \le 1, \ [\omega]([\lambda]) \in A\} \\
&= [\omega]\big(\{\lambda \in \ml_{g,m} \mid \ext_x(\lambda) \le 1, \ [\lambda] \in [\omega]^{-1}(A)\}\big) \\
&= [\omega](\operatorname{Cone}_x([\omega]^{-1}(A))).
\end{align*}
Since the Thurston measure $\ThursM$ is $\mcg(\Sigma_{g,m})$-invariant, it follows that
\begin{align}
\label{eq:invariance_ThP}
\PThursM^{[\omega](x)}(A)
&= \frac{\ThursM(\operatorname{Cone}_{[\omega](x)}(A))}{\ThursM(\operatorname{Cone}_{[\omega](x)}(\pml_{g,m}))} \\
&= \frac{\ThursM(\operatorname{Cone}_x([\omega]^{-1}(A)))}{\ThursM(\operatorname{Cone}_x(\pml_{g,m}))} \notag \\
&= \PThursM^x([\omega]^{-1}(A)) \notag \\
&= [\omega]_*(\PThursM^x)(A). \notag
\end{align}
By applying a similar calculation, we also obtain
\begin{equation}
\label{eq:invariance_ThP2}
\PThursM^{y}(A)=\int_A\left(\frac{\ext_x(\lambda)}{\ext_y(\lambda)}\right)^{3g-3+m}d\PThursM^{x}([\lambda])
\end{equation}
for $x$, $y\in \teich_{g,m}$ (cf. \cite[\S2.3.1]{MR2913101}. See also \cite[Corollary 15.3]{MR4633651})

\subsubsection{Pluriharmonic measure in the sense of Demailly}

Let $\partial^{\mathrm{ue}} \Bers{x_0}$ be the set of totally degenerate groups without accidental parabolics whose ending lamination is the support of a uniquely ergodic measured lamination.  
For $\varphi \in \partial^{\mathrm{ue}} \Bers{x_0}$, let $\lambda_\varphi$ be the uniquely ergodic measured lamination whose support is the ending lamination of $\varphi$.  
We define a function
\[
\mathbb{P} \colon \teich_{g,m} \times \teich_{g,m} \times \partial \Bers{x_0} \to \mathbb{R}
\]
by
\[
\mathbb{P}(x,y,\varphi) =
\begin{cases}
\left( \frac{\ext_x(\lambda_\varphi)}{\ext_y(\lambda_\varphi)} \right)^{3g - 3 + m} & (\varphi \in \partial^{\mathrm{ue}} \Bers{x_0}), \\
1 & \text{otherwise}.
\end{cases}
\]

In \cite{MR4633651}, the author showed the following.

\begin{theorem}[Pluriharmonic measure]
\label{thm:pluriharmonic_measure}
For $x \in \teich_{g,m}$, the pushforward measure
\[
\pushThursMBers_x := (\Xi_{x_0})_* \PThursM^x
\]
is the pluriharmonic measure at $x$ on the Bers slice $\Bers{x_0}$ in the sense of Demailly \cite{MR881709}, where $\Xi_{x_0}$ is the map defined in \eqref{eq:map_PMLmf_to_BB}. Furthermore,
\[
d\pushThursMBers_x = \mathbb{P}(x,y,\cdot)\, d\pushThursMBers_y,
\]
for all $x,y \in \teich_{g,m}$.  
Hence, the factor $\mathbb{P}(x,y,\cdot)$ is interpreted as the Poisson kernel.
\end{theorem}

\Cref{thm:pluriharmonic_measure} provides a complex-analytic description of the Thurston measure on $\pml_{g,m}$.  
The Thurston measures on $\pml_{g,m}$ have previously been studied from a dynamical perspective.  
For instance, Masur \cite{MR644018} analyzed the measure $\PThursM^x$ to investigate the dynamical properties of the Teichm\"uller geodesic flow.  
In \cite{MR2913101}, Athreya, Bufetov, Eskin, and Mirzakhani observed that the family $\{\PThursM^x\}_{x \in \teich_{g,m}}$ constitutes the unique conformal density (up to scaling) for the action of the mapping class group on $\teich_{g,m}$.  
The only possible dimension is $6g - 6 + 2m$, as determined by the measure classification result of \cite{MR2424174} (see also \cite[\S3]{MR2495764}).

\subsubsection{Radial limit theorem}

The following is proved in \cite{MR4830064}:

\begin{theorem}[Radial limit theorem]
\label{thm:radial_limit}
Let $u$ be a bounded pluriharmonic function on $\teich_{g,m}$. Then there exists a measurable set $\mathcal{E}_0 = \mathcal{E}_0(u) \subset \pml_{g,m}$ such that:
\begin{itemize}
\item[(1)] $\PThursM^x(\mathcal{E}_0) = 1$ for all $x \in \teich_{g,m}$;
\item[(2)] for any $[\lambda] \in \mathcal{E}_0$ and $x \in \teich_{g,m}$, the radial limit
\[
\lim_{t \to \infty} u\big(\TRay{x;[\lambda]}(t)\big)
\]
exists;
\item[(3)] for any $x_1, x_2 \in \teich_{g,m}$ and $[\lambda] \in \mathcal{E}_0(u)$,
\[
\lim_{t \to \infty} u\big(\TRay{x_1;[\lambda]}(t)\big) = \lim_{t \to \infty} u\big(\TRay{x_2;[\lambda]}(t)\big).
\]
\end{itemize}
Furthermore, the set $\mathcal{E}_0$ is contained in $\pml_{g,m}^{\mathrm{ue}}$.
\end{theorem}

The idea of the proof is to consider the family of Teichm\"uller disks passing through a given point $x_0 \in \teich_{g,m}$.  
By restricting to Teichm\"uller disks, a given pluriharmonic function becomes a bounded harmonic function on the disk.  
Hence, the Fatou theorem, namely \Cref{thm:radial_limit}, can be applied.  
The condition (2) in \Cref{thm:radial_limit} follows from Masur's criterion (cf. \Cref{thm:masur_geodesic}).

\begin{definition}[Radial limit function]
Let $u$ be a bounded pluriharmonic function on $\teich_{g,m}$.  
The \emph{radial limit function}\index{radial limit function of pluriharmonic function} $u^*$ of $u$ is the bounded measurable function on $\pml_{g,m}$ defined by
\[
u^*([\lambda]) :=
\begin{cases}
\displaystyle \lim_{t \to \infty} u\big(\TRay{x;[\lambda]}(t)\big) & \text{if } [\lambda] \in \mathcal{E}_0(u), \\
0 & \text{otherwise},
\end{cases}
\]
for some (equivalently, any) $x \in \teich_{g,m}$.  
By condition (3) in \Cref{thm:radial_limit}, this definition is independent of the choice of $x$.
\end{definition}

When appropriate, we also assume that the set $\mathcal{E}_0(u)$ is invariant under the action of a subgroup of $\mcg(\Sigma_{g,m})$, since $\mcg(\Sigma_{g,m})$ is a countable group.  
If $\mathcal{E}_0(u)$ is invariant under the action of a subgroup $H < \mcg(\Sigma_{g,m})$, then condition (3) in \Cref{thm:radial_limit} further implies that the action of any mapping class $[\omega] \in H$ on both $\teich_{g,m}$ as in \eqref{eq:mcg_action_T} and on $\pml_{g,m}$ as in \eqref{eq:actionMCG_on_PML} commutes with taking radial limits:
\begin{equation}
\label{eq:invariance_radial_limits}
(u \circ [\omega])^* = u^* \circ [\omega].
\end{equation}

\subsubsection{Poisson integral}

In a forthcoming paper \cite{bounded-PIF-miyachi2025}, using radial limits of bounded pluriharmonic functions, we obtain the following result, which serves as a full analogue of the Fatou theorem (\Cref{thm:Fatou1}).

\begin{theorem}[Poisson integral formula]
\label{thm:PIF_bdd_pluri_harmonic_T}
Fix $x_0 \in \teich_{g,m}$.  
Let $u$ be a bounded pluriharmonic function on $\teich_{g,m}$, and let $u^* \in L^\infty(\pml_{g,m}, \PThursM^{x_0})$ denote its radial limit. Then
\[
u(x) = \int_{\pml_{g,m}} u^*([\lambda]) \left( \frac{\ext_{x_0}(\lambda)}{\ext_{x}(\lambda)} \right)^{3g - 3 + m} d\PThursM^{x_0}([\lambda])
\]
for all $x \in \teich_{g,m}$.
\end{theorem}

Since $\pml_{g,m}^{\textrm{ue}} \subset \pml_{g,m}^{\textrm{mf}}$, \Cref{thm:radial_limit}, together with \Cref{prop:limit_T-ray} and \eqref{eq:map_PMLmf_to_BB}, ensures the existence of radial limit functions (via Teichmüller geodesic rays) for bounded pluriharmonic functions on $\Bers{x_0}$. \Cref{thm:PIF_bdd_pluri_harmonic_T} then yields the corresponding Poisson integral formula. Namely, we have:

\begin{theorem}[Poisson integral formula for $\Bers{x_0}$]
\label{thm:PIF_bdd_pluri_harmonic_TBers}
Let $u$ be a bounded pluriharmonic function on $\Bers{x_0}$, and let $u^* \in L^\infty(\partial \Bers{x_0}, \pushThursMBers_{x_0})$ denote its radial limit. Then
\[
u(x) = \int_{\partial \Bers{x_0}} u^*(\varphi) \mathbb{P}(x_0, x, \varphi) \, d\pushThursMBers_{x_0}(\varphi)
\]
for all $x \in \teich_{g,m} \cong \Bers{x_0}$.
\end{theorem}

As a corollary, we obtain:

\begin{corollary}[Radial limits are non-constant]
\label{coro:radial_limit_non-constant}
The radial limit of any non-constant bounded pluriharmonic function on $\teich_{g,m}$ is non-constant as a measurable function.
\end{corollary}

The Poisson integral formulas in \Cref{thm:PIF_bdd_pluri_harmonic_T} and \Cref{thm:PIF_bdd_pluri_harmonic_TBers} were first established in~\cite{MR4633651} for bounded pluriharmonic functions on $\Bers{x_0}$ that extend continuously to the Bers compactification.

\Cref{coro:radial_limit_non-constant} was considered in~\cite{MR4830064} in the context of bounded holomorphic functions, where a version of the F. and M. Riesz theorem was proved for their boundary value functions. This result was then applied to show that the action of the Torelli group on $\pml_{g}$ is not ergodic. We will return to this topic in \Cref{remark:1}.

Let $f \in L^1(\pml_{g,m}, \PThursM^{x_0})$. We define the \emph{Poisson integral}\index{Poisson integral}\index{Teichmüller space!Poisson integral} of $f$ by
\[
\mathbb{P}[f](x) = \int_{\pml_{g,m}} f([\lambda]) \left( \frac{\ext_{x_0}(\lambda)}{\ext_{x}(\lambda)} \right)^{3g - 3 + m} d\PThursM^{x_0}([\lambda]).
\]
From the Gardiner formula discussed in \S\ref{subsec:Gardiner-formula-revisited}, the partial derivatives $\partial \mathbb{P}[f]|_{x} \in \mathfrak{C}^{1,0}_x(\teich_{g,m})$ and $\overline{\partial} \mathbb{P}[f]|_{x} \in \mathfrak{C}^{0,1}_x(\teich_{g,m})$ are given by:
\begin{align*}
\partial \mathbb{P}[f]|_{x} &=
(3g - 3 + m)\int_{\pml_{g,m}}
f([\lambda])
\left( \frac{\ext_{x_0}(\lambda)}{\ext_{x}(\lambda)} \right)^{3g - 3 + m}
\dfrac{q_{\lambda,x}}{\|q_{\lambda,x}\|}  
d\PThursM^{x_0}([\lambda]), \\
\overline{\partial} \mathbb{P}[f]|_{x} &=
(3g - 3 + m)\int_{\pml_{g,m}}
f([\lambda])
\left( \frac{\ext_{x_0}(\lambda)}{\ext_{x}(\lambda)} \right)^{3g - 3 + m} \dfrac{\overline{q_{\lambda,x}}}{\|q_{\lambda,x}\|}
d\PThursM^{x_0}([\lambda]).
\end{align*}
Therefore, the function $f$ defines a holomorphic function $\mathbb{P}[f]$ on $\teich_{g,m}$ if and only if
\begin{equation}
\label{eq:CR}
\int_{\pml_{g,m}}
f([\lambda])
\left( \frac{\ext_{x_0}(\lambda)}{\ext_{x}(\lambda)} \right)^{3g - 3 + m} \dfrac{\overline{q_{\lambda,x}}}{\|q_{\lambda,x}\|}
d\PThursM^{x_0}([\lambda]) = 0
\end{equation}
for all $x \in \teich_{g,m}$. In particular, by \Cref{thm:PIF_bdd_pluri_harmonic_T}, the radial limit $f^*$ of a bounded holomorphic function $f$ on $\teich_{g,m}$ satisfies \eqref{eq:CR}.

\subsection{Future challenge: Function theory and Fourier analysis}
\subsubsection{Herglotz-type formula and Hardy space}
\label{subsub:Her}
G.~Herglotz \cite{Herglotz_G_1911} showed that for any positive harmonic function $u$ on the unit disk $\mathbb{D}$, there exists a unique Borel measure $\mu$ on $\partial \mathbb{D}$ such that
\[
u(z) = \int_{\partial \mathbb{D}} \dfrac{1 - |z|^2}{|z - e^{i\theta}|^2} \, d\mu(e^{i\theta})
\]
for all $z \in \mathbb{D}$ (cf.\ \cite[Note in \S1]{MR268655}).

Given the natural development of function theory, once the Poisson integral formula is established, a Herglotz-type theorem becomes a natural question. The corresponding problem in our context is as follows:

\begin{problem}[Herglotz-type formula]
\label{problem:Her}
For any positive pluriharmonic function $u$ on $\teich_{g,m}$, does there exist a Borel measure $\nu$ on $\pml_{g,m}$ such that
\[
u(x) = \mathbb{P}[\nu](x) =
\int_{\pml_{g,m}}
\left( \frac{\ext_{x_0}(\lambda)}{\ext_{x}(\lambda)} \right)^{3g - 3 + m} d\nu([\lambda])
\]
for all $x \in \teich_{g,m}$?
\end{problem}

An affirmative answer to \Cref{problem:Her} is expected to yield meaningful measures on $\pml_{g,m}$, which would be useful for studying the ergodic theory of the action of subgroups of $\mcg(\Sigma_{g,m})$ discussed in \S\ref{sec:function_theory_dymamics}.

The Poisson integral formulas (\Cref{thm:PIF_bdd_pluri_harmonic_T} and \Cref{thm:PIF_bdd_pluri_harmonic_TBers}) represent bounded pluriharmonic functions on $\teich_{g,m}$ by their boundary functions. However, in contrast to the Fatou theorem (\Cref{thm:Fatou1}), it is not clear which $L^\infty$ functions on $\pml_{g,m}$ are realized as boundary values of bounded pluriharmonic functions.

\begin{problem}
\label{problem:bdd}
Characterize the boundary functions on $\pml_{g,m}$ of bounded pluriharmonic functions on $\teich_{g,m}$.
\end{problem}

\begin{problem}
\label{problem:bdd2}
Characterize Borel measures on $\pml_{g,m}$ whose Poisson integrals are bounded pluriharmonic functions on $\teich_{g,m}$.
\end{problem}

\begin{problem}
\label{problem:hardy}
Study the Hardy spaces on $\teich_{g,m}$.
\end{problem}

The boundary value function (i.e., the radial limit function) of a bounded holomorphic function is essentially characterized by \eqref{eq:CR}. However, its analytic and geometric properties remain largely unknown.

\begin{problem}
\label{problem:CR}
Study the analytic and geometric properties of bounded measurable functions that satisfy the tangential CR-equation given in \eqref{eq:CR}.
\end{problem}
\subsubsection{Analytic measures}
\label{subsubsec:analytic_measures}
The original Poisson kernel in \eqref{eq:PIF} has the form
\begin{align*}
\dfrac{1 - |z|^2}{|z - e^{i\theta}|^2}
&= \operatorname{Re} \left( \dfrac{e^{i\theta} + z}{e^{i\theta} - z} \right)
= 1 + \sum_{n = 1}^\infty e^{-in\theta} z^n + \sum_{n = 1}^\infty e^{in\theta} \overline{z}^{n}.
\end{align*}
Therefore, for an $L^1$ function $f$ on $[0, 2\pi]$, the condition
\[
\int_0^{2\pi} \overline{\partial} \left( \dfrac{1 - |z|^2}{|z - e^{i\theta}|^2} \right) f(\theta)\, d\theta = 0
\]
is equivalent to the vanishing of the negative Fourier coefficients of $f$, that is,
\[
a_n = \int_0^{2\pi} e^{-in\theta} f(\theta)\, d\theta = 0 \quad \text{for all } n < 0.
\]

F.~and M.~Riesz \cite{F_and_M_Riesz_1916} showed that if a finite complex regular Borel measure $\nu$ on the circle satisfies
\begin{equation}
\label{eq:analytic_measure}
\int_0^{2\pi} e^{in\theta} \, d\nu(\theta) = 0 \quad \text{for all } n = 1, 2, \ldots,
\end{equation}
then $\nu$ is absolutely continuous with respect to the Lebesgue measure. Moreover, if the total variation of $\nu$ vanishes on a set of positive Lebesgue measure, then $\nu$ must be the zero measure.

From a function-theoretic point of view, these observations are deeply connected through the theory of Hardy spaces (cf.\ \cite[Theorem 3.7]{MR268655}). More recently, measures satisfying \eqref{eq:analytic_measure} have been generalized to regular Borel measures on compact abelian groups in the context of Fourier analysis. Such measures are called \emph{analytic measures} (see, e.g., \cite{MR144152}, \cite{MR157193}, \cite{MR209771}).

Motivated by this analogy, we call a finite complex regular Borel measure $\nu$ on $\pml_{g,m}$ an \emph{analytic measure} if
\begin{equation}
\label{eq:CR2}
\int_{\pml_{g,m}}
\left( \frac{\ext_{x_0}(\lambda)}{\ext_{x}(\lambda)}\right)^{3g - 3 + m} \dfrac{\overline{q_{\lambda,x}}}{\|q_{\lambda,x}\|}
\, d\nu([\lambda]) = 0
\end{equation}
for all $x \in \teich_{g,m}$.

By the preceding discussion, the Poisson integral
\[
\mathbb{P}[\nu](x) =
\int_{\pml_{g,m}}
\left( \frac{\ext_{x_0}(\lambda)}{\ext_{x}(\lambda)}\right)^{3g - 3 + m} \, d\nu([\lambda])
\]
defines a holomorphic function on $\teich_{g,m}$.

\begin{problem}[Fourier coefficients]
\label{problem:Fourier}
Is there an analogue of Fourier coefficients for Borel (regular) measures on $\pml_{g,m}$ that characterizes whether a given measure is an analytic measure?
\end{problem}

\begin{problem}[Absolute continuity]
\label{problem:abs_conti}
Is every analytic measure absolutely continuous with respect to the Thurston measure?
\end{problem}
\section{Dynamics on Projective Measured Laminations}
\label{sec:DynPML}
In this section, we recall the classification of subgroups of $\mcg(\Sigma_{g,m})$ by McCarthy and Papadopoulos.

\subsection{Prototype example: Kleinian groups}
For references, see \cite{MR1219310}, \cite{MR1638795}, \cite{MR1862839}, and \cite{Thuston-LectureNote}.

A \emph{Kleinian group}\index{Kleinian group} is a discrete subgroup of $\psl_2(\mathbb{C})$. The \emph{limit set}\index{limit set}\index{Kleinian group!limit set} $\Lambda(\Gamma)$ is the set of accumulation points of the orbit $\Gamma(p_0)$ for some (equivalently, any) $p_0 \in \mathbb{H}^3$.  
Since $\Gamma$ is discrete, $\Lambda(\Gamma)$ is a closed subset of $\partial \mathbb{H}^3 = \hat{\mathbb{C}}$.
The complement $\Omega(\Gamma) = \hat{\mathbb{C}} \setminus \Lambda(\Gamma)$ is called the \emph{region of discontinuity}\index{region of discontinuity}\index{Kleinian group!region of discontinuity} of $\Gamma$.

The action of $\Gamma$ on $\Omega(\Gamma)$ is properly discontinuous.
The limit set is characterized as the minimal closed $\Gamma$-invariant subset of $\hat{\mathbb{C}}$.
The cardinality of $\Lambda(\Gamma)$ is either at most $2$ or uncountable. A Kleinian group $\Gamma$ is called \emph{elementary}\index{elementary group}\index{Kleinian group!elementary group} if $\Lambda(\Gamma)$ is finite, and \emph{non-elementary}\index{non-elementary group}\index{Kleinian group!non-elementary group} otherwise. Elementary Kleinian groups are virtually abelian and have been completely classified.

The limit set $\Lambda(\Gamma)$ of a non-elementary Kleinian group $\Gamma$ is a perfect set.  
For such a group $\Gamma$, the limit set $\Lambda(\Gamma)$ coincides with the closure of the set of fixed points of all loxodromic elements in $\Gamma$.  
This characterization implies that $\Lambda(\Gamma_1) = \Lambda(\Gamma)$ for any normal subgroup $\Gamma_1$ of $\Gamma$.  
Indeed, for a loxodromic element $\gamma_1 \in \Gamma_1$ with fixed point $z_0 \in \hat{\mathbb{C}}$ and any $\gamma \in \Gamma$, the conjugate $\gamma \gamma_1 \gamma^{-1}$ is a loxodromic element in $\Gamma_1$ with fixed point $\gamma(z_0)$.  
This shows that the set of loxodromic fixed points of $\Gamma_1$ is $\Gamma$-invariant, and hence so is $\Lambda(\Gamma_1)$.  
This observation implies that any normal subgroup of a non-elementary Kleinian group is also non-elementary.  
It is also known that $\Lambda(\Gamma_1) = \Lambda(\Gamma)$ whenever $\Gamma_1$ is a finite-index subgroup of $\Gamma$.

\subsection{McCarthy--Papadopoulos classification}
\label{subsec:McP_class}
The dynamics of subgroups of the mapping class group, including their limit sets and regions of discontinuity, were investigated by Masur~\cite{MR837978} and McCarthy--Papadopoulos~\cite{MR982564}.

\subsubsection{Thurston’s classification of mapping classes}
We begin with Thurston’s classification of mapping classes\index{Thurston classification of mapping classes}.  
For any $[\omega] \in \mcg(\Sigma_{g,m})$, exactly one of the following holds:
\begin{enumerate}
\item $[\omega]$ is of finite order;
\item $[\omega]$ is \emph{reducible}\index{reducible}\index{Thurston classification of mapping classes!reducible}, that is, there exists a system $A = \{\alpha_i\}_{i=1}^n \subset \mathcal{S}$ such that $i(\alpha_i,\alpha_j) = 0$ for $i \ne j$ and $[\omega](A) = A$;
\item $[\omega]$ is \emph{pseudo-Anosov}\index{pseudo Anosov}\index{Thurston classification of mapping classes!pseudo Anosov}, that is, there exists a pair $(\lambda^s, \lambda^u)$ of transverse measured geodesic laminations and a constant $K > 1$ such that $[\omega](\lambda^s) = (1/K)\lambda^s$ and $[\omega](\lambda^u) = K\lambda^u$.
\end{enumerate}

In case (3), the projective classes $[\lambda^s]$ and $[\lambda^u]$ are the repelling and attracting fixed points, respectively, of the action of $[\omega]$ on $\pml_{g,m}$, in the sense that  
for any $x \in \teich_{g,m} \cup (\pml_{g,m} \setminus \{[\lambda^s], [\lambda^u]\})$,
\[
\lim_{n \to \infty} [\omega]^n(x) = [\lambda^u], \quad
\text{and} \quad
\lim_{n \to -\infty} [\omega]^n(x) = [\lambda^s].
\]
Moreover, $\lambda^s$ and $\lambda^u$ are uniquely ergodic  
(cf., e.g., \cite[Exposés 9, 12]{MR568308} and \cite{MR956596}).
Following McCarthy--Papadopoulos \cite{MR982564}, we call the pair $\{\lambda^u, \lambda^s\}$ the \emph{pseudo-Anosov pair}\index{pseudo-Anosov pair} associated with $[\omega]$.  
A measured geodesic lamination $\lambda$ whose projective class is fixed by a pseudo-Anosov mapping class is called a \emph{pseudo-Anosov lamination}\index{pseudo-Anosov lamination}.

\subsubsection{McCarthy--Papadopoulos classification}
\label{subsubsec:McP_classfication}
A subgroup $H$ is said to be \emph{reducible}\index{reducible}\index{subgroups of mapping class group!reducible} if it preserves a nonempty finite collection of disjoint simple closed curves $\mathcal{S} \subset \pml_{g,m}$. Otherwise, $H$ is called \emph{irreducible}\index{irreducible}\index{subgroups of mapping class group!irreducible}.

A subgroup $H$ of $\mcg(\Sigma_{g,m})$ is called \emph{dynamically irreducible}\index{dynamically irreducible}\index{subgroups of mapping class group!dynamically irreducible} if it has a unique nonempty minimal closed invariant subset in $\pml_{g,m}$. Otherwise, $H$ is said to be \emph{dynamically reducible}\index{dynamically reducible}\index{subgroups of mapping class group!dynamically reducible}.

A set of pseudo-Anosov mapping classes is said to be \emph{independent} if no two elements in the set have the same fixed point set in $\pml_{g,m}$. A subgroup $H$ of $\mcg(\Sigma_{g,m})$ is called \emph{sufficiently large}\index{sufficiently large}\index{subgroups of mapping class group!sufficiently large} if it contains an independent pair of pseudo-Anosov elements.

Let $H$ be a subgroup of $\mcg(\Sigma_{g,m})$.  
Let $\Lambda_0(H)$ denote the set of pseudo-Anosov laminations associated with pseudo-Anosov mapping classes in $H$, and let $\Lambda(H)$ be its closure.  
When $H$ is sufficiently large, $\Lambda(H)$ is the unique $H$-invariant minimal closed subset of $\pml_{g,m}$ (cf.\ \cite[Theorem 4.1]{MR982564}).  

Let $L = \{\lambda^u, \lambda^s\}$ be a pseudo-Anosov pair in $\ml_{g,m}$, and let $|L|$ denote the corresponding pair in $\pml_{g,m}$.  
If $H$ is a subgroup of the stabilizer of $|L|$ in $\mcg(\Sigma_{g,m})$, we say that $H$ is \emph{pseudo-Anosov stabilizing}\index{pseudo-Anosov stabilizing}\index{subgroups of mapping class group!pseudo-Anosov stabilizing}.  
If, in addition, $H$ contains an element that exchanges $[\lambda^u]$ and $[\lambda^s]$, we say that $H$ is of \emph{symmetric} type; otherwise, $H$ is of \emph{asymmetric} type.

McCarthy and Papadopoulos \cite{MR982564} gave the following classification of subgroups of $\mcg(\Sigma_{g,m})$ from a dynamical point of view.

\begin{theorem}[McCarthy--Papadopoulos classification]
\label{thm:M-P}
Let $H$ be a subgroup of $\mcg(\Sigma_{g,m})$.  
Then $H$ is dynamically irreducible if and only if it is of one of the following types:
\begin{itemize}
\item[{\rm (1)}]
$H$ is sufficiently large;
\item[{\rm (2)}]
$H$ is an infinite pseudo-Anosov stabilizing subgroup of symmetric type;
\item[{\rm (3)}]
$H$ is infinite and reducible, and $(g,m) = (1,1)$ or $(0,4)$.
\end{itemize}
Otherwise, $H$ is dynamically reducible if and only if it is of one of the following types:
\begin{itemize}
\item[{\rm (4)}]
$H$ is finite;
\item[{\rm (5)}]
$H$ is a pseudo-Anosov stabilizing subgroup of asymmetric type;
\item[{\rm (6)}]
$H$ is infinite and reducible, and $(g,m)$ is neither $(1,1)$ nor $(0,4)$.
\end{itemize}
\end{theorem}

Following McCarthy--Papadopoulos \cite{MR982564}, subgroups of types (2), (3), (4), or (5) are called \emph{elementary groups}\index{elementary group}\index{subgroups of mapping class group!elementary group}, since they are all virtually cyclic and belong to a slightly more general class, namely the virtually abelian groups. The action of a subgroup of type (6) can, in some cases, be reduced to cases of lower topological complexity (cf.\ \cite[Lemma 4.5]{MR982564}).

\subsubsection{Subgroups that are sufficiently large}
\label{subssub:subgroups_suf_large}
Suppose a subgroup $H$ of the mapping class group $\mcg(\Sigma_{g,m})$ is sufficiently large. In this case, the set $\Lambda(H)$ is called the \emph{limit set}\index{limit set}\index{subgroups of mapping class group!limit set} of $H$.  
Sufficiently large subgroups share several properties with non-elementary Kleinian groups.  
For instance, $\Lambda(H)$ is a perfect set, and either $\Lambda(H) = \pml_{g,m}$ or $\Lambda(H)$ has empty interior.  
Furthermore, by an argument analogous to that used in the theory of Kleinian groups, any subgroup $H_1 < H$ is also sufficiently large and satisfies $\Lambda(H_1) = \Lambda(H)$, provided that $H_1$ is either an infinite normal subgroup or a finite-index subgroup (see \cite[\S5]{MR982564}).  
Note, however, that in some cases a sufficiently large subgroup may contain a normal subgroup that is not itself sufficiently large (cf.~\cite[Remark in \S5]{MR982564}).

\begin{example}[McCarthy--Papadopoulos {\cite[\S5, Example 1]{MR982564}}]
\label{ex:McP}
The mapping class group $\mcg(\Sigma_{g,m})$ acts minimally on $\pml_{g,m}$ (cf.\ \cite[Expos\'e 6, \S VII]{MR568308}), and hence $\Lambda(\mcg(\Sigma_{g,m})) = \pml_{g,m}$, so $\mcg(\Sigma_{g,m})$ is sufficiently large.  
In particular, the Torelli group $\mathcal{I}_g$ of $\mcg(\Sigma_g)$ is also sufficiently large, and $\Lambda(\mathcal{I}_g) = \pml_{g}$.
\end{example}

\subsubsection{Region of discontinuity}
Let $H$ be a sufficiently large subgroup of $\mcg(\Sigma_{g,m})$. Following Masur \cite{MR837978} and McCarthy-Papadopoulos \cite[\S6]{MR982564}, we define
\[
Z\Lambda(H) = \{[\mu] \in \pml_{g,m} \mid i(\mu,\lambda) = 0 \text{ for some } [\lambda] \in \Lambda(H)\}.
\]
Then $Z\Lambda(H)$ is closed and $H$-invariant. If $\Lambda(H)$ is a proper subset of $\pml_{g,m}$, then so is $Z\Lambda(H)$ (cf.\ \cite[Proposition 6.1]{MR982564}).  
The complement $\Delta(H) = \pml_{g,m} \setminus Z\Lambda(H)$ is called the \emph{region of discontinuity}\index{region of discontinuity}\index{subgroups of mapping class group!region of discontinuity} of $H$.  
$H$ acts properly discontinuously on $\Delta(H)$ (\cite[Theorem 7.17]{MR982564}; see also \cite[Theorem 2.1]{MR837978}).

McCarthy and Papadopoulos observed that $Z\Lambda(H)$ contains the limit points of the action of $H$ on the Thurston compactification of $\teich_{g,m}$, in the sense that $[\lambda] \in \pml_{g,m}$ lies in $Z\Lambda(H)$ if there exist an infinite sequence $\{h_n\}_n$ of distinct elements of $H$ and $y \in \teich_{g,m} \cup \pml_{g,m}$ such that $h_n(y) \to [\lambda]$ (\cite[Proposition 8.1]{MR982564}).

Kent and Leininger extended McCarthy and Papadopoulos' result, proving that the action of $H$ is properly discontinuous on $\teich_{g,m} \cup \Delta(H)$ (cf.\ \cite[Theorem 2.1]{MR2465691}).

\begin{example}[Masur {\cite{MR837978}}]
\label{ex:Masur_Handle_body}
Let $H_g$ be a handlebody of genus $g$ and $\Sigma_g$ be regarded as the boundary $\partial H_g$.  
The handlebody group $\modu(H_g)$ of genus $g$ is the subgroup of $\mcg(\Sigma_g)$ consisting of homeomorphisms that extend to homeomorphisms of $H_g$.  
Masur studied the action of $\modu(H_g)$ and showed that it acts properly discontinuously on $\Delta(H)$.  
Moreover, the limit set $\Lambda(\modu(H_g))$ coincides with the closure of the set of simple closed curves in $\mathcal{S} \subset \pml_{g,m}$ that bound discs in $H_g$.  
Furthermore, $\Lambda(\modu(H_g))$ has empty interior.
\end{example}

\section{Function theory and Dynamics}
\label{sec:function_theory_dymamics}
Let $(X, \mathcal{B}, \mu)$ be a measure space.  
A measure $\mu$ on $X$ is said to be \emph{quasi-invariant}\index{quasi-invariant measure} under a measurable action of a group $G$ on $X$ if for every $g \in G$ and every measurable set $E \in \mathcal{B}$ with $\mu(E) = 0$, we have $\mu(g^{-1}(E)) = 0$.  
This is equivalent to the condition that the push-forward measure $g_*\mu$ is absolutely continuous with respect to $\mu$ for all $g \in G$.  

A measurable action of $G$ on a measure space $(X, \mathcal{B}, \mu)$, for which the measure $\mu$ is quasi-invariant, is said to be \emph{ergodic}\index{ergodic measure} if, for every $E \in \mathcal{B}$ such that the symmetric difference between $g(E)$ and $E$ is $\mu$-null for all $g \in G$, either $E$ or its complement $X \setminus E$ is $\mu$-null.

We are primarily concerned with the actions of subgroups of the mapping class group on $\pml_{g,m}$.
The Thurston measures $\PThursM^x$ (for $x \in \teich_{g,m}$) on $\pml_{g,m}$ are quasi-invariant under the action of $\mcg(\Sigma_{g,m})$.  
Indeed, let $[\omega] \in \mcg(\Sigma_{g,m})$ and $x \in \teich_{g,m}$. From \eqref{eq:invariance_ThP}, if $A \subset \pml_{g,m}$ is $\PThursM^x$-null, then
\[
\PThursM^x([\omega]^{-1}(A)) = \PThursM^{[\omega](x)}(A) = 0,
\]
since $\PThursM^x$ and $\PThursM^{[\omega](x)}$ are mutually absolutely continuous by \eqref{eq:invariance_ThP2}.
It is known that the action of $\mcg(\Sigma_{g,m})$ on $\pml_{g,m}$ is ergodic;  
see Masur~\cite{MR644018} and Rees~\cite{MR662738}.

\subsection{Ergodicity of the action of subgroups}
In this section, we discuss the ergodic action on $\pml_{g,m}$ from the point of view of function theory.

\subsubsection{Seidel's result}
Seidel~\cite{Seidel_metric} proved the following theorem:

\begin{theorem}[Seidel]
\label{thm:seidel}
Let $\Gamma$ be a Fuchsian group acting on $\mathbb{D}$. Then the action of $\Gamma$ on $\partial \mathbb{D}$ is ergodic\footnote{Seidel used the term “metrically transitive” to refer to ergodicity.} with respect to the Lebesgue measure on $\partial \mathbb{D}$ if and only if the Riemann surface $\mathbb{D}/\Gamma$ admits no non-constant bounded harmonic functions.
\end{theorem}

To connect with the next section, we briefly recall the proof of Seidel's theorem.

Suppose that the action of $\Gamma$ on $\partial \mathbb{D}$ is ergodic.  
Assume, for the sake of contradiction, that the Riemann surface $\mathbb{D}/\Gamma$ admits a non-constant bounded harmonic function $u$. Let $\tilde{u}$ denote the lift of $u$ to $\mathbb{D}$. Then $\tilde{u}$ is a bounded harmonic function on $\mathbb{D}$ that is $\Gamma$-invariant.  
Hence, its boundary value function $\tilde{u}^*$ is a bounded measurable function on $\partial \mathbb{D}$ that is also $\Gamma$-invariant. By the ergodicity of the action, $\tilde{u}^*$ must be constant almost everywhere.  
On the other hand, by the Fatou theorem~\eqref{eq:PIF}, we have
\[
\tilde{u}(z) = \int_0^{2\pi} \frac{1 - |z|^2}{|z - e^{i\theta}|^2} \tilde{u}^*(e^{i\theta}) \, \frac{d\theta}{2\pi}
\]
for $z \in \mathbb{D}$. It follows that $\tilde{u}$, and hence $u$, is constant, contradicting our assumption.

Conversely, suppose that the action of $\Gamma$ is not ergodic.  
Then there exists a $\Gamma$-invariant measurable subset $E \subset \partial \mathbb{D}$ such that neither $E$ nor its complement $\partial \mathbb{D} \setminus E$ has Lebesgue measure zero.  
Let $\boldsymbol{1}_E$ denote the characteristic function of $E$, and define
\begin{equation}
\label{eq:Seidel_characteristic}
\tilde{u}(z) = \int_0^{2\pi} \frac{1 - |z|^2}{|z - e^{i\theta}|^2} \boldsymbol{1}_E(e^{i\theta}) \, \frac{d\theta}{2\pi}.
\end{equation}
Then $\tilde{u}$ is a bounded harmonic function on $\mathbb{D}$ that is $\Gamma$-invariant and non-constant.  
Therefore, the Riemann surface $\mathbb{D}/\Gamma$ admits a non-constant bounded harmonic function.

\begin{remark}
Seidel's result also holds for higher-dimensional hyperbolic spaces; see \cite[Theorem 6.2.2]{MR1041575}.
\end{remark}


\subsubsection{Ergodicity of subgroups of $\mcg(\Sigma_{g,m})$}
Motivated by Seidel's result, we obtain the following theorem.

\begin{theorem}
\label{thm:main4}
Let $H$ be a subgroup of the mapping class group of $\Sigma_{g,m}$.
If the action of $H$ on $\pml_{g,m}$ is ergodic with respect to the Thurston measure, then $\teich_{g,m}/H$ admits no non-constant bounded pluriharmonic functions.
\end{theorem}

This result is an immediate consequence of \Cref{coro:radial_limit_non-constant}, together with an application of Seidel's argument \cite{Seidel_metric}, as discussed in the preceding section.

\begin{proof}
Suppose there exists a non-constant bounded pluriharmonic function $u$ on $\teich_{g,m}/H$. Then its lift $\tilde{u}$ to $\teich_{g,m}$ is a non-constant, $H$-invariant, bounded pluriharmonic function on $\teich_{g,m}$.
The boundary value $\tilde{u}^*$ of $\tilde{u}$ can then be regarded as an $H$-invariant measurable function on $\pml_{g,m}$ (cf.~\eqref{eq:invariance_radial_limits}).
By \Cref{coro:radial_limit_non-constant}, this boundary value $\tilde{u}^*$ is non-constant as a measurable function.
Therefore, the action of $H$ on $\pml_{g,m}$ is not ergodic.
\end{proof}

\begin{remark}
\label{remark:2}
The converse of \Cref{thm:main4} is non-trivial, in contrast to the one-dimensional case. Indeed, the difficulty in our setting arises from the fact that the Poisson kernel is not pluriharmonic (cf.~\cite[Remark 12.1]{MR4633651}), whereas the Poisson kernel in \eqref{eq:Seidel_characteristic} is harmonic in the variable $z \in \mathbb{D}$. Therefore, Seidel's argument cannot be applied directly.
\end{remark}

\begin{remark}
\label{remark:1}
In \cite{MR4830064}, the author investigates the non-ergodicity of the action of the Torelli group $\mathcal{I}_g$ on $\pml_g$ for $g \ge 2$. In the course of the discussion, the author essentially observes that the ergodicity of subgroup actions on $\pml_{g,m}$ implies the non-existence of non-constant bounded holomorphic functions on the quotient space. \Cref{thm:main4} provides a refinement of this observation.
\end{remark}

%

\subsubsection{Ergodic decomposition}  
\label{subsubsec:ergodic_decom}

It is straightforward to verify that if a subgroup $H$ of the mapping class group of $\Sigma_{g,m}$ acts ergodically on $\pml_{g,m}$, then its action is minimal in the sense that the limit set $\Lambda(H)$, defined in \S\ref{subsubsec:McP_classfication}, coincides with the entire space $\pml_{g,m}$. This observation implies the following:

\begin{proposition}
\label{prop:McP_ergodic}
Let $H$ be a subgroup of $\mcg(\Sigma_{g,m})$. If the action of $H$ on $\pml_{g,m}$ is ergodic, then $\Lambda(H) = \pml_{g,m}$, and $H$ is sufficiently large and hence dynamically irreducible.
\end{proposition}

Note that the converse of \Cref{prop:McP_ergodic} does not hold. As discussed in \Cref{ex:McP}, when $g \geq 2$, the Torelli group $\mathcal{I}_g$ satisfies $\Lambda(\mathcal{I}_g)=\pml_{g}$  but does not act ergodically.

Let $H$ be a sufficiently large subgroup of $\mcg(\Sigma_{g,m})$. Suppose that $\Lambda(H) = \pml_{g,m}$, but the action of $H$ on $\pml_{g,m}$ is not ergodic. Then, by the Ergodic Decomposition Theorem~\cite{MR1784210}, there exist a standard Borel space $(Y, \mathcal{B})$, a probability measure $\nu$ on $(Y, \mathcal{B})$, and a family $\{p_y \mid y \in Y\}$ of probability measures on $(\pml_{g,m}, \PThursM^{x_0})$ satisfying the following properties:

\begin{enumerate}
\item
For every Borel set $B \subset \pml_{g,m}$, the map $Y \ni y \mapsto p_y(B)$ is a Borel function, and
\[
\PThursM^{x_0}(B) = \int_Y p_y(B) \, d\nu(y);
\]
\item
For every $y \in Y$, the measure $p_y$ is quasi-invariant and ergodic under the action of $H$;
\item
If $y, y' \in Y$ and $y \ne y'$, then $p_y$ and $p_{y'}$ are mutually singular.
\end{enumerate}

\begin{problem}
\label{problem:erg}
Study the ergodic decomposition for the action of the Torelli group on $\pml_g$.
\end{problem}

For instance, in the discussion given in \cite{MR4830064}, the author considers the radial limit function of the period map from $\teich_{g}$ to the Siegel upper half-plane.  
Since the radial limit function is a bounded $\mathcal{I}_g$-invariant measurable function on $\pml_g$, the decomposition of $\pml_g$ into the level sets of this function yields a measurable partition that reflects the dynamical properties of the Torelli group.

\subsection{Horospherical limit points}
\label{subsec:Horospherical_limit_points}

In his celebrated paper \cite{MR624833}, Sullivan developed a theory concerning the ergodic action of discrete subgroups of the isometry group at infinity in hyperbolic space.  
The purpose of this section is to explore analogues of Sullivan's results for subgroups of the mapping class group.

We begin by introducing notions from measurable dynamics. For reference, see \cite[Chapter 2, \S6.1]{MR1041575}.

Let $H$ be a subgroup of the mapping class group of $\Sigma_{g,m}$.  
A set $W \subset \pml_{g,m}$ is called a \emph{wandering set}\index{wandering set} if $\PThursM^{x_0}(W \cap \omega(W)) = 0$ for every $\omega \in H$.  
If the action of $H$ admits a wandering set of positive measure, then $H$ is said to be \emph{dissipative}\index{dissipative}\index{subgroups of mapping class group!dissipative}; otherwise, it is called \emph{conservative}\index{conservative}\index{subgroups of mapping class group!conservative}.  
Moreover, $H$ is said to be \emph{totally dissipative}\index{totally dissipative}\index{subgroups of mapping class group!totally dissipative} if there exists a wandering set $W$ such that
\[
W^* = \bigcup_{\omega \in H} \omega(W)
\]
satisfies $\PThursM^{x_0}(\pml_{g,m} \setminus W^*) = 0$.

We now define small and big horospherical limit points, following the approach of Tukia \cite{MR1469798}:

\begin{definition}[Small and Big Horospherical Limit Sets]
Fix a base point $x_0 \in \teich_{g,m}$, and let $H$ be a subgroup of the mapping class group.
\begin{itemize}
\item
A point $[\lambda] \in \pml_{g,m}$ is called a big horospherical limit point\index{big horospherical limit point}\index{subgroups of mapping class group!big horospherical limit point} of $H$ if there exists a sequence $\{\omega_n\} \subset H$ such that
\begin{equation}
\label{eq:horo_definition1}
\frac{\ext_{\omega_n(x_0)}(\lambda)}{\ext_{x_0}(\lambda)}
\end{equation}
remains bounded above, and $d_T(x_0, \omega_n(x_0)) \to \infty$ as $n \to \infty$.
\item
A point $[\lambda] \in \pml_{g,m}$ is called a small horospherical limit point\index{small horospherical limit point}\index{subgroups of mapping class group!small horospherical limit point} of $H$ if there exists a sequence $\{\omega_n\} \subset H$ such that
\begin{equation}
\label{eq:horo_definition2}
\frac{\ext_{\omega_n(x_0)}(\lambda)}{\ext_{x_0}(\lambda)} \to 0
\end{equation}
as $n \to \infty$.
\end{itemize}
\end{definition}

Following Matsuzaki \cite{MR2038133}, we denote by $\Lambda_h(H)$ and $\Lambda_H(H)$ the sets of small and big horospherical limit points of $H$, respectively, and refer to them as the \emph{small} and \emph{big horospherical limit sets}\index{horospherical limit sets}\index{subgroups of mapping class group!horospherical limit set}. Clearly, $\Lambda_h(H) \subset \Lambda_H(H)$ for any subgroup $H \subset \mcg(\Sigma_{g,m})$.

The conditions in \eqref{eq:horo_definition1} and \eqref{eq:horo_definition2} are equivalent to the statements that $\ext_{\omega_n(x_0)}(\lambda)$ remains bounded above and tends to zero, respectively, as $n \to \infty$.  
These quotients are introduced in order to define a function on $\pml_{g,m}$.  
Although one could alternatively define horospherical limit points using hyperbolic lengths, we adopt the extremal length formulation due to its connection with the Poisson kernel appearing in the expression in \Cref{thm:PIF_bdd_pluri_harmonic_T}.

By the quasiconformal invariance of extremal lengths \eqref{eq:K-qc-extremal_length}, we have
\begin{align*}
\frac{\ext_{\omega(x_2)}(\lambda)}{\ext_{x_2}(\lambda)}
&\le
e^{2 d_T(\omega(x_1), \omega(x_2)) + 2 d_T(x_1, x_2)} \frac{\ext_{\omega(x_1)}(\lambda)}{\ext_{x_1}(\lambda)} \\
&= e^{4 d_T(x_1, x_2)} \frac{\ext_{\omega(x_1)}(\lambda)}{\ext_{x_1}(\lambda)}
\end{align*}
for any $x_1, x_2 \in \teich_{g,m}$ and $\omega \in H$.  
Therefore, the definitions of small and big horospherical limit points are independent of the choice of the base point $x_0 \in \teich_{g,m}$.

The horospherical limit sets $\Lambda_h(H)$ and $\Lambda_H(H)$ are invariant under the action of $H$.  
We verify this only for $\Lambda_h(H)$; the argument for $\Lambda_H(H)$ is analogous.  
Let $[\lambda] \in \Lambda_h(H)$ and $h \in H$.  
By definition, there exists a sequence $\{\omega_n\} \subset H$ such that
\[
\frac{\ext_{\omega_n(x_0)}(\lambda)}{\ext_{x_0}(\lambda)} \to 0 \quad \text{as } n \to \infty.
\]
Define $\omega'_n = h \omega_n h^{-1} \in H$ for each $n \in \mathbb{N}$. Then,
\[
\frac{\ext_{\omega'_n(h(x_0))}(h(\lambda))}{\ext_{h(x_0)}(h(\lambda))}
= \frac{\ext_{h \omega_n(x_0)}(h(\lambda))}{\ext_{h(x_0)}(h(\lambda))}
= \frac{\ext_{\omega_n(x_0)}(\lambda)}{\ext_{x_0}(\lambda)}
\to 0 \quad \text{as } n \to \infty.
\]
Hence $h([\lambda]) = [h(\lambda)] \in \Lambda_h(H)$, as claimed.

The following is an analogue of a part of Sullivan's dichotomy (cf.~\cite[Theorems III, IV]{MR624833}).

\begin{proposition}[Conservativity on $\Lambda_H$]
\label{prop:Wandering_Dissipative}
Suppose that $H$ is a subgroup of the mapping class group of $\Sigma_{g,m}$.  
Let $W$ be a wandering set for the action of $H$. Then,
$$
\PThursM^{x_0}\!\left(\Lambda_H(H) \cap W\right) = 0.
$$
In particular, the action of $H$ on the big horospherical limit set $\Lambda_H(H)$ is conservative in the sense that $\Lambda_H(H)$ contains no wandering set of positive measure.
\end{proposition}

\begin{proof}
The conclusion is immediate if $W$ is null. Suppose instead that $\PThursM^{x_0}(W) > 0$.  
Then, from \eqref{eq:invariance_ThP} and \eqref{eq:invariance_ThP2}, we have
\begin{align*}
1
&\ge \PThursM^{x_0}\left(\bigcup_{\omega \in H} \omega(W)\right)
= \sum_{\omega \in H} \PThursM^{x_0}\left(\omega(W)\right)
= \sum_{\omega \in H} \PThursM^{\omega^{-1}(x_0)}(W) \\
&= \int_W \left( \sum_{\omega \in H}
\left( \frac{\ext_{x_0}(\lambda)}{\ext_{\omega^{-1}(x_0)}(\lambda)} \right)^{3g - 3 + m}
\right) d\PThursM^{x_0}([\lambda]).
\end{align*}
Therefore, the sum
\begin{equation}
\label{eq:converge_sum}
\sum_{\omega \in H}
\left( \frac{\ext_{x_0}(\lambda)}{\ext_{\omega^{-1}(x_0)}(\lambda)} \right)^{3g - 3 + m}
\end{equation}
converges for $\PThursM^{x_0}$-almost every $[\lambda] \in W$.

Now, suppose $[\lambda] \in \Lambda_H(H) \cap W$.  
By the definition of the big horospherical limit set, there exist a constant $M > 0$ and a sequence $\{\omega_n\}_n \subset H$ such that
\[
\ext_{\omega_n^{-1}(x_0)}(\lambda) \le M \quad \text{for all } n \in \mathbb{N},
\]
and $d_T(x_0, \omega_n^{-1}(x_0)) \to \infty$ as $n \to \infty$. Then,
\begin{align*}
\frac{\ext_{x_0}(\lambda)}{\ext_{\omega_n^{-1}(x_0)}(\lambda)}
\ge \frac{\ext_{x_0}(\lambda)}{M}
\end{align*}
which is bounded below independently of $n$. This contradicts the convergence of \eqref{eq:converge_sum}, and hence the intersection $\Lambda_H(H) \cap W$ must be $\PThursM^{x_0}$-null.
\end{proof}

\begin{corollary}
\label{coro:totally_dissipative}
If $H$ is totally dissipative, then $\Lambda_H(H)$ is $\PThursM^{x_0}$-null.
\end{corollary}

Since $\Delta(H)$ is open in $\pml_{g,m}$ and $H$ acts properly discontinuously on $\Delta(H)$, any point of $\Delta(H)$ is contained in a wandering set.

\begin{proposition}
\label{prop:Horo_PropDisconti}
For any sufficiently large subgroup $H$ of $\mcg(\Sigma_{g,m})$, we have $\PThursM^{x_0}(\Lambda_H(H) \cap \Delta(H)) = 0$.  
Moreover, if $[\lambda] \in \Lambda_H(H)$ is uniquely ergodic, then $[\lambda] \in \Lambda(H)$.
\end{proposition}

\begin{proof}
The assertion $\PThursM^{x_0}(\Lambda_H(H) \cap \Delta(H)) = 0$ follows from \Cref{prop:Wandering_Dissipative}.

Let $[\lambda] \in \Lambda_H(H)$, and suppose that $\lambda$ is uniquely ergodic.  
By definition, there exists a sequence $\{\omega_n\}_n \subset H$ such that
$\ext_{\omega_n(x_0)}(\lambda)$ remains bounded above for all $n \in \mathbb{N}$ and $d_T(x_0, \omega_n(x_0)) \to \infty$ as $n \to \infty$.  
Then,
\[
e^{-2 d_T(x_0, \omega_n(x_0))} \ext_{\omega_n(x_0)}(\lambda)
\to 0
\]
as $n \to \infty$.  
This implies that the sequence $\{\omega_n(x_0)\}_n$ converges to $[\lambda] \in \pml_{g,m}$ in the Gardiner-Masur compactification (cf.~\cite[Theorem 3]{MR3009545}), and hence also converges to $[\lambda]$ in the Thurston compactification (cf.~\cite[Corollary 1]{MR3009545}).  
Therefore, $[\lambda]$ is a limit point in the sense of McCarthy and Papadopoulos.  
By \cite[Proposition 8.1]{MR982564}, it follows that $[\lambda] \in Z\Lambda(H)$.  
By the definition of $Z\Lambda(H)$, there exists $[\lambda'] \in \Lambda(H)$ such that $i(\lambda, \lambda') = 0$, which implies $[\lambda] = [\lambda'] \in \Lambda(H)$.
\end{proof}

From Sullivan's dichotomy~\cite[Theorem IV]{MR624833}, it is natural to ask the following (see also \cite[Theorem 1]{MR1469798}):

\begin{problem}
\label{problem:10}
Let $H$ be a sufficiently large subgroup of $\mcg(\Sigma_{g,m})$. Are the following statements equivalent?
\begin{enumerate}
\item
The fundamental domain has zero measure in $\pml_{g,m}$;
\item
The action of $H$ on $\pml_{g,m}$ is conservative;
\item
The big horospherical limit set $\Lambda_H(H)$ has full measure in $\pml_{g,m}$.
\end{enumerate}
\end{problem}

\subsection{Convex cocompact groups}
\label{subsec:convex_cocompact}
Farb--Mosher~\cite{MR1914566} and Kent--Leininger~\cite{MR2465691} studied convex cocompact subgroups of $\mcg(\Sigma_{g,m})$ from a geometric viewpoint.

Following Kent and Leininger, a point $[\lambda]\in \Lambda(H)$ is called a \emph{conical limit point}\index{conical limit}\index{subgroups of mapping class group!conical limit point} if for any Teichm\"uller geodesic ray with direction $[\lambda]$, there exists a number $R > 0$ such that some $H$-orbit intersects the $R$-neighborhood of the geodesic ray in an infinite set. Let $\Lambda_C(H)$ be the set of conical limit points of $H$. By definition, $\Lambda_0(H) \subset \Lambda_C(H)$. As noted by Kent and Leininger, by Masur's criterion (\Cref{thm:divergence_non-ue}), any conical limit point is uniquely ergodic. Hence, from McCarthy and Papadopoulos's criterion~\cite[Proposition 8.1]{MR982564}, we have $\Lambda_C(H) \subset \Lambda(H)$.

Therefore, when $H$ is sufficiently large, we obtain
\[
\overline{\Lambda_0(H)} = \overline{\Lambda_C(H)} = \Lambda(H),
\]
since $\Lambda(H)$ is the unique $H$-invariant minimal closed subset of $\pml_{g,m}$ (cf.~\cite[Theorem 4.1]{MR982564}).

\begin{proposition}
\label{prop:conical_horospherical}
For a subgroup $H$ of $\mcg(\Sigma_{g,m})$, we have $\Lambda_C(H) \subset \Lambda_h(H)$.
\end{proposition}

\begin{proof}
Fix $x_0 \in \teich_{g,m}$.  
Let $[\lambda] \in \pml_{g,m}$ be a conical limit point of $H$.  
Then $\lambda$ is uniquely ergodic, and there exists an infinite sequence $\{\omega_n\}_n$ in $H$ and a constant $R > 0$ such that $\omega_n(x_0)$ lies within the $R$-neighborhood of the Teichm\"uller geodesic ray emanating from $x_0$ in the direction $[\lambda]$.  
For each $n \in \mathbb{N}$, choose $t_n > 0$ such that
\[
d_T(\omega_n(x_0), \TRay{x_0;[\lambda]}(t_n)) \le R.
\]
Since $H$ acts properly discontinuously on $\teich_{g,m}$, it follows that $t_n \to \infty$ as $n \to \infty$.  
Therefore, we have
\[
\frac{\ext_{\omega_n(x_0)}(\lambda)}{\ext_{x_0}(\lambda)}
\le e^{2R} \cdot \frac{\ext_{\TRay{x_0;[\lambda]}(t_n)}(\lambda)}{\ext_{x_0}(\lambda)}
= e^{2R} e^{-2t_n} \to 0
\]
as $n \to \infty$.
This implies that $[\lambda] \in \Lambda_h(H)$.
\end{proof}

Let $H$ be a subgroup of $\mcg(\Sigma_{g,m})$.
Following Farb and Mosher~\cite{MR1914566}, we call a subgroup $H$ of $\mcg(\Sigma_{g,m})$ \emph{convex cocompact}\index{convex cocompact}\index{subgroups of mapping class group!convex cocompact} if some orbit of $H$ in $\teich_{g,m}$ is quasi-convex.  
Kent and Leininger showed that $\Lambda(H) = \Lambda_C(H)$ for a convex cocompact group $H$ (cf.~\cite[Theorem 1.2]{MR2465691}). This means that every $[\lambda] \in \Lambda(H)$ is uniquely ergodic and $Z\Lambda(H) = \Lambda(H) = \Lambda_C(H)$ for a convex cocompact group $H$.

\begin{proposition}
\label{prop:convex_cocompact}
For a convex cocompact and sufficiently large subgroup $H$ of $\mcg(\Sigma_{g,m})$, we have
\[
Z\Lambda(H) = \Lambda(H) = \Lambda_C(H) = \Lambda_h(H) = \Lambda_H(H).
\]
\end{proposition}

\begin{proof}
This proposition follows from an argument similar to that of \Cref{prop:Horo_PropDisconti}.  
For completeness, we provide the proof.

Let $[\lambda] \in \Lambda_H(H)$. By definition, there exist $M > 0$ and a sequence $\{\omega_n\}_n \subset H$ such that $\ext_{\omega_n(x_0)}(\lambda) \le M$ for all $n \in \mathbb{N}$ and $d_T(x_0, \omega_n(x_0)) \to \infty$ as $n \to \infty$.  
By passing to a subsequence, we may assume that $\{\omega_n(x_0)\}_n$ converges to some $[\lambda'] \in \pml_{g,m}$ in the Thurston compactification.  
Since $H$ is convex cocompact, it follows from~\cite[Proposition 8.1]{MR982564} that $[\lambda'] \in Z\Lambda(H) = \Lambda(H) = \Lambda_C(H)$. Hence, $\lambda'$ is uniquely ergodic.

From~\cite[Corollary 1]{MR3009545}, the sequence $\{\omega_n(x_0)\}_n$ also converges to $[\lambda']$ in the Gardiner--Masur compactification.  
Then, by~\cite[Corollary 2]{MR3009545}, we have
\begin{align*}
\frac{i(\lambda,\lambda')^2}{\ext_{x_0}(\lambda')}
&= \lim_{n \to \infty} e^{-2d_T(\omega_n(x_0), x_0)} \ext_{\omega_n(x_0)}(\lambda)
\\
&\le \lim_{n \to \infty} M e^{-2d_T(\omega_n(x_0), x_0)} = 0.
\end{align*}
Therefore, $[\lambda] = [\lambda'] \in \Lambda(H)$.
\end{proof}

In the case of Kleinian groups $G$, the condition $\Lambda(G) = \Lambda_h(G)$ implies that $G$ is convex cocompact (cf.~\cite[Proposition 8]{MR2038133}).

\begin{problem}
\label{problem:7}
Is a subgroup $H$ of $\mcg(\Sigma_{g,m})$ convex cocompact if $\Lambda(H) = \Lambda_h(H)$?
\end{problem}

From the proof of \Cref{prop:Horo_PropDisconti}, any uniquely ergodic point $[\lambda] \in \Lambda_H(H)$ lies in $\Lambda(H)$. Hence, the following can be regarded as a weak version of \Cref{problem:7}:

\begin{problem}
\label{problem:12}
Is a subgroup $H$ of $\mcg(\Sigma_{g,m})$ convex cocompact if every point in $\Lambda_H(H)$ is uniquely ergodic?
\end{problem}

\begin{problem}
\label{problem:15}
Is a convex cocompact group totally dissipative?
\end{problem}

If \Cref{problem:15} is resolved affirmatively, then it follows from \Cref{coro:totally_dissipative} and \Cref{prop:convex_cocompact} that $\Lambda(H)$ has measure zero for a convex cocompact group~$H$.

\subsection{Future challenge: Functions, Actions and Flows}
\label{subsec:erg_flow}

Kaimanovich showed the following (cf.~\cite[Theorem 6.4]{MR1738739}).

\begin{theorem}[Kaimanovich]
\label{thm:kaimanovich}
For a Fuchsian group $G$ acting on $\mathbb{H}$, the following properties are equivalent:
\begin{enumerate}
\item
The horocyclic flow on the quotient space $\mathbb{H}/G$ is ergodic with respect to the Liouville invariant measure;
\item
The linear action of $G$ on $\mathbb{R}^2$ is ergodic;
\item
The action of $G$ on the boundary $\partial \mathbb{H}$ is ergodic with respect to the smooth measure class;
\item
There are no non-constant bounded harmonic functions on $\mathbb{H}/G$.
\end{enumerate}
\end{theorem}

The horocyclic flow on the moduli space was introduced by Masur~\cite{MR632453}.  
Masur also showed that the horocyclic flow is ergodic~\cite{MR787893}.  
Motivated by Kaimanovich's observation, it is natural to ask the following:

\begin{problem}
\label{problem:9}
Let $H$ be a sufficiently large subgroup of $\mcg(\Sigma_{g,m})$. Are the following properties equivalent?
\begin{enumerate}
\item
The horocyclic flow on the quotient space $\teich_{g,m}/H$ is ergodic with respect to the Masur-Veech invariant measure;
\item
The action of $H$ on $\ml_{g,m}$ is ergodic with respect to the Thurston measure;
\item
The action of $H$ on $\pml_{g,m}$ is ergodic with respect to the Thurston measure;
\item
$\teich_{g,m}/H$ admits no non-constant bounded pluriharmonic functions.
\end{enumerate}
\end{problem}

\Cref{thm:main4} shows that condition (3) implies condition (4).  
Condition (2) implies condition (3) because any non-null and non-full invariant set $A \subset \pml_{g,m}$ can be extended to a non-null and non-full invariant subset of $\ml_{g,m}$ by taking the rays in directions given by $A$.

\begin{remark}
Although not explicitly stated in \Cref{remark:2}, the condition of "pluriharmonicity" in (4) of \Cref{problem:9} is substantially stronger than mere "harmonicity", as it is defined by a system of linear partial differential equations. It is conceivable that this condition could be replaced by a more suitable one.
\end{remark}

Although not directly related, the Hopf--Tsuji--Sullivan theory \cite{MR624833} motivates the following question (cf., e.g., \cite{MR1738739}, \cite{MR1041575}, and \cite{MR1638795}):

\begin{problem}
\label{problem:11}
Let $H$ be a sufficiently large subgroup of $\mcg(\Sigma_{g,m})$. Are the following properties equivalent?
\begin{enumerate}
\item
The Teichm\"uller geodesic flow on $\teich_{g,m}/H$ is ergodic;
\item
The diagonal action of $H$ on $\pml_{g,m} \times \pml_{g,m}$ is ergodic;
\item
The quotient $\teich_{g,m}/H$ admits no pluricomplex Green function.
\end{enumerate}
\end{problem}

In the case of Fuchsian groups, there exists a group $G$ such that $\mathbb{H}/G$ admits a Green function but no non-constant bounded harmonic functions (cf.~\cite[Chapter IV, \S25]{MR0114911}). This means that $G$ acts ergodically on $\partial \mathbb{H}$ but not on the product $\partial \mathbb{H} \times \partial \mathbb{H}$. Motivated by this phenomenon, it is natural to ask the following:

\begin{problem}
\label{problem:13}
Does there exist a subgroup $H$ of $\mcg(\Sigma_{g,m})$ that acts ergodically on $\pml_{g,m}$ but not on the product $\pml_{g,m} \times \pml_{g,m}$?
\end{problem}

\subsection{Future challenge: Monodromies of holomorphic families}
\label{subsec:Monodromy}

A \emph{holomorphic family of Riemann surfaces}\index{holomorphic family of Riemann surfaces} of type $(g,m)$ over a Riemann surface $B$ is a triple $(\mathcal{M}, \pi, B)$ such that:
\begin{enumerate}
\item there are a two dimensional complex manifold $\hat{\mathcal{M}}$, a non-singular one dimensional analytic subset of $\hat{\mathcal{M}}$ (possibly $C=\emptyset$) and a holomorphic map $\hat{\pi}\colon \hat{\mathcal{M}}\to B$ such that $\hat{\pi}$ is proper and of maximal rank at every point of $\hat{\mathcal{M}}$;
\item
$\mathcal{M}=\hat{\mathcal{M}}\setminus C$, and $\pi=\hat{\pi}|_{\mathcal{M}}$, and the fiber $\pi^{-1}(b)$ over each point $b \in B$ is a Riemann surface of type $(g,m)$.  
\end{enumerate}

Assume that $B$ is a hyperbolic Riemann surface and that it is represented as $B = \mathbb{D}/\Gamma$, where $\Gamma$ is the Fuchsian group associated with $B$. Let ${\rm pr} \colon \mathbb{D} \to B$ denote the projection.  
It is known that there exists a holomorphic map $\Phi \colon \mathbb{D} \to \teich_g$, called the \emph{representation}\index{representation}\index{holomorphic family of closed Riemann surfaces!representation} of the holomorphic family $(\mathcal{M}, \pi, B)$, and a homomorphism $\rho \colon \Gamma \to {\rm Mod}_{g,m}$, called the \emph{monodromy homomorphism}\index{monodromy homomorphism}\index{holomorphic family of Riemann surfaces!monodromy homomorphism}, such that $\rho(\gamma) \circ \Phi = \Phi \circ \gamma$ on $\mathbb{D}$ for all $\gamma \in \Gamma$, and for each $z \in \mathbb{D}$, the underlying surface of $\Phi(z) \in \teich_g$ is biholomorphic to the fiber $\pi^{-1}({\rm pr}(z))$, where ${\rm Mod}_{g,m}$ denotes the Teichm\"uller modular group (cf. \S\ref{subsec:Action_MCG}).  

Let $H = \rho(\Gamma)$.  
When $B$ belongs to the class $\mathcal{O}_{HB}$—meaning that $B$ admits no non-constant bounded harmonic functions—some two points in the quotient space $\teich_{g,m}/H$ cannot be distinguished by any bounded pluriharmonic function.  
Indeed, let $u$ be a bounded pluriharmonic function on $\teich_{g,m}/\rho(\Gamma)$. Let $\tilde{u} \colon \teich_{g,m} \to \mathbb{R}$ be the lift of $u$. Since $\tilde{u}$ is $\rho(\Gamma)$-invariant, the composition $\tilde{u} \circ \Phi$ is a bounded $\Gamma$-invariant harmonic function on $\mathbb{D}$, which descends to a bounded harmonic function on $B$. 
Hence, $\tilde{u}$ is constant on the image of $\Phi$. 

By a similar argument, when $B$ is of class $\mathcal{O}_{AB}$—that is, $B$ admits no non-constant bounded holomorphic functions—then some two points in $\teich_{g,m}/H$ cannot be distinguished by bounded holomorphic functions.

Note that any two points in $\teich_{g,m}$ can be separated by bounded holomorphic functions, since $\teich_{g,m}$ is biholomorphic to a bounded domain.  
Thus, the function-theoretic structure of the base Riemann surface $B$ influences both the structure of the subgroup $H$ and that of the quotient space $\teich_{g,m}/H$.  
This separability by bounded holomorphic functions arises naturally and plays a central role in the study of the Banach algebra of bounded holomorphic functions; see \cite{MR268655}, for instance.


\begin{problem}
\label{problem:4-1}
Does $\teich_{g,m}/H$ admit no non-constant bounded pluriharmonic functions when the base Riemann surface $B$ belongs to the class $\mathcal{O}_{HB}$?
\end{problem}

\begin{problem}
\label{problem:4-2}
Does $\teich_{g,m}/H$ admit no non-constant bounded holomorphic functions when the base Riemann surface $B$ belongs to the class $\mathcal{O}_{AB}$?
\end{problem}

In light of \Cref{thm:main4}, it is natural to ask the following:

\begin{problem}
\label{problem:4}
Is the action of the monodromy group $\rho(\Gamma)$ ergodic when the base Riemann surface $B$ belongs to the class $\mathcal{O}_{HB}$?
\end{problem}

A weaker version of this question is the following:

\begin{problem}
\label{problem:5}
Is the action of $\rho(\Gamma)$ on $\pml_{g,m}$ minimal when the base Riemann surface $B$ belongs to the class $\mathcal{O}_{HB}$?
\end{problem}

In fact, the above four problems—\Cref{problem:4-1}, \Cref{problem:4-2}, \Cref{problem:4}, and \Cref{problem:5}—stem from the seminal work of Shiga~\cite{MR1466656}. Specifically, Shiga showed that when $B$ is of class $\mathcal{O}_G$, the image $\rho(\Gamma)$ is infinite and irreducible. Hence, these problems remain meaningful even when $B$ belongs to the class $\mathcal{O}_G$.
Observe from \Cref{thm:main4} that an affirmative answer to \Cref{problem:4} would also yield affirmative answers to the other three problems.

Affirmative answers to \Cref{problem:4-1} and \Cref{problem:4-2} can be summarized, in slogan form, as follows:  
if $B$ is of class $\mathcal{O}_{HB}$ or $\mathcal{O}_{AB}$, then so is $\teich_{g,m}/H$.  

\begin{problem}
\label{problem:6}
Which properties of $B$ as a Riemann surface are inherited by $\teich_{g,m}/H$?  
For instance, if $B$ is of class $\mathcal{O}_G$, does $\teich_{g,m}/H$ also admit no pluricomplex Green function?
\end{problem}


\end{document}